\documentclass[a4paper]{article}

\usepackage{comment}

\usepackage[all]{xy}\usepackage[latin1]{inputenc}        %accents
\usepackage[dvips]{graphics,graphicx}
\usepackage{amsfonts,amssymb,amsmath,xcolor,mathrsfs, dsfont,amstext}
\usepackage{amsbsy, amsopn, amscd, amsxtra, amsthm,authblk, enumerate}
\usepackage{upref}
\usepackage{geometry}
\usepackage[displaymath]{lineno}
\usepackage{float}
\usepackage{bbm, bm}

\usepackage[colorlinks,
            linkcolor=blue,
            anchorcolor=green,
            citecolor=blue
            ]{hyperref}

\numberwithin{equation}{section}

\def\e{\varepsilon}
\def\epsilon{\varepsilon}
\def\eps{\varepsilon}

\newcommand{\wt}{\widetilde}
\def\alb#1\ale{\begin{align*}#1\end{align*}}
\newcommand{\eqb}{\begin{equation}}
\newcommand{\eqe}{\end{equation}}

\newcommand{\bbC}{\mathbb{C}}

\newcommand{\bbE}{\mathbb{E}}
\newcommand{\bbH}{\mathbb{H}}

\newcommand{\cC}{\mathcal{C}}

\newcommand{\cD}{\mathcal{D}}

\newcommand{\QD}{\mathrm{QD}}
\newcommand{\QT}{\mathrm{QT}}
\newcommand{\LF}{\mathrm{LF}}
\newcommand{\SLE}{\mathrm{SLE}}
\newcommand{\CLE}{\mathrm{CLE}}
\newcommand{\Wd}{\mathrm{Weld}}
\newcommand{\Md}{{\mathcal{M}}^\mathrm{disk}}

\newtheorem{theorem}{Theorem}[section]
\newtheorem{lemma}[theorem]{Lemma}
\newtheorem{proposition}[theorem]{Proposition}
\newtheorem*{proposition*}{Proposition}

\newtheorem*{corollary*}{Corollary}
\newtheorem{definition}[theorem]{Definition}
\newtheorem*{definitions*}{Definitions}

\newtheorem*{example*}{\bf Example}
\theoremstyle{remark}
\newtheorem{remark}[theorem]{Remark}

\numberwithin{equation}{section}

\renewcommand{\H}{\mathbb{H}}
\newcommand{\R}{\mathbb{R}}

\newcommand{\BD}{\mathrm{BD}}
\newcommand{\C}{\mathbb{C}}

\newcommand{\Loop}{\mathcal{L}^o}

\newcommand{\ep}[1]{\overline{\overline{#1}}}
\newcommand{\BS}{\mathrm{BS}}

\newcommand{\QA}{\mathrm{QA}}
\newcommand{\fd}{f}
\newcommand{\Mfd}{{\mathcal{M}}^\fd}
\newcommand{\Mfr}{{\mathcal{M}}^{\rm f.r.}}
\newcommand{\Mfl}{{\mathcal{M}}^{\rm f.l.}}
\newcommand{\cf}{\mathfrak{f}}

\begin{document}

\title{Annulus crossing formulae for critical planar percolation}
\date{}

\author{Xin Sun\thanks{Beijing International Center for Mathematical Research, Peking University.}  \quad \qquad Shengjing Xu\thanks{University of Pennsylvania.} \quad \qquad Zijie Zhuang$^\dagger$}

\maketitle

\begin{abstract}
We derive exact formulae for three basic annulus crossing events for the critical planar Bernoulli percolation in the continuum limit. The first is for the probability that there is an open path connecting the two boundaries of an annulus of inner radius $r$ and outer radius $R$. The second is for the probability that there are both open and closed paths connecting the two annulus boundaries. These two results were predicted by Cardy  based on non-rigorous Coulomb gas arguments. Our third result gives the probability that there are two disjoint open paths connecting the two boundaries. Its leading asymptotic as $r/R\to 0$ is captured by the so-called backbone exponent, a transcendental number recently determined by Nolin, Qian and two of the authors. This exponent is the unique real root to the equation $\frac{\sqrt{36 x +3}}{4} + \sin (\frac{2 \pi \sqrt{12 x +1}}{3} ) =0$, other than $-\frac{1}{12}$ and $\frac14$. Besides these three real roots, this equation has countably many complex roots.  
Our third result shows that these roots  appear exactly as exponents of the subleading terms in the crossing formula. This suggests that the backbone exponent is part of a conformal field theory (CFT) whose bulk spectrum contains this set of roots. Expanding the same crossing probability as $r/R\to 1$, we obtain a series with logarithmic corrections at every order, suggesting that the backbone exponent is related to a logarithmic boundary CFT. Our proofs are based on the coupling between SLE curves and Liouville quantum gravity (LQG). The key is to encode the annulus crossing probabilities by the random moduli of certain LQG surfaces with annular topology, whose law can be extracted from the dependence of the LQG annuli partition function on their boundary lengths. 
\end{abstract}

\setcounter{tocdepth} {1}
\tableofcontents
 
\section{Introduction} 
\label{sec:intro}

Bernoulli percolation is a fundamental lattice model in statistical physics, serving as an ideal playground for studying phase transitions and critical phenomena. The continuum limit of critical planar Bernoulli percolation is particularly interesting due to its conformal invariance.  A number of fundamental formulae for this model were derived through non-rigorous conformal field theory approaches, such as~\cite{cardy-formula, Watts1996, SimmonsKlebanZiff2007}. One of the most famous examples is Cardy's formula for the crossing probability of a rectangle. This formula was established for critical site percolation on the triangular lattice by Smirnov~\cite{Sm01}. 
This paves the way for describing  the scaling limit of planar percolation by  Schramm-Loewner evolution (SLE)~\cite{Sc00} curves.
In particular, it was shown in~\cite{CN06} that the full scaling limit is characterized by a collection of SLE$_6$-type loops, 
known as the conformal loop ensemble with parameter $6$ (CLE$_6$).

In~\cite{Car06}, Cardy predicted the exact formula for the probability that there is an open path crossing an annulus in the scaling limit of critical percolation. In this paper, we rigorously prove this formula. Furthermore, we prove the formula for the probability that there are both an open path and a closed path crossing the annulus, again predicted by Cardy in an earlier work~\cite{Car02}. Both formulae are expressed as ratios of Dedekind eta functions with different modular parameters. In addition, we derive the exact formula for the probability that there are two disjoint open paths crossing the annulus, namely the monochromatic two-arm crossing probability. This formula  is a series expansion in terms of the modular parameter of the annulus, which was not known in physics. In fact, the leading asymptotic of this crossing probability is given by the so-called backbone exponent, recently derived by Nolin, Qian, and two of us~\cite{NQSZ23}. Its value is a transcendental number, which is a root of an elementary equation. We show that other roots of this equation capture the asymptotics of the remaining terms in this expansion. When expanding the series in terms of the dual modular parameter, we find logarithmic behavior; see Theorem~\ref{thm:backbone-crossing}.
A physicist-oriented exposition of~\cite{NQSZ23} and our Theorem~\ref{thm:backbone-crossing} can be found in~\cite{NQSZ-physical-2024}.

Cardy's formula for the rectangular crossing can be easily derived via stochastic calculus, assuming SLE$_6$ convergence. However, this does not appear to be the case for the annulus crossing probabilities that we consider. Our derivation is based on Liouville quantum gravity (LQG) on the annulus.
The key is to understand the LQG surface that describes the scaling limit of a random triangulation of the annulus decorated by percolation configurations where the desired crossing events occur. According to the philosophy by Ang, Remy, and the first-named author in~\cite{ARS22}, the law of the modulus of this random surface should capture the corresponding crossing probability. We analyze these LQG annuli using the coupling between CLE$_6$ and LQG. The monochromatic two-arm case is technically the most difficult. See Section~\ref{subsec:proof-strategy} for an overview of the proof.

\subsection{Main results}\label{subsec:main-result}
In this paper, we focus on critical Bernoulli site percolation on the triangular lattice $\delta \mathbb{T}$ with mesh size $\delta>0$. That is, we color each vertex on $\delta \mathbb{T}$ independently as black (=open) or white (=closed) with equal probability. A black (resp.\ white) path is a sequence of neighboring black (resp.\ white) vertices on $\delta \mathbb{T}$. For $0<r<R$, let the annulus $A(r, R) := \{ z \in \mathbb{C}: r < |z|_2 < R \} $ and let $A(r,R)_\delta$ be its discretization, i.e., the largest connected set of $\delta \mathbb{T}$ that is contained in $A(r,R)$. The one-arm event is defined as follows:
\begin{equation*}
\begin{aligned}
    \mathcal{A}_B(r, R, \delta) &:= \{ \mbox{There exists a black path}\mbox{ crossing }A(r,R)_\delta \}.
\end{aligned}
\end{equation*}
Similarly, we define the polychromatic (resp.\ monochromatic) two-arm event $\mathcal{A}_{BW}(r, R, \delta)$ (resp.\ $\mathcal{A}_{BB}(r, R, \delta)$) as the event that there exist both a black and a white path (resp.\ two disjoint black paths) crossing $A(r,R)_\delta$. By the seminal work of Smirnov~\cite{Sm01}, combined with~\cite{CN06, gps-pivotal}, we know that for fixed $0<r<R$,
\begin{equation}\label{eq:def-p(r,R)}
    p_\mathsf{a}(r,R) := \lim_{\delta \rightarrow 0} \mathbb{P}[\mathcal{A}_\mathsf{a}(r, R, \delta)] \mbox{ exists for all } \mathsf{a} \in \{B, BW, BB\}.
\end{equation}
See Lemma~\ref{lem:percolation-limit} for the proof of the existence of the limit. In this paper, we will derive the exact values of these limits. 
Theorem~\ref{thm:cardy-formula} provides the exact formulae for $p_B(r,R)$ and $ p_{BW}(r,R)$, as predicted by Cardy in~\cite{Car02, Car06}.
\begin{theorem}\label{thm:cardy-formula}
For $0<r<R$, let $\tau = \frac{1}{2 \pi} \log(\frac{R}{r})$ be the modulus of $A(r,R)$. We have
\begin{align}
    p_B(r,R) &= \sqrt{\frac{3}{2}} \cdot \frac{\eta(6 i \tau ) \eta\left(\frac{3}{2} i \tau \right)}{\eta(2 i \tau ) \eta(3 i \tau )}\,; \label{eq:1arm}\\
    p_{BW}(r,R) &= \sqrt{3} \cdot \frac{ \eta(i \tau) \eta(6 i \tau)^2}{\eta(3 i \tau) \eta(2 i \tau)^2}\,,\label{eq:2arm}
\end{align}
where $\eta(z) = e^{\frac{i \pi z}{12}} \prod_{n=1}^\infty (1-e^{2 n i \pi z})$ is the \emph{Dedekind eta function}.
\end{theorem}

Cardy formula for the rectangular crossing can be derived from a standard SLE$_6$ calculation using It\^{o} calculus; see e.g.~\cite{werner-notes}. The derivation boils down to solving a hypergeometric ODE, which can be viewed as the BPZ equation for a four-point CFT correlation function on the disk. One can try to use the same strategy to prove Theorem~\ref{thm:cardy-formula}. However, in the annulus case, one would encounter a PDE involving both $\tau$ and the location of an auxiliary point on the annulus boundary. Indeed, in~\cite{Dub04} Dub\'edat obtained a description of $p_{BW}(r,R)$ via such a PDE, and it appears hard to extract the right-hand side of~\eqref{eq:2arm} from it. 
In light of this, we obtain a nontrivial PDE result by proving~\eqref{eq:2arm}. We will review our proof of Theorem~\ref{thm:cardy-formula} based on LQG in Section~\ref{subsec:proof-strategy}.  

As shown in~\cite{Car02, Car06}, $p_{B}(r,R)$ and $p_{BW}(r,R)$ can be expressed as follows:
\begin{align}
    p_B(r, R) =\sqrt{\frac{3}{2}} \cdot \frac{  \sum_{k \in \mathbb{Z}} ( \tilde q^{2 h_{4k - \frac{1}{2},0}} - \tilde q^{2 h_{4k+\frac{3}{2},0}} )}{\prod_{n=1}^\infty (1-\tilde q^{2n})} \textrm{ and }  p_{BW}(r, R) = \sqrt{3} \cdot \frac{\sum_{k \in \mathbb{Z}} ( \tilde q^{2 h_{0,6k+1}} - \tilde q^{2 h_{0,6k+2}})}{\prod_{n=1}^\infty (1-\tilde q^{2n})};\label{eq:expand-tilde-q}\\
    p_B(r, R) = \frac{\sum_{k \in \mathbb{Z}} ( q^{h_{1,4k + 1}} - q^{h_{1,4k+3}} )}{\prod_{n=1}^\infty (1-q^n)} \textrm{ and } p_{BW}(r, R) = \frac{\sum_{k \in \mathbb{Z}} ( q^{h_{1,6k+2}} + q^{h_{1,6k+4}}- 2 q^{h_{1,6k+3}} )}{\prod_{n=1}^\infty (1-q^n)},\label{eq:expand-q}
\end{align}
where $h_{r,s} = \frac{(3r - 2s)^2-1}{24}$. In the terminology of conformal field theory,
expansions of the form~\eqref{eq:expand-tilde-q} (resp.~\eqref{eq:expand-q}) are called the closed (resp., open) channel expansions.
The expression for $h_{r,s}$ comes from the Kac formula with the Virasoro central charge equal to $0$. 
These closed channel expansions correspond to summing over the spectrum of a bulk CFT, while the open channel corresponds to a boundary CFT. See the discussion below Equation (2) in \cite{Car06}, as well as the discussion below Corollary 1.10 in~\cite{ARS22}. Despite this appealing interpretation, so far the only rigorous way to obtain~\eqref{eq:expand-tilde-q} and~\eqref{eq:expand-q} is through Theorem~\ref{thm:cardy-formula}. 

Theorem~\ref{thm:backbone-crossing} provides the closed channel and open channel expansions for $p_{BB}(r,R)$.
\begin{theorem}\label{thm:backbone-crossing}
For $0<r<R$, let $\tau = \frac{1}{2 \pi} \log(\frac{R}{r})$ and $\tilde q = e^{-2 \pi \tau}$. 
Let $\mathcal{S} = \{s\in \mathbb{C} : \sin(4 \pi \sqrt{\frac{s}{3}}) + \frac{3}{2}\sqrt{s} = 0 \} \setminus \{0, \frac13 \},$
where $\sqrt{s}$ is defined using arguments in $(-\frac{\pi}{2}, \frac{\pi}{2}]$ for $s \in \mathbb{C}$. We have:
\begin{equation}
\label{eq:thm-backbone-2}
p_{BB}(r,R) = \frac{1}{\prod_{n=1}^\infty (1-\tilde q^{2n})} \sum_{s \in \mathcal{S}} \frac{- \sqrt{3} \sin(\frac{2\pi}{3}  \sqrt{3 s}) \sin(\pi \sqrt{3 s}) }{\cos(\frac{4\pi }{3}  \sqrt{3 s}) + \frac{3 \sqrt{3}}{8 \pi} } \tilde q^{s -\frac{1}{12} }.  
\end{equation}
Let $q = e^{-\pi/\tau}$. Using the convention that $(-1)!! = 1$, we have:
\begin{equation}\label{eq:thm-backbone-1}
   p_{BB}(r, R) = \frac{1 + \sum_{j=1}^\infty \big(g_j(\tau) q^{\frac{2j^2-j}{3}} + h_j(\tau) q^{\frac{2j^2+j}{3}} \big)}{\prod_{n=1}^\infty (1-q^n)} ,
\end{equation}
where 
\begin{equation*}
\begin{aligned}
&g_j(\tau) = \sum_{\substack{m,n,k \geq 0,\\ n + 2m = j \mbox{ }{\rm or} \mbox{ } j-1}} (-1)^{j+k} \binom{n+m}{m}\binom{n}{2k}(2k-1)!! \big(\frac{\sqrt{3}(4j-1)}{4 \tau} \big)^n \big(\frac{12 \tau}{\pi(4j-1)^2} \big)^k;\\
&h_j(\tau) = \sum_{\substack{m,n,k\geq 0,\\ n + 2m = j - 1 \mbox{ }{\rm or} \mbox{ } j-2}} (-1)^{j+k} \binom{n+m}{m}\binom{n}{2k}(2k-1)!! \big(\frac{\sqrt{3}(4j+1)}{4 \tau} \big)^n \big(\frac{12 \tau}{\pi(4j+1)^2} \big)^k.
\end{aligned}
\end{equation*}
\end{theorem}

As explained in the proof of Lemma~\ref{lem:absolute-convergent}, except for the backbone exponent, no other elements of $\mathcal{S}$ are real. 
The numerical values (to three decimal places) of $\mathcal{S}$ are:
$$
\mathcal{S} = \{ 0.440,\, 2.194 \pm 0.601i,\, 5.522 \pm 1.269 i,\, 10.361 \pm 2.020 i, \ldots \},
$$
where the roots are arranged by their real parts.
Note that if $z\in \mathcal{S}$, then $\bar{z} \in \mathcal{S}$ as well, ensuring that the right-hand side of \eqref{eq:thm-backbone-2} is real. 
The closed channel expansion~\eqref{eq:thm-backbone-2} suggests that the backbone exponent is part of a conformal field theory whose bulk spectrum contains $\mathcal S$. 

In the open channel expansion~\eqref{eq:thm-backbone-1}, both $g_j(\tau)$ and $h_j(\tau)$ are polynomials in $\tau^{-1}$ with degrees $j$ and $j-1$, respectively. The first two terms are given by $g_1(\tau) = -1 - \frac{3 \sqrt{3}}{4} \tau^{-1}$ and $g_2(\tau) = 1 + (\frac{7 \sqrt{3}}{4} - \frac{9}{4 \pi}) \tau^{-1} + \frac{147}{16} \tau^{-2}$; $h_1(\tau) = -1$ and $h_2(\tau) = 1 + \frac{9 \sqrt{3}}{4} \tau$.  
As we see logarithmic corrections to the power law at every term, the open channel expansion suggests that the backbone exponent is related to a logarithmic boundary CFT. Focusing on the first logarithmic term, we get:
\begin{equation}\label{eq:asymp-tau-0}
p_{BB}(r, R) = 1 - (1 + \frac{3 \sqrt{3}}{4} \tau^{-1} ) q^{\frac{1}{3}} + o(q^{\frac{1}{3}}) \quad \mbox{as $q \to 0$.}
\end{equation}

The appearance of $\tau^{-1}  q^{\frac{1}{3}}$ can be explained by a qualitative argument, as we now elaborate. By conformal invariance, it is equivalent to consider percolation on a cylinder of size $1 \times \tau $ with periodic boundary conditions in the 1-direction. The modulus of this cylinder is $\tau$. In this setting, $1 - p_B(r, R)$ represents the probability of having a white cluster wrapping around the cylinder. By Russo-Seymour-Welsh estimates, this probability is up-to-constants equivalent to the probability of having a hard crossing in a rectangle of size $1 \times \tau $, which is explicitly determined by Cardy's formula~\cite{cardy-formula, Sm01}. By simple calculations, we can show that $1 - p_B(r, R)$ is up-to-constants equivalent to $q^{\frac{1}{3}}$ as $\tau$ tends to 0. Similarly, $1 - p_{BW}(r, R)$ is also up-to-constants equivalent to $q^{\frac{1}{3}}$ as $\tau$ tends to 0. These two asymptotic expansions are consistent with~\eqref{eq:expand-q}. To see the appearance of $\tau^{-1}  q^{\frac{1}{3}}$ in~\eqref{eq:asymp-tau-0}, note that $p_B(r, R) - p_{BB}(r, R)$ is the probability of having one black crossing without having two disjoint black crossings. This event is equivalent to the existence of a white cluster wrapping around the cylinder, except for one black site on it that belongs to a black crossing. Due to the freedom of this single black crossing, we can deduce that $p_B(r, R) - p_{BB}(r, R)$ is up-to-constants equivalent to $\tau^{-1} q^{1/3}$.

The logarithmic nature of the CFT behind percolation has been extensively discussed in physics; see e.g.~\cite{Gurarie-log, Ridout-2013-log}.
Recently, Camia and Feng~\cite{CamiaFeng2024} rigorously established the logarithmic behavior of certain percolation observables using qualitative arguments. Compared to~\cite{CamiaFeng2024}, we determine the exact expansion at all orders, where the terms after $\tau^{-1}  q^{\frac{1}{3}}$ are difficult to deduce through qualitative arguments. On the other hand, a highlight of~\cite{CamiaFeng2024} is the identification of logarithmic behavior in both the bulk and the boundary. In the fourth point of Section~\ref{subsec:outlook}, we give an example of an exactly solvable observable for percolation on the annulus with logarithmic behavior at all orders in the closed channel expansion, which corresponds to the bulk. See also~\cite{Peltola-log-UST} for the logarithmic behavior in the setting of uniform spanning tree and SLE$_8$. 

\subsection{Proof strategy based on quantum gravity on the annulus}
\label{subsec:proof-strategy}
Our proof is based on quantum gravity on the annulus. The scaling limit of uniform type random planar maps is a natural discretization of random surfaces in pure 2D quantum gravity. The scaling limit as metric-measure spaces was first established for the sphere in~\cite{legall-uniqueness, miermont-brownian-map}, and later extended to other topologies~\cite{bet-mier-disk, bet-mier-surface}. The limiting object is called the \textit{Brownian surface}. For the Brownian sphere and the Brownian disk, Miller and Sheffield \cite{lqg-tbm1, lqg-tbm2, lqg-tbm3} showed that the geometry is governed by 
Liouville quantum gravity (LQG) with parameter $\gamma=\sqrt{8/3}$ under a conformal coordinate. The random geometry of LQG is governed by random fields that locally look like a Gaussian free field.
It was understood in~\cite{AHS17, cercle2021unit, AHS21} that variants of the Gaussian free field describing Brownian surfaces can be obtained from Liouville conformal field theory (CFT).  
This also applied to the Brownian annulus, except that the conformal modulus of the annulus becomes random. 
The exact law of the annulus was conjectured in bosonic string theory~\cite{Remy-annulus, Martinec-annulus} and proved in~\cite{ARS22}. The idea is that on the one hand, 
the exact formula for the partition function of Liouville CFT on a fixed annulus with given boundary lengths was derived by Wu~\cite{BaojunWu22}. 
On the other hand, the partition function for the Brownian annulus with the same boundary length dependence is known from the enumeration results for planar maps. 
Comparing these two, one can extract the law of the modulus for the Brownian annulus.

We use $p_B(r, R)$ to demonstrate our proof strategy for Theorems~\ref{thm:cardy-formula} and \ref{thm:backbone-crossing}. Consider the critical percolation on uniform random planar maps with annular topology, under the condition that there exists an open path crossing the annulus. Conditioning on the conformal modulus, the scaling limit of this random annulus can still be described by Liouville CFT, as in the case of the Brownian annulus. However, the law of the moduli for these two random annuli are different: the new law favors configurations that are easier to form a percolation crossing. In fact, the Radon-Nikodym derivative between these two laws is exactly $p_B(r, R)$.
The first step in our derivation of $p_B(r, R)$ is to make this intuition rigorous. We do not appeal to the discrete setting but instead argue directly in the continuum using CLE$_6$ coupled with $\sqrt{8/3}$-LQG.
Our second step is to compute the partition function for the new annulus with given boundary lengths. This, combined with the argument from~\cite{ARS22},
gives the modulus law for the new annulus and hence solves $p_B(r, R)$. The derivation of $p_{BW}(r, R)$ and $p_{BB}(r, R)$ follows the same strategy.

Unlike in the case of the Brownian annulus, we do not use the enumeration results for planar maps to  compute the partition function for the new random annulus with given boundary lengths. We think that the variants corresponding to $p_B(r, R)$ and $p_{BW}(r, R)$ might be accessible from discrete methods, 
but the case for $p_{BB}(r, R)$ would be challenging, as the backbone exponent is transcendental. 
The key technical tool we use instead is a relation between the law of the conformal radii of certain SLE curves on the disk (viewed from the center) and their LQG boundary lengths. This type of relation was first observed in~\cite{AHS21} based on the conformal welding of LQG surfaces and the Girsanov theorem.
Based on this relation, the partition function for the new random annulus in solving $p_B(r, R)$ and $p_{BW}(r, R)$ can be obtained in various  ways by considering a proper conformal radius observable. Our approach relies on the realization of the Brownian annulus by removing a metric ball from the Brownian disk. The needed formulae for the conformal radius are from~\cite{SSW09, ASYZ24}, respectively. 

The case for $p_{BB}(r, R)$ requires additional ideas. The needed conformal radius observable is the following. 
Given a loop from $\CLE_6$ on the unit disk $\mathbb D$, its outer boundary is defined to be the boundary of the unbounded component after removing the loop from the plane. Then the outer boundaries of all CLE$_6$ loops form a random collection of simple loops, each of which looks like an SLE$_{8/3}$-type curve by SLE duality~\cite{Dub05, zhan2008duality}. Let $\eta$ be the outermost loop of this new loop ensemble surrounding the origin. Let ${\rm CR}(0, D_\eta)$ be the conformal radius of the domain $D_\eta$ inside $\eta$, viewed from the origin. The law of ${\rm CR}(0, D_\eta)$ is essentially known from~\cite{NQSZ23, Wu23}, which can give the law of the LQG length of $\eta$ when the $\sqrt{8/3}$-LQG geometry is put on $\mathbb D$ to make it a Brownian disk. However, what we actually need is the law of ${\rm CR}(0, D_\eta)$ conditioning on $\eta \cap \partial   \mathbb{D}=\emptyset$. This is because on this event, the LQG surface bounded by $\partial \mathbb D$ and $\eta$ forms a random annulus whose random modulus encodes $p_{BB}(r, R)$. 
Deriving this conditioned conformal radius is the main technical challenge addressed in Section~\ref{sec:backbone}.

\subsection{Outlook and perspectives}
\label{subsec:outlook}
In this section, we discuss some future work and open questions.

\begin{itemize} 
\item Our method based on LQG on the annulus is quite general. While the Bernoulli percolation and $\sqrt{8/3}$-LQG surfaces offer additional convenience, we can apply the same method to other models that have SLE and CLE as their scaling limits, including Fortuin-Kasteleyn (FK) percolation, the ${\rm Q}$-Potts spin model, and the ${\rm O}(n)$-loop model. For models other than the Bernoulli percolation, the boundary conditions along the annulus will make a difference. We plan to apply the LQG method to study various cases of interest in physics. In particular, for general FK percolation, assuming its convergence to CLE, we plan to derive similar formulae as in~\eqref{eq:thm-backbone-2} and~\eqref{eq:thm-backbone-1} for the counterpart of $p_{BB}(r,R)$ under wired boundary conditions.

\item  Although our formulae for $p_B(r,R)$, $p_{BW}(r,R)$, and $p_{BB}(r,R)$ have a CFT interpretation in terms of the closed channel and open channel expansions, the underlying CFT structure is still poorly understood. 
The cases of $p_B(r,R) $ and $p_{BW}(r,R)$ should be related to the critical loop CFT recently studied by Nivesvivat, Ribault, Jacobsen, Saleur, and others, which is still an active research topic in the physics community; see~\cite{Nivesvivat:2023kfp} and references therein. In~\cite{ACSW-integrability}, Ang, Cai, Wu, and the first-named author computed certain three-point functions related to this CFT, expressed through the so-called imaginary DOZZ formula. The proof is also based on LQG. It would be extremely interesting to reveal the Virasoro action and the operator product expansion of this CFT, as predicted in physics.  Since our formula for $p_{BB}(r,R)$ is new and does not correspond to any existing CFT in the literature, it poses an exciting challenge to physicists. This also adds motivation for deriving the monochromatic $k$-arm exponents with $k\ge 3$ and their corresponding annulus crossing probabilities.

\item Both our formulae for $p_B(r,R)$ and $p_{BW}(r,R)$ were predicted by Cardy using the Coulomb gas method, where the scaling limit of lattice observables is encoded by functional integrals with respect to the Gaussian free field.  It would be extremely interesting to understand how to implement this powerful method rigorously. This method has predicted various other interesting formulae. For example, Cardy~\cite{Car02} gave the percolation crossing formulae for alternating $2n$-arm events for $n \geq 1$, where $p_{BW}(r,R)$ corresponds to $n=1$.
This method can also be applied to lattice models on the torus. For example, it has been used to  derive  the law of the number of non-contractible clusters in percolation on the torus; see~\cite{dFSaleurZuber1987, Pinson1994}. Proving these predictions is currently beyond the reach of our LQG method.

\item 

We see logarithmic behavior in the open channel expansion~\eqref{eq:thm-backbone-2} of $p_{BB}(r,R)$. We can also propose other percolation events in the annulus where logarithmic behavior appears in the closed channel expansion, related to logarithmic bulk CFT. For instance, in Bernoulli percolation on the annulus with black-white boundary conditions, let $p(r,R)$ be the probability that there is exactly one non-contractible percolation interface. Our method can be used to prove the following formula from~\cite{Car06}:
$$
p(r,R) = \frac{1}{\prod_{n=1}^\infty (1-q^n)} \sum_{m \in \mathbb{Z}} (-1)^{m-1} m q^{\frac23 m^2 - m + \frac13}.
$$
Using the Poisson summation formula, the closed channel expansion for $p(r,R)$ is
$$
\frac{3 \sqrt{3}}{4} \cdot \frac{1}{\prod_{n=1}^\infty (1-\tilde q^{2n})} \sum_{k \geq 1, {\rm odd}} \tilde q^{\frac{3 k^2}{16}- \frac{1}{12}} \Big( \tau k \sin(\frac34 \pi k) - \cos( \frac34 \pi k) \Big),
$$
where we see logarithmic corrections to the power law at every term. In particular, the leading term of this probability as $\tau \to \infty$ (or equivalently, $\tilde q \to 0$) is $\frac{3 \sqrt{6}}{8}(\tau + 1) \tilde q^{\frac{5}{48}} + o(\tilde q^{\frac{5}{48}})$.

\item Equation~\eqref{eq:expand-tilde-q} implies that for critical Bernoulli percolation on the triangular lattice with mesh size $\delta$, $\mathbb{P}[\mathcal{A}_B(r,1,\delta)] = \sqrt{\frac{3}{2}} r^{\frac{5}{48}}(1 - r^{\frac32} + o(r^{\frac32}))$, where $\mathcal{A}_B(r,1,\delta)$ is the one-arm crossing event on the annulus $A(r,1)$. It is plausible that such sharp estimates still hold when $r$ is of microscopic order $\delta$. From~\eqref{eq:expand-tilde-q} and~\eqref{eq:thm-backbone-2}, we can also extract similar sharp estimates for the polychromatic and monochromatic two-arm events. In~\cite{DGLZ22}, this type of estimate is proved for polychromatic arm events without identifying the subleading order. Under this asymptotic assumption, Ziff~\cite{Ziff11} heuristically derived the so-called correct-to-scaling exponent for percolation. We expect that obtaining the precise subleading order at the microscopic level is technically challenging. 

\end{itemize}

\medskip
\noindent\textbf{Organization of the paper.} In Section~\ref{sec:percolation-cle}, we give some background on critical percolation and CLE$_6$. In Section~\ref{sec:lqg}, we provide preliminaries on LQG and Brownian surfaces. In Section~\ref{sec:welding}, we introduce more LQG surfaces and prove conformal welding results for certain CLE loops. We prove Theorems~\ref{thm:cardy-formula} and \ref{thm:backbone-crossing} in Sections~\ref{sec:proof-Cardy} and \ref{sec:backbone}, respectively, modulo a few integral identities that we prove in Section~\ref{sec:calculation}.

\medskip
\noindent\textbf{Acknowledgements.} We thank Jesper Jacobsen and Baojun Wu for helpful discussions. X.S.\ was supported by National Key R\&D Program of China (No.\ 2023YFA1010700). 
X.S. was also partially supported by the NSF grant DMS-2027986, the NSF Career grant DMS-2046514, and a start-up grant from the University of Pennsylvania 
at the beginning of the project.
Z.Z.\ is partially supported by NSF grant DMS-1953848. 
\section{Critical Bernoulli percolation and CLE$_6$}\label{sec:percolation-cle}
In this section, we review the relation between critical planar Bernoulli percolation and CLE$_6$. In Section~\ref{subsec:percolation}, we introduce critical Bernoulli percolation on the triangular lattice and the quad-crossing space. In Section~\ref{subsec:cle6}, we introduce CLE$_6$ as the scaling limit of critical percolation and use it to represent the annulus crossing probabilities. We also review the exact law of conformal radii of some CLE$_6$ loops in Proposition~\ref{prop:formula-CR}.

For $z \in \mathbb{C}$ and $0 < r <R$, let $B(z,R) = \{ w \in \mathbb{C}: |w-z|_2 < R\}$ and $A(z,r,R) = \{ w \in \mathbb{C}: r < |w-z|_2 < R\}$. When $z=0$, we will write $A(r,R)$ short for $A(0,r,R)$. Let $\mathbb{D} = B(0,1)$ be the unit disk. A doubly connected domain is a connected domain in $\mathbb{C}$ whose complement consists of two connected components. In this paper, we only consider doubly connected domains whose boundaries are Jordan curves. By the Riemann mapping theorem, any such doubly connected domain $D$ can be conformally mapped to an annulus $A(e^{-2 \pi \tau}, 1)$ for some $\tau>0$. Let ${\rm Mod}(D) = \tau$ be the modulus of $D$. We call such a domain \textit{nice} if both of its boundaries are piecewise smooth curves.

\subsection{Critical Bernoulli percolation and its quad-crossing scaling limit}
\label{subsec:percolation}

In this paper, we focus on the triangular lattice $\mathbb{T} = (V_{\mathbb{T}}, E_{\mathbb{T}})$ with 
$$V_{\mathbb{T}} = \{ x+ye^{i \pi/3} \in \mathbb{C}: x,y \in \mathbb{Z}\} \quad \mbox{and} \quad E_{\mathbb{T}} = \{ (v,v') \in V_{\mathbb{T}} : |v-v'|_2 = 1\}.$$
We define the critical Bernoulli site percolation on $\mathbb{T}$ by coloring each vertex either black (=open) or white (=closed) with equal probability. We also consider percolation on the triangular lattice $\mathbb{T}$ as a random coloring of hexagons on its dual lattice $\mathbb{T}^*$, where each vertex (resp.\ edge) on $\mathbb{T}$ corresponds to a hexagonal face (resp.\ edge) on $\mathbb{T}^*$. Let $\mathbb{P}_{\delta}$ denote the law of the critical Bernoulli site percolation on the rescaled lattice $\delta \mathbb{T}$ with mesh size $\delta > 0$. For a connected domain $D$ and $\delta>0$, let its discretization $D_\delta$ be the largest connected set of $\delta \mathbb{T}$ that is contained in $D$. Note that when $\delta$ is small enough, such set is unique. For a set $A \subset V_{\mathbb{T}}$, let $\partial_i A := \{ x \in A: \exists y \in V_{\mathbb{T}} \setminus A \mbox{ such that }(x,y) \in E_{\mathbb{T}}\}$.

For a doubly connected domain $D$, when $\delta$ is small enough, $\partial_i D_\delta$ has two connected components corresponding to the inner and outer boundaries of $D$, respectively. Let the one-arm event $\mathcal{A}_B(D_\delta)$ be the event that there exists a black path connecting the two connected components of $\partial_i D_\delta$. Similarly, we define the polychromatic (resp.\ monochromatic) two-arm event $\mathcal{A}_{BW}(D_\delta)$ (resp.\ $\mathcal{A}_{BB}(D_\delta)$) as the event that there exist both a black and a white path (resp.\ two disjoint black paths) connecting the two connected components of $\partial_i D_\delta$. The following lemma is a consequence of the fact that critical Bernoulli percolation on the triangular lattice has a conformally invariant scaling limit~\cite{Sm01, CN06, gps-pivotal}.

\begin{lemma}\label{lem:percolation-limit}
    For any doubly connected domain $D$, $\lim_{\delta \rightarrow 0} \mathbb{P}_{\delta}[\mathcal{A}_B(D_\delta)]$ exists and moreover the limit is a decreasing and continuous function of ${\rm Mod}(D)$. The same holds for $\lim_{\delta \rightarrow 0} \mathbb{P}_{\delta}[ \mathcal{A}_{BW}(D_\delta)]$ and $\lim_{\delta \rightarrow 0} \mathbb{P}_{\delta}[\mathcal{A}_{BB}(D_\delta)]$.
\end{lemma}

\begin{remark}\label{rmk:doubly-connected}
    Theorems~\ref{thm:cardy-formula} and~\ref{thm:backbone-crossing} are stated for the continuum limit of the critical planar percolation on $A(r,R)$. By Lemma~\ref{lem:percolation-limit}, the same results hold for any doubly connected domain. Moreover, the crossing probabilities only depend on the conformal modulus of the domain. In light of this, we will use $p_\mathsf{a}(\tau)$ in most parts of the paper to represent the scaling limit of crossing probabilities, where $\mathsf{a} \in \{B, BW, BB\}$ represents different arm events and $\tau$ is the modulus of the doubly connected domain. In particular, for any $R>r>0$, $D = A(r,R)$, and $\mathsf{a} \in \{B, BW, BB\}$, we have
\begin{equation}
\label{eq:percolation-converge-annulus}
\lim_{\delta \rightarrow 0}\mathbb{P}_{\delta}[\mathcal{A}_\mathsf{a}(D_\delta)] = p_\mathsf{a}(\tau) \quad \mbox{with}\quad \tau = \frac{1}{2 \pi} \log(\frac{R}{r}).
\end{equation}
It is easy to see from percolation estimates that $\lim_{\tau \rightarrow 0} p_\mathsf{a}(\tau) = 1$.
\end{remark}

In the rest of this subsection, we provide the proof of Lemma~\ref{lem:percolation-limit} for the case $\mathsf{a} \in \{B, BW\}$. The case $\mathsf{a} = BB$ will be postponed to Section~\ref{subsec:cle6}, which relies on the convergence of percolation towards CLE$_6$. We first review the quad-crossing space introduced in~\cite{ss-quad-crossing} as a way to describe the scaling limit of percolation. Let $E \subset \mathbb{C}$ be an open domain. A quad is a homeomorphism $Q$ from $[0,1]^2$ into $E$. Let $\mathcal{Q}_E$ be the space of all quads equipped with the following metric:
$$
d_\mathcal{Q}(Q_1,Q_2):= \inf_\phi \sup_{z \in [0,1]^2} |Q_1(z) - Q_2(\phi(z))|,
$$where the infimum is over all homeomorphisms $\phi: [0,1]^2 \rightarrow [0,1]^2$ that preserve the four corners $\{0,1\} \times \{0,1\}$. A crossing of a quad $Q$ is a connected closed subset of $Q([0,1]^2)$ that intersects both $Q(\{0\} \times [0,1])$ and $Q(\{1\} \times [0,1])$. In light of this, there is a natural partial order on quads: we write $Q_1 \leq Q_2$ if and only if every crossing of $Q_1$ contains a crossing of $Q_2$. 

A subset $S \subset \mathcal{Q}_E$ is called hereditary if for any $Q_1 \in S$ and $Q_2 \in \mathcal{Q}_E$ satisfying $Q_1 \leq Q_2$, we also have $Q_2 \in S$. We call a closed hereditary subset of $\mathcal{Q}_E$ a quad-crossing configuration on $E$ and denote the space of quad-crossing configurations by $\mathcal{H}(E)$. According to~\cite{ss-quad-crossing}, $\mathcal{H}(E)$ can be endowed with a metric $d_\mathcal{H}$ such that the topological space $(\mathcal{H}(E), d_\mathcal{H})$ is a separable compact Hausdorff space and satisfies the following properties. For any $Q \in \mathcal{Q}_E$, the set $\boxminus_Q := \{S \in \mathcal{H}(E): Q \in S \}$ is closed, and for any open subset $U \subset \mathcal{Q}_E$, the set $\boxdot_U := \{S \in \mathcal{H}(E): S \cap U = \emptyset \}$ is closed. In addition, for any dense subset $\mathcal{Q}_0 \subset \mathcal{Q}_E$, the events $\{\boxminus_Q: Q \in \mathcal{Q}_0\}$ generate the Borel $\sigma$-algebra of $(\mathcal{H}(E), d_\mathcal{H})$.

Any discrete percolation configuration can naturally induce a quad-crossing configuration. Specifically, we associate $\omega \in \mathcal{H}(E)$ with it such that for all $Q \in \mathcal{Q}_E$, $Q \in \omega$ if and only if the union of percolation-wise open hexagons contains a crossing of the quad $Q$. For $\delta>0$, let $\omega_\delta$ be the random variable in $\mathcal{H}_E$ induced by a Bernoulli-$\frac{1}{2}$ site percolation sampled from $\mathbb{P}_{\delta}$. With slight abuse of notation, we also let $\mathbb{P}_{\delta}$ denote the law of $\omega_\delta$. It was proved in Section 2.3 of \cite{gps-pivotal}, based on~\cite{Sm01,CN06}, that $\mathbb{P}_{\delta}$ weakly converges to a measure on $\mathcal{H}(E)$ for the $d_\mathcal{H}$-metric. This is due to the fact that $(\mathcal{H}(E), d_\mathcal{H})$ is a compact topological space, and hence, the sequence of measures $(\mathbb{P}_\delta)_{\delta>0}$ is tight. Moreover, for any finite set of quads, the joint law of the crossing events of these quads weakly converges to a unique limit characterized by SLE$_6$ curves. Hence, $(\mathbb{P}_\delta)_{\delta>0}$ weakly converges to a limit, which we denote by $\mathbb{P}_E$.

For a conformal map $f: E_1 \rightarrow E_2$ and a quad-crossing configuration $\omega \in \mathcal{H}(E_1)$, the pushforward $\widetilde \omega := f \circ \omega \in \mathcal{H}(E_2)$ is defined such that $\boxminus_{f(Q)}(\widetilde \omega) = \boxminus_Q(\omega)$ for all $Q \in \mathcal{Q}_{E_1}$. Similarly, for open domains $E_2 \subset E_1$ and $\omega \in \mathcal{H}(E_1)$, the restriction $\widehat \omega := \omega|_{E_2} \in \mathcal{H}(E_2)$ is defined such that $\boxminus_{Q}(\widehat \omega) = \boxminus_Q(\omega)$ for all $Q \in \mathcal{Q}_{E_2}$.

\begin{lemma}\label{lem:local}
    The family of measures $(\mathbb{P}_E)_{E \subset \mathbb{C}}$ satisfies the following properties:
    \begin{enumerate}[(i)]
        \item For any open domains $E_2 \subset E_1$, $\mathbb{P}_{E_2}$ has the same law as $\omega|_{E_2}$ where $\omega$ is sampled from $\mathbb{P}_{E_1}$. \label{claim:lem-local-2}
        
        \item For any two simply or doubly connected domains $E_1$ and $E_2$, and a conformal map $f: E_1 \rightarrow E_2$, we have $f \circ \mathbb{P}_{E_1} = \mathbb{P}_{E_2}$.\label{claim:lem-local-1}
    \end{enumerate}
    
\end{lemma}

\begin{proof}
    Claim~\eqref{claim:lem-local-2} follows from the scaling limit result and the fact that the discrete percolation configurations satisfy the restriction property. Next, we prove Claim~\eqref{claim:lem-local-1}. The case for simply connected domains was proved in Section 2.3 of~\cite{gps-pivotal}, based on~\cite[Theorem 7]{CN06}. We first extend the result to nice doubly connected domains. For two nice doubly connected domains $E_1$ and $E_2$ with a conformal map $f: E_1 \rightarrow E_2$, let $\omega_1$ (resp.\ $\omega_2$) be a sample from $\mathbb{P}_{E_1}$ (resp.\ $\mathbb{P}_{E_2}$). We cut $E_1$ into a simply connected domain $E_1'$ using a smooth cut. Then, we have $f \circ \mathbb{P}_{E_1'} = \mathbb{P}_{f(E_1')}$, and thus, there exists a coupling such that $f \circ (\omega_1|_{E_1'}) = \omega_2|_{f(E_1')}$. Since the cut is smooth, by~\cite[Theorem 1.19]{ss-quad-crossing}, $\omega_1|_{E_1'}$ and $\omega_2|_{f(E_1')}$ can almost surely determine $\omega_1$ and $\omega_2$, respectively. Moreover, for any fixed $Q \in \mathcal{Q}_{E_1}$, we can use $f \circ \mathbb{P}_Q = \mathbb{P}_{f(Q)}$ to check that under this coupling, $\boxminus_{Q}(\omega_1) = \boxminus_{f(Q)}(\omega_2)$ a.s. This implies that $f \circ \omega_1 = \omega_2$ a.s., and thus $f \circ \mathbb{P}_{E_1} = \mathbb{P}_{E_2}$. The general case of Claim~\eqref{claim:lem-local-1} follows by approximating doubly connected domains with increasing sequences of nice doubly connected domains.
\end{proof}

Now we express arm events as measurable events with respect to the quad-crossing space as in Section 2.4 of \cite{gps-pivotal}. Let $D$ be a doubly connected domain and consider the quad-crossing space $(\mathcal{H}(D), d_\mathcal{H})$. Let $A \subset D$ be a doubly connected domain such that $D \setminus \overline A$ has two connected components, both with annular topology. For $\omega \in \mathcal{H}(D)$, the polychromatic two-arm event $\widetilde{\mathcal{A}}_{BW}(A)$ occurs if and only if there exist two quads $Q_1, Q_2 \in \mathcal{Q}(D)$ such that $Q_1 \in \omega$, $Q_2 \not \in \omega$, and $Q_1(\{0\} \times [0,1])$ and $Q_1(\{1\} \times [0,1])$ (resp.\ $Q_2([0,1]\times \{0\})$ and $Q_2([0,1]\times \{1\})$) are contained in the two connected components of $D \setminus \overline A$ respectively; see Figure~\ref{fig:0}. The one-arm event $\widetilde{\mathcal{A}}_B(A)$ can be defined similarly. By \cite[Lemma 2.5]{gps-pivotal}, both of these events are open in $(\mathcal{H}(D), d_\mathcal{H})$. 

\begin{figure}[h]
\centering
\includegraphics[scale = 0.8]{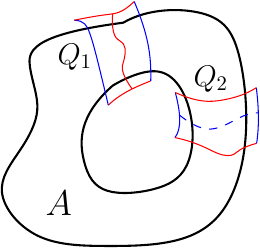}
\caption{Illustration of $\widetilde{\mathcal{A}}_{BW}(A)$. The red boundaries represent $Q_i(\{0\} \times [0,1])$ and $Q_i(\{1 \} \times [0,1])$, while the blue boundaries represent $Q_i([0,1] \times \{0\})$ and $Q_i([0,1] \times \{1\})$ for $i \in \{1,2\}$. The red crossing in $Q_1$ represents an open crossing. The dashed blue crossing in $Q_2$ represents a closed crossing, which exists because $Q_2 \not \in \omega$.} 
\label{fig:0}
\end{figure}

\begin{proof}[Proof of Lemma~\ref{lem:percolation-limit} for the case $\mathsf{a} \in \{ B, BW \}$]
    We will only present the proof for $\mathsf{a} = BW$, and the case for $\mathsf{a} = B$ follows the same argument verbatim.
    
    We begin with the case where $D$ is a nice doubly connected domain. For $\epsilon>0$, let $D^\epsilon$ be the largest connected component of $\{z \in D: d(z, \partial D) > \epsilon\}$. Consider the arm events of $D^\epsilon$ defined with respect to the quad-crossing space $\mathcal{H}(D)$, and recall the definition of $\mathcal{A}_{BW}(D_\delta)$ before Lemma~\ref{lem:percolation-limit}. We first show that $\lim_{\delta \rightarrow 0} \mathbb{P}_{\delta}[\mathcal{A}_{BW}(D_\delta)]$ exists, and moreover,
    \begin{equation}\label{eq:prop2.1-0}
    \lim_{\delta \rightarrow 0} \mathbb{P}_{\delta}[ \mathcal{A}_{BW}(D_\delta)] = \mathbb{P}_D[\omega \in \cap_{\epsilon>0} \widetilde{\mathcal{A}}_{BW}(D^\epsilon)].
    \end{equation}
    
    By \cite[Lemma 2.9]{gps-pivotal}, we have 
    \begin{equation}\label{eq:prop2.1-1}
    \lim_{\delta \rightarrow 0} \mathbb{P}_{\delta}[\omega_\delta \in \widetilde{\mathcal{A}}_{BW}(D^\epsilon)] = \mathbb{P}_D[\omega \in \widetilde{\mathcal{A}}_{BW}(D^\epsilon)] \quad \mbox{for all $\epsilon>0$.}
    \end{equation}The analog of~\eqref{eq:prop2.1-1} for the one-arm event is also proved in \cite[Lemma 2.9]{gps-pivotal}, but the case for the monochromatic two-arm event has not been addressed there. This is the main reason we separate the proof for the case $\mathsf a = BB$.

    Next, we prove the following inequality using percolation estimates:
    \begin{equation}\label{eq:prop2.1-3}
    \limsup_{\delta \rightarrow 0} \big{|}\mathbb{P}_{\delta}[ \mathcal{A}_{BW}(D_\delta)] - \mathbb{P}_{\delta}[ \omega_\delta \in \widetilde{\mathcal{A}}_{BW}(D^\epsilon) ] \big{|} \leq o(\epsilon),
    \end{equation}
    which will be used to show that the left-hand side of~\eqref{eq:prop2.1-1} converges as $\epsilon$ tends to zero. Consider the same sample from $\mathbb{P}_{\delta}$ for both events in~\eqref{eq:prop2.1-3}. By definition, we have $\mathcal{A}_{BW}(D_\delta) \subset \widetilde{\mathcal{A}}_{BW}(D^\epsilon)$ for all sufficiently small $\delta$. We claim that there exists a constant $C>0$ depending only on $D$ such that on the event $\widetilde{\mathcal{A}}_{BW}(D^\epsilon) \setminus \mathcal{A}_{BW}(D_\delta)$, there exists some point $x$ on $\partial_i D_\delta$ with a polychromatic $3$-arm event occurring in $A(x, C\epsilon, C^{-1})$. Suppose that there exists a crossing open cluster of $D^\epsilon$ but not of $D_\delta$ (the case for the existence of such closed cluster can be treated similarly). Let $y$ be any intersection point of the crossing open cluster of $D^\epsilon$ with $\partial D^\epsilon$. Let $x$ be the closest point on $\partial_i D_\delta$ to $y$. If the open cluster does not intersect the connected component of $\partial_i D_\delta$ containing $x$, then along the boundary of the crossing open cluster, we can find two disjoint closed paths from $B(x, C \epsilon)$ to distance $C^{-1}$ for some constant $C>0$. In addition, an open path can be found within the open cluster. This proves the claim. Using the fact that the percolation half-plane $3$-arm exponent is larger than 1~\cite{smirnov-werner-percolation} and that the boundary of $D$ is piecewise smooth, we see that $\mathbb{P}_{\delta}[\widetilde{\mathcal{A}}_{BW}(D^\epsilon) \setminus \mathcal{A}_{BW}(D_\delta)] \leq o(\epsilon)$ for all $\delta>0$. This proves~\eqref{eq:prop2.1-3}. Combining~\eqref{eq:prop2.1-1} and \eqref{eq:prop2.1-3} and taking $\epsilon$ to zero yields~\eqref{eq:prop2.1-0}, where we also used that $\widetilde{\mathcal{A}}_{BW}(D^\epsilon)$ is an increasing event in $\epsilon$.
    
    By~~\eqref{eq:prop2.1-0}, $\lim_{\delta \rightarrow 0} \mathbb{P}_{\delta}[\mathcal{A}_{BW}(D_\delta)]$ can be expressed as the probability of an event measurable with respect to the quad-crossing space $\mathcal{H}(D)$. Hence, by~\eqref{claim:lem-local-1} in Lemma~\ref{lem:local}, this limiting crossing probability depends only on the modulus for nice doubly connected domains. Since $\mathcal{A}_{BW}(D_\delta)$ is a decreasing event with respect to the domain $D$, we have that $\lim_{\delta \rightarrow 0} \mathbb{P}_{\delta}[\mathcal{A}_{BW}(D_\delta)]$ is a decreasing function with respect to ${\rm Mod}(D)$. Using similar arguments to~\eqref{eq:prop2.1-3}, we also see that it is a continuous function of ${\rm Mod}(D)$. 
    
    Finally, we extend the result to general doubly connected domains. For any doubly connected domain $D$, we can approximate it by a decreasing or increasing sequence of nice doubly connected domains such that the moduli of these
    domains converge to ${\rm Mod}(D)$. Combining this with the fact that $\mathcal{A}_{BW}(D_\delta)$ is a decreasing event with respect to the domain $D$ and the continuity of the limiting crossing probabilities with respect to the modulus, we conclude that $\lim_{\delta \rightarrow 0} \mathbb{P}_{\delta}[\mathcal{A}_{BW}(D_\delta)]$ exists and depends only on ${\rm Mod}(D)$. \qedhere
    
\end{proof}

\subsection{CLE$_6$ background and its encoding of the crossing events}
\label{subsec:cle6}

In this section, we first review the construction of CLE$_6$ on the disk and its relation to the scaling limit of critical Bernoulli percolation. Then, in Lemma~\ref{lem:CLE-12arm}, we relate $p_B(\tau)$ and $p_{BW}(\tau)$ to CLE$_6$ on the disk. Next, we introduce CLE$_6$ on the annulus by viewing it as the scaling limit of critical percolation (Lemma~\ref{lem:CLE-double-domain}). In Lemma~\ref{lem:CLE-backbone}, we relate the monochromatic two-arm event to CLE$_6$ on the annulus. As a consequence, we complete the proof of Lemma~\ref{lem:percolation-limit} for the case $\mathsf{a} = BB$. We also review some exact formulae for the conformal radii of certain CLE loops in Proposition~\ref{prop:formula-CR}.

Let us now provide some background on CLE$_6$ on the disk. It was introduced in \cite{CN06} as the scaling limit of critical percolation. The general case where $\kappa \in (8/3,8)$ was constructed in \cite{Sheffield09, SW12}. We will review its construction on the unit disk $\mathbb{D}$ using the continuum exploration tree~\cite{Sheffield09}. Let $a_1, a_2,\ldots$ be a countable dense set of points in $\mathbb{D}$. For $\kappa = 6$, the radial SLE$_6$ process satisfies the target invariance property; see \cite[Proposition 3.14]{Sheffield09}. Therefore, we can find a coupling of radial SLE$_6$ curves in $\mathbb{D}$ from 1 targeted at $a_1,a_2,\ldots$ such that for any $k,l \geq 1$, the radial SLE$_6$ processes targeted at $a_k$ and $a_l$ are the same up to the first time it separates $a_k$ and $a_l$, and then it branches into two curves which continue in the two connected components containing $a_k$ and $a_l$ respectively.

For each $a \in \overline{\mathbb{D}}$, we can take a subsequence $(a_{k_n})$ converging to $a$, and uniquely define a curve $\eta^a$ targeted at $a$ which follows the law of a radial SLE$_6$ curve. Without loss of generality, we assume $a_1 = 0$. Now we review the construction of the CLE loop $\Loop$ surrounding $0$ which almost surely exists. 

\begin{enumerate}[(i)]

    \item Let $\eta:=\eta^o$ be the radial SLE$_6$ curve in $\mathbb{D}$ from 1 and targeted at $0$. Let $\sigma$ be the first time that $\eta$ forms a closed loop around $0$ in the counterclockwise direction.

    \item Let $\sigma'$ be the last time before $\sigma$ that $\eta$ forms a closed loop around 0 ($\sigma' = 0$ if there is no such loop), and $z$ be the leftmost intersection point of $\eta([\sigma', \sigma]) \cap \partial (\mathbb{D} \setminus \eta([0,\sigma']))$ on the boundary of the connected component of $\mathbb{D} \setminus \eta([0,\sigma'])$ containing 0.

    \item Let $\theta$ be the last time before $\sigma$ that $\eta$ visits $z$. We reparametrize the branch $\eta^{z}$ such that $\eta^z ([0,\theta]) = \eta ([0,\theta])$. Then $\Loop$ is defined to be the loop $\eta^z ([\theta,\infty))$.  
\end{enumerate}
The loop $\Loop$ also agrees in law with the concatenation of $\eta([\theta,\sigma])$ and an independent chordal $\SLE_6$ curve in the connected component of $\mathbb{D} \backslash \eta([0,\sigma])$ containing $z$, from $\eta(\sigma)$ to $z$. For $k \geq 2$, the corresponding loop $\mathcal{L}^{a_k}$ surrounding $a_k$ can be constructed analogously. The non-nested CLE$_6$ is then defined by the collection of loops $\{\mathcal{L}^{a_k} : k \geq 1  \}$. By conformal invariance, we can extend this definition to any simply connected domain. (Nested) CLE$_6$ can be obtained by iteratively sampling non-nested CLE$_6$ in the simply connected domains enclosed by these loops. For $\kappa \in (4,8)$, we can obtain the non-nested CLE$_\kappa$ by replacing the radial SLE$_6$ processes in the above construction with radial SLE$_\kappa(\kappa - 6)$ processes whose force point is immediately to the left of the starting point. The nested version follows from the same iteration procedure.

For a continuous loop in $\mathbb{C}$, we can parametrize it by $\mathbb{R} / \mathbb{Z}$, and the distance between two continuous loops is defined by
$$
d_\ell(\gamma_1, \gamma_2):=\inf_\phi \sup_{t \in \mathbb{R}/ \mathbb{Z}} | \gamma_1(t) -  \gamma_2(\phi(t))|,
$$
where the infimum is over all homeomorphisms $\phi: \mathbb{R}/ \mathbb{Z} \rightarrow \mathbb{R}/ \mathbb{Z}$. Collections of loops will be equipped with the Hausdorff distance induced by $d_\ell$. With abuse of notation, we will also use $d_\ell$ to denote the distance on collections of loops. 

By the seminal works~\cite{Sm01, CN06}, we know that the scaling limit of critical Bernoulli percolation can be described by CLE$_6$. Specifically, consider the critical Bernoulli percolation on $\mathbb{D}_\delta$ for $\delta>0$, where all the sites outside $\mathbb{D}_\delta$ are considered to be black. The percolation interfaces are the circuits on $\delta \mathbb{T}^*$ that separate the black and white clusters; see Figure~\ref{fig:1} (left). Then the percolation interfaces converge in law to CLE$_6$ on $\mathbb{D}$ under the distance $d_\ell$. In particular, the outermost percolation interface that surrounds 0 converges in law to $\Loop$ under the distance $d_\ell$.

\begin{figure}[h]
\centering
\includegraphics[scale = 0.8]{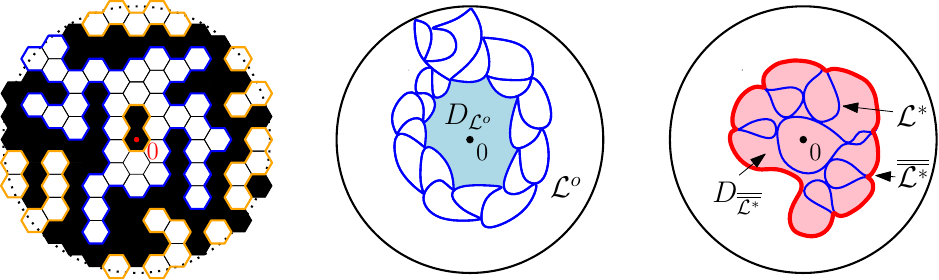}
\caption{\textbf{Left:} Critical Bernoulli percolation on $\mathbb{D}_\delta$, where all vertices outside $\mathbb{D}_\delta$ (not drawn in the figure) are assumed to be black. The colored circuits are percolation interfaces between black and white clusters, with the blue circuit being the outermost one surrounding 0. As $\delta \rightarrow 0$, these circuits converge in law to CLE$_6$. \textbf{Middle:} $\Loop$ is the outermost CLE$_6$ loop that surrounds 0 and the light blue region $D_{\Loop}$ is the simply connected component of $\mathbb{D} \setminus \Loop$ that contains 0. In this figure, we have $\Loop \cap \partial \mathbb{D} \neq \emptyset$. \textbf{Right:} The blue loop $\mathcal{L}^*$ is the outermost CLE$_6$ loop whose outer boundary $\ep{\mathcal{L}^*}$, colored in red, surrounds 0. The pink region $D_{\ep{\mathcal{L}^*}}$ is the simply connected component of $\mathbb{D} \setminus \ep{\mathcal{L}^*}$ that contains 0. In this figure, we have $\mathcal{L}^* \cap \partial \mathbb{D} = \emptyset$.} 
\label{fig:1}
\end{figure}

Using this convergence result, we can derive the following lemma. Recall  $p_{B}(\tau)$ and $p_{BW}(\tau)$ from Lemma~\ref{lem:percolation-limit}.

\begin{lemma}\label{lem:CLE-12arm}
    For any Jordan loop $\ell$ in $\mathbb{D}$ that surrounds $0$, let $A$ be the connected component of $\mathbb{D} \setminus \ell$ that contains 0, and $\tau = {\rm Mod}(\mathbb{D} \setminus \overline A)$ where $\overline A$ is the closure of $A$. Sample a CLE$_6$ on $\mathbb{D}$ and let $\Loop$ be the outermost CLE$_6$ loop that surrounds 0. Let $D_{\Loop}$ be the connected component of $\mathbb{D} \setminus \Loop$ that surrounds 0; see Figure~\ref{fig:1} (middle). Then we have:
    \begin{equation}
        \label{eq:CLE-12arm}
        p_B(\tau) = 1 - \mathbb{P}[A \subset D_{\Loop}] \quad \mbox{and} \quad p_{BW}(\tau) = 2 \cdot \mathbb{P}[\Loop \cap \partial \mathbb{D} \neq \emptyset, A \not \subset D_{\Loop}].
    \end{equation}
\end{lemma}

\begin{proof}
    Let $D = \mathbb{D} \setminus \overline A$. Consider a critical Bernoulli percolation on $\mathbb{D}_\delta$ and let $\mathcal{L}_\delta$ be the outermost percolation interface that surrounds 0. Then as $\delta \rightarrow 0$, $\mathcal{L}_\delta$ converges in law to $\Loop$ for the distance $d_\ell$. We assumed that all the sites outside $\mathbb{D}_\delta$ are black, and thus, the sites neighboring to $\mathcal{L}_\delta$ and enclosed by it are white, while those neighboring to $\mathcal{L}_\delta$ but not enclosed by it are black; see Figure~\ref{fig:1} (left). Then the one-arm event $\mathcal{A}_B(D_\delta)$ happens if and only if $\mathcal{L}_\delta$ does not enclose all the sites on the inner boundary of $D_\delta$. (For small $\delta$, $\partial_i D_\delta$ has two connected components. We say a site is on the inner (resp.\ outer) boundary of $D_\delta$ if it belongs to the connected component of $\partial_i D_\delta$ corresponding to the inner (resp.\ outer) boundary.) Taking $\delta$ to 0 and using Lemma~\ref{lem:percolation-limit}, we see that $1 - \mathbb{P}[ A \subset D_{\Loop}] \leq p_B(\tau) \leq 1 - \mathbb{P}[\overline A \subset D_{\Loop}]$. Let $(A_n)_{n \geq 1}$ be an increasing sequence of domains that converges to $A$ and satisfies $\lim_{n \rightarrow \infty} {\rm Mod}(\mathbb{D} \setminus \overline{A}_n) = \tau$. Then using similar arguments as before and the fact that $\overline A_n \subset A$, we have $p_B({\rm Mod}(\mathbb{D} \setminus \overline{A}_n)) \leq 1 - \mathbb{P}[\overline A_n \subset D_{\Loop}] \leq 1 - \mathbb{P}[A \subset D_{\Loop}]$. Taking $n \rightarrow \infty$ and using the continuity of $p_B(\tau)$ from Lemma~\ref{lem:percolation-limit}, we obtain that $p_B(\tau) \leq 1 - \mathbb{P}[A \subset D_{\Loop}]$. Therefore, $p_B(\tau) = 1 - \mathbb{P}[A \subset D_{\Loop}]$. This proves the first equation in~\eqref{eq:CLE-12arm}. Similarly, we can also show that $p_B(\tau) = 1 - \mathbb{P}[\overline A \subset D_{\Loop}]$, and thus, $\mathbb{P}[A \subset D_{\Loop}] = \mathbb{P}[\overline A \subset D_{\Loop}]$.

    Next, we prove the second equation in~\eqref{eq:CLE-12arm}. Consider the outermost percolation cluster in $\mathbb{D}_\delta$ that surrounds 0, which can be either black or white. By symmetry of the black and white sites under $\mathbb{P}_{\delta}$, we see that $\mathbb{P}_{\delta}[\mathcal{A}_{BW}(D_\delta)]$ equals twice of the probability that the polychromatic two-arm event occurs and this cluster is white. We claim that this is equivalent to the event that $\mathcal{L}_\delta$ touches the boundary of $\mathbb{D}_\delta$ and also intersects the inner boundary of $D_\delta$. First, this white cluster must contain some site on the boundary of $\mathbb{D}_\delta$, otherwise we can find another percolation cluster surrounding it, and thus $\mathcal{L}_\delta$, which is the outer boundary of this white cluster, intersects the boundary of $\mathbb{D}_\delta$. Furthermore, the polychromatic two-arm event $\mathcal{A}_{BW}(D_\delta)$ happens if and only if this white cluster does not enclose all the sites on the inner boundary of $D_\delta$. This proves the claim. Similar to~\eqref{eq:prop2.1-3}, we can show that with probability at least $1 - o_\epsilon(1)$, on the event $d(\mathcal{L}_\delta, \partial \mathbb{D}) < \epsilon$, $\mathcal{L}_\delta$ intersects the boundary of $\mathbb{D}_\delta$. Therefore, taking $\delta$ to 0 and using Lemma~\ref{lem:percolation-limit}, we obtain that $p_{BW}(\tau) = 2 \cdot \mathbb{P}[\Loop \cap \partial \mathbb{D} \neq \emptyset, A \not \subset D_{\Loop}]$. \qedhere
    
\end{proof}

Next, we express the probability of the monochromatic two-arm event in terms of CLE$_6$ and prove Lemma~\ref{lem:percolation-limit} for the case $\mathsf{a} = BB$. In Lemma~\ref{lem:CLE-12arm}, we used CLE$_6$ on $\mathbb{D}$ to express $p_B(\tau)$ and $p_{BW}(\tau)$. To express the monochromatic two-arm event, we need to consider CLE$_6$ defined on a doubly connected domain. We will not need its explicit law in this paper, but instead consider it as the scaling limit of critical Bernoulli percolation on a nice doubly connected domain with black boundary conditions, i.e., we specify all the sites outside the domain to be black. We will also show that it can almost surely determine the quad-crossing space on the doubly connected domain, and vice versa. 

We first review the case on the disk. As shown in \cite{gps-pivotal, hs-cardy-embedding}, the quad-crossing configuration and CLE$_6$ on the disk (and more generally, on simply connected domains with piecewise smooth boundaries) can almost surely determine each other. Specifically, for $\delta>0$, consider a critical Bernoulli percolation on $\mathbb{D}_\delta$ with black boundary conditions and let the induced quad-crossing configuration in $\mathcal{H}(\mathbb{D})$ be $\omega_\delta$. Let $\Gamma_\delta$ be the collection of percolation interfaces. Then, as explained in~\cite{gps-pivotal}, $(\omega_\delta, \Gamma_\delta)$ converge jointly in law as $\delta \rightarrow 0$ with respect to $d_{\mathcal{H}(\mathbb{D})}$ and $d_\ell$. Let $(\omega, \Gamma)$ be the limit, where $\Gamma$ has the same law as CLE$_6$ on $\mathbb{D}$. Then, the quad-crossing configuration $\omega$ can almost surely determine $\Gamma$~\cite{gps-pivotal, hs-cardy-embedding}, and $\Gamma$ can also almost surely determine $\omega$~\cite{CN06, gps-pivotal}, namely they are measurable with respect to each other.

This result can be extended to a nice doubly connected domain.
\begin{lemma}\label{lem:CLE-double-domain}
    For a nice doubly connected domain $D$ and $\delta>0$, sample a critical Bernoulli percolation with black boundary conditions on $D_\delta$. Let $\omega_\delta$ be the induced quad-crossing configuration in $\mathcal{H}(D)$ and $\Gamma_\delta$ be the collection of percolation interfaces. Then, as $\delta \rightarrow 0$, $(\omega_\delta, \Gamma_\delta)$ converge jointly in law with respect to $d_{\mathcal{H}(D)}$ and $d_\ell$. Moreover, let $(\omega, \Gamma)$ be the limit. Then $\omega$ and $\Gamma$ can almost surely determine each other.
\end{lemma}

\begin{proof}
    We cut $D$ into a simply connected domain $D'$ using a smooth cut. Consider the restriction of critical Bernoulli percolation on $D'_\delta$ endowed with black boundary conditions. Let $\omega_\delta'$ be the induced quad-crossing configuration in $\mathcal{H}(D')$. Then, as $\delta \rightarrow 0$, $\omega_\delta'$ converges in law to $\mathbb{P}_{D'}$. By Skorohod's representation theorem, we can find a coupling such that $\omega_\delta'$ converges almost surely to $\omega' \in \mathcal{H}(D')$. By~\cite[Theorem 1.19]{ss-quad-crossing}, we can use $\omega'$ to determine a quad-crossing configuration $\omega \in \mathcal{H}(D)$. Furthermore, by~\cite[Theorem 1.5]{ss-quad-crossing}, for any fixed $Q \in \mathcal{Q}_D$, $\boxminus_{Q}(\omega_\delta)$ is determined by $\omega'_\delta$ up to an error of probability $o_\delta(1)$. Therefore, $\boxminus_{Q}(\omega_\delta)$ converges almost surely to $\boxminus_{Q}(\omega)$, which implies that $\omega_\delta$ converges almost surely to $\omega$ for $d_{\mathcal{H}(D)}$.
    
    Next, we use $\omega$ to construct $\Gamma$ and show that $\Gamma_\delta$ converges almost surely to $\Gamma$ for $d_\ell$. First, observe that for any simply connected domain $\tilde D \subset D$ with a piecewise smooth boundary, $\omega|_{\tilde D}$ can determine a non-simple loop ensemble on $\tilde D$ which follows the law of CLE$_6$ on $\tilde D$. Moreover, by the convergence of critical percolation on $\tilde D$, the percolation interfaces on $\tilde D_\delta$ induced by $\omega_\delta$ (endowed with black boundary conditions) will converge almost surely to this loop ensemble. We take another simply connected domain $D'' \subset D$ that covers the cut associated with $D'$, and consider the loop ensemble $\Gamma'$ (resp.\ $\Gamma''$) induced by $\omega|_{D'}$ (resp.\ $\omega|_{D''})$ on $D'$ (resp.\ $D''$). Then we can determine $\Gamma$ by gluing the loop ensembles $\Gamma'$ and $\Gamma''$, which is well-defined because $\Gamma'$ and $\Gamma''$ are consistent in any simply connected domain due to the convergence of critical percolation in these domains. We can also check that $\Gamma_\delta$ converges to $\Gamma$, as percolation interfaces are locally finite. Meanwhile, the macroscopic interfaces in $\Gamma_\delta$ converge in any simply connected domain to either the loop segments in $\Gamma'$ or $\Gamma''$. Hence, we have $(\omega_\delta, \Gamma_\delta)$ jointly converging to $(\omega, \Gamma)$. 
    
    Finally, we show that under this coupling, $\omega$ and $\Gamma$ can almost surely determine each other. On the one hand, $\omega$ determines both $\Gamma'$ and $\Gamma''$, which in turn determine $\Gamma$. On the other hand, $\Gamma$ determines $\Gamma'$, which then determines $\omega'$ and $\omega$.\qedhere

\end{proof}

\begin{remark}
    There are at least two different versions of CLE$_6$ on a doubly connected domain, depending on whether the colors of Bernoulli percolation on the inner and outer boundaries are considered the same or not. In this paper, we will only consider the case where the boundaries are assigned the same color. In both cases, it is possible to write down the explicit law of CLE$_6$~\cite{sww-cle-doubly, gmq-cle-inversion} and show that the percolation interfaces converge in law to them. Law of the number of non-contractable loops in these two CLE$_6$ can also be explicitly calculated following the LQG method in this paper, as predicted in~\cite{Car06}. We will not pursue this here.
\end{remark}

For a nice doubly connected domain $D$, Lemma~\ref{lem:CLE-double-domain} defines a joint law for a quad-crossing configuration $\omega$ and a non-simple loop ensemble $\Gamma$ on $D$. The marginal law of $\omega$ follows $\mathbb{P}_D$. With slight abuse of notation, we will also use $\mathbb{P}_D$ to denote their joint law. We will refer to the law of $\Gamma$ as the CLE$_6$ on $D$. Using conformal mappings and Claim~\eqref{claim:lem-local-1} in Lemma~\ref{lem:local}, we can define the joint law of $(\omega, \Gamma)$ on any doubly connected domain, which is conformally invariant (albeit without the scaling limit statement). In the following lemma, we prove the convergence of the monochromatic two-arm crossing probabilities on nice doubly connected domains by expressing it using CLE$_6$.

\begin{lemma}\label{lem:CLE-backbone}
    For any nice doubly connected domain $D$ that surrounds 0, sample $\Gamma$ from the law of the CLE$_6$ on $D$. Then, we have:
    \begin{equation}\label{eq:CLE-backbone}
    \lim_{\delta \rightarrow 0}\mathbb{P}_\delta[\mathcal{A}_{BB}(D_\delta)] = \mathbb{P}_D[\mbox{there does not exist }\gamma \in \Gamma \mbox{ such that $\ep{\gamma}$ separates 0 and $\infty$}],
    \end{equation}
    where $\ep{\gamma}$ is the outer boundary of $\gamma$ seen from $\infty$, i.e., the boundary of the infinite connected component of $\mathbb{C} \setminus \gamma$.
\end{lemma}

\begin{proof}
    Sample a critical percolation on $D_\delta$ with black boundary conditions and let $\Gamma_\delta$ be the percolation interfaces. As $\delta \rightarrow 0$, by Lemma~\ref{lem:CLE-double-domain}, $\Gamma_\delta$ weakly converges to $\Gamma$ for the metric $d_\ell$. Define the external frontier of any percolation interface as the sites neighboring it that can be connected to infinity without crossing other such sites. Then, the external frontiers of percolation interfaces converge to the outer boundaries of CLE$_6$ loops. By Menger's theorem, the event $\mathcal{A}_{BB}(D_\delta)$ occurs if and only if there is no percolation interface whose external frontier separates $0$ and $\infty$; see Section 2.2 of \cite{NQSZ23} for more details. Moreover, for $\lambda>0$, with probability at least $1-o_\lambda(1)$, if the external frontier of the $\lambda$-neighbor of any percolation interface separates $0$ and $\infty$, then the external frontier of the percolation interface itself has already separated $0$ and $\infty$. This is because otherwise, a polychromatic 6-arm event would occur in the bulk, or near the boundary, a polychromatic 4-arm event. Meanwhile, the whole-plane 6-arm exponent is larger than 2 and the half-plane 4-arm exponent is larger than 1~\cite{smirnov-werner-percolation}. Taking $\delta$ to 0 yields the lemma.
\end{proof}

Now we finish the proof of Lemma~\ref{lem:percolation-limit}.

\begin{proof}[Proof of Lemma~\ref{lem:percolation-limit} for the case $\mathsf{a} = BB$]
    By Lemma~\ref{lem:CLE-backbone}, for any nice doubly connected domain $D$, $ \lim_{\delta \rightarrow 0}\mathbb{P}_\delta[\mathcal{A}_{BB}(D_\delta)]$ exists and depends only on ${\rm Mod}(D)$. Similar to~\eqref{eq:prop2.1-3} and using the fact that the half-plane 3-arm exponent is larger than 1, we can show that $ \lim_{\delta \rightarrow 0}\mathbb{P}_\delta[\mathcal{A}_{BB}(D_\delta)]$ is a continuous function of ${\rm Mod}(D)$. Thus, we can extend the convergence to any doubly connected domain by approximating them with nice doubly connected domains.
\end{proof}

Finally, we collect exact formulae for the conformal radii of certain domains defined by CLE$_6$ loops from \cite{SSW09, Wu23, NQSZ23, ASYZ24}. For a simply connected domain $D \subset \mathbb{C}$ and $z \in D$, let $f: \mathbb{D} \rightarrow D$ be a conformal map with $f(0) = z$. The conformal radius of $D$ seen from $z$ is defined as ${\rm CR}(z,D) := |f'(z)|$. Recall from Lemma~\ref{lem:CLE-12arm} the definition of $\Loop$ and $D_{\Loop}$, which is relevant to $p_B(\tau)$ and $p_{BW}(\tau)$. For the monochromatic two-arm event, it is natural to consider the outermost CLE$_6$ loop on $\mathbb{D}$ whose outer boundary surrounds 0. See Figure~\ref{fig:1} (right). Let $\mathcal{L}^*$ denote this loop, and let  $\ep{\mathcal{L}^*}$ denote its outer boundary. Let $D_{\ep{\mathcal{L}^*}}$ be the simply connected component of $\mathbb{D} \setminus \ep{\mathcal{L}^*}$ that contains 0. 

\begin{proposition}\label{prop:formula-CR}
    Let $\kappa = 6$. We have:
    \begin{enumerate}[(A)]
\item For $\lambda \leq \frac{3\kappa}{32}+\frac{2}{\kappa}-1$, $\mathbb{E} [ {\rm CR}(0, D_{\Loop})^{\lambda}] = \infty$, and for $\lambda > \frac{3\kappa}{32}+\frac{2}{\kappa}-1$,

\begin{equation}\label{eq:CR-CLE}
\mathbb{E} [ {\rm CR}(0, D_{\Loop})^\lambda] = \frac{\cos(\pi \frac{\kappa-4}{\kappa}) }{ \cos(\frac{\pi}{\kappa} \sqrt{(\kappa-4)^2 - 8 \kappa \lambda})}.
\end{equation}

\item Let $T= \{ \Loop \cap \partial \mathbb{D} \not = \emptyset \}$. Then, we have $\mathbb{P}[T] = \frac{1}{2}$. For $\lambda \leq \frac{\kappa}{8}-1$, $\mathbb{E} [ {\rm CR}(0, D_{\Loop})^{\lambda} 1_{T} ] =\infty$, and for $\lambda > \frac{\kappa}{8}-1$,

\begin{equation}\label{eq:CR-CLE-touch} \mathbb{E} [ {\rm CR}(0, D_{\Loop})^{\lambda} 1_{T} ] = \frac{2\cos(\pi \frac{\kappa-4}{\kappa}) \sin(\pi\frac{\kappa-4}{4\kappa} \sqrt{(\kappa-4)^2 - 8 \kappa \lambda})}{\sin(\frac{\pi}{4} \sqrt{(\kappa-4)^2 - 8 \kappa \lambda})}.
\end{equation}

\item Let $\beta_2$ be the backbone exponent, which is the unique solution in the interval $(\frac{1}{4}, \frac{2}{3})$ to be equation $\frac{\sqrt{36 x +3}}{4} + \sin (\frac{2 \pi \sqrt{12 x +1}}{3}) =0$. For $\lambda \leq - \beta_2$, $\mathbb{E} [ {\rm CR}(0, D_{\ep{\mathcal{L}^*}})^{\lambda}] = \infty$, and for $\lambda > - \beta_2$, we have
\begin{equation}\label{eq:CR-backbone}
\mathbb{E} [ {\rm CR}(0,  D_{\ep{\mathcal{L}^*}})^\lambda] = \frac{\kappa \sin(\frac{8 \pi}{\kappa})}{4 \sin(\frac{\pi \kappa}{4})} \cdot \frac{\sin(\frac{\pi}{4} \sqrt{(\kappa-4)^2 - 8 \kappa \lambda})}{\sin(\frac{2 \pi}{\kappa} \sqrt{(\kappa-4)^2 - 8 \kappa \lambda}) - \frac{1}{4} \sin(\frac{8 \pi}{\kappa}) \sqrt{(\kappa-4)^2 - 8 \kappa \lambda}}.
\end{equation}

\end{enumerate}
\end{proposition}

\begin{proof}
    Equation~\eqref{eq:CR-CLE} is from \cite[Equation (3)]{SSW09}, the identity $\mathbb{P}[T] = \frac{1}{2}$ is from \cite[Theorem 1.1]{ASYZ24}, and Equation~\eqref{eq:CR-CLE-touch} is from \cite[Theorem 1.2]{ASYZ24}. Equation~\eqref{eq:CR-backbone} follows from \cite[Proposition 7.12]{Wu23} and \cite[Theorem 1.4]{NQSZ23}; see Remark 2.4 in \cite{NQSZ23} for the derivation.
\end{proof}

\section{Liouville quantum gravity surfaces and Brownian surfaces}
\label{sec:lqg}

In Section~\ref{subsec:liouville}, we introduce Liouville fields on the annulus and the disk. In Section~\ref{subsec:KPZ}, we review the encoding of the annulus modulus in terms of annulus boundary lengths established in~\cite{ARS22}. In Section~\ref{subsec:brownian-surface}, we give some background on Brownian surfaces. In the rest of the paper, we let $\mathbb{H}=\{z\in \C: {\rm Im} z>0 \}$ be the upper half plane. For $\tau>0$, let $\cC_\tau = [0,\tau]\times[0,1]/{\sim}$ be a finite cylinder where $\sim$ is the identification of $[0, \tau] \times \{0\}$ and $[0, \tau] \times \{1\}$. The modulus of $\cC_\tau$ is $\tau$. Let $\partial_1 \cC_\tau=\{0\}\times [0,1]/{\sim}$ and $\partial_2 \cC_\tau = \{\tau\}\times [0,1]/{\sim}$. For a finite measure $M$, we write $|M|$ as its total mass and define $M^\# = |M|^{-1} M$ as the probability measure proportional to $M$. For infinite measures, we will continue to use probabilistic terminology such as random variable and law. Let $|z|_+ := \max \{|z|, 1\}$ for $z \in \mathbb{C}$.

\subsection{Liouville fields on the annulus and the disk}
\label{subsec:liouville}

We first review the two-dimensional Gaussian free field (GFF); for more background, see e.g.~\cite{sheff-gff, dubedat-coupling}. Let $D$ be a planar domain that is conformally equivalent to a disk or an annulus. Let $\rho(d x)$ be a compactly supported probability measure on $D$ such that  $\iint_{D \times D} \log \frac{1}{|x-y|} \, \rho(d x) \rho (d y)<\infty$. Endow the space $\{f\in C^\infty(D): \int_D f(x) \rho(d x)=0 \}$ with the inner product $(f,g)_\nabla := \frac{1}{2 \pi} \int_D (\nabla f \cdot \nabla g ) d x$, and let $H(D;\rho)$ be its Hilbert closure. Let $(f_n)_{n\ge 1}$ be an orthonormal basis for $H(D;\rho)$. Then, the random series $\sum_{n=1}^\infty \alpha_n f_n$ converges almost surely as a random generalized function, where $(\alpha_n)_{n\ge 1}$ is a sequence of independent standard Gaussian random variables. We normalize $\sum_{n=1}^\infty \alpha_n f_n$ to have average zero over $\rho(d x)$, and call the resulting law on the random generalized function a \emph{free boundary Gaussian free field} (which depends on the choice of $\rho$).

Now we introduce the Liouville field on the annulus. For concreteness, we focus on the horizontal cylinder $\cC_\tau$ to describe the annulus. 
\begin{definition}\label{def:LF-tau}
For $\tau>0$, let $\rho(dx)$ be a probability measure on $\cC_\tau$ as defined before and $\mathbb{P}_{\tau,\rho}$ be the law of the free boundary Gaussian free field on $\cC_\tau$ with $\int_{\cC_\tau} h(x)\rho(d x) =0$. Sample $(h, \mathbf c)$ from $\mathbb{P}_{\tau,\rho}\times dc$ and set $\phi =  h +\mathbf c$. We write $\LF_{\tau}$ for the law of $\phi$, and call $\phi$ a \emph{Liouville field on $\cC_\tau$}.
\end{definition}

By translation invariance of the Lebesgue measure $dc$, the measure $\LF_{\tau}$ does not depend on the choice of the measure $\rho$. Next, we introduce the Liouville field on the disk. We focus on the upper half plane $\mathbb{H}$. We fix $\gamma \in (0,2)$ and set $Q=\frac{\gamma}{2}+\frac{2}{\gamma}$ (although in the application related to critical percolation, we will only consider $\gamma = \sqrt{8/3}$).

\begin{definition}\label{def:LF-H} Let $\mathbb{P}_{\mathbb{H}}$ be the law of the free boundary Gaussian free field on $\mathbb{H}$ with average zero over the semi-circle $\{z\in \mathbb{H}: |z|=1\}$. Sample $(h, \mathbf c)$ from $\mathbb{P}_{\mathbb{H}}\times e^{-Qc} dc$ and set $\phi (z)=  h(z) +\mathbf c- 2Q\log |z|_+$, where $|z|_+ = \max\{|z| , 1\}$.  
We write  $\LF_{\mathbb{H}}$ for the law of $\phi$, and call $\phi$  a  \emph{Liouville field on $\mathbb{H}$}.
\end{definition}
 
We also need Liouville fields with bulk or boundary insertions on the disk.

\begin{definition}\label{def:LF}
    Let $(\alpha, u) \in \mathbb{R} \times \mathbb{H}$ and $(\beta_i,s_i) \in \mathbb{R} \times (\partial \mathbb{H} \cup \{\infty \})$ for $1 \leq i \leq m$, where $m \geq 0$ and all $s_i$'s are distinct. We also assume that $s_i \neq \infty$ for $i \geq 2$. Using the convention that $G_{\mathbb{H}}(z,\infty) := \lim_{w \rightarrow \infty} G_{\mathbb{H}} (z,w) = 2 \log|z|_+$, we let 
    \begin{equation*}
    \begin{aligned}
    &\quad C_{\mathbb{H}}^{(\alpha,u),(\beta_i,s_i)_i} := \\
    & \quad \begin{cases}
        (2 {\rm Im} u)^{-\frac{\alpha^2}{2}} |u|_+^{-2\alpha(Q-\alpha)} \prod_{i=1}^m |s_i|_+^{-\beta_i(Q-\frac{\beta_i}{2})} \times e^{ \sum_{ 1 \leq i < j \leq m } \frac{\beta_i\beta_j}{4} G_{\mathbb{H}}(s_i,s_j) +  \sum_{i=1}^m \frac{\alpha\beta_i}{2} G_{\mathbb{H}}(u,s_i)}\quad  \mbox{  if }s_1 \neq \infty \\
        (2 {\rm Im} u)^{-\frac{\alpha^2}{2}} |u|_+^{-2\alpha(Q-\alpha)} \prod_{i=2}^m |s_i|_+^{-\beta_i(Q-\frac{\beta_i}{2})} \times e^{ \sum_{ 1 \leq i < j \leq m} \frac{\beta_i\beta_j}{4} G_{\mathbb{H}}(s_i,s_j) +  \sum_{i=1}^m \frac{\alpha\beta_i}{2} G_{\mathbb{H}}(u,s_i)}\quad  \mbox{  if }s_1 = \infty
    \end{cases}
    \end{aligned}
    \end{equation*}
    Sample $(h, \bm c)$ from $C_{\mathbb{H}}^{(\alpha,u),(\beta_i,s_i)_i} P_{\mathbb{H}} \times [e^{(\frac{1}{2} \sum_{i=1}^m \beta_i + \alpha - Q)c} dc]$ and set $\phi(z) = h(z) - 2Q\log|z|_+ +\frac{1}{2} \sum_{i=1}^m \beta_i G_{\mathbb{H}} (z,s_i) + \alpha G_{\mathbb{H}}(z,u) + \bm c$. We write ${\rm LF}_{\mathbb{H}}^{(\alpha,u),(\beta_i,s_i)_i}$ for the law of $\phi$.
\end{definition}

\subsection{Solving the random modulus from the annulus boundary lengths}
\label{subsec:KPZ}
We start by reviewing the concept of quantum surfaces. Consider tuples $(D, h, z_1, \ldots, z_n)$, where $D$ is a planar domain, $h$ is a random generalized function on $D$, and $z_i$ are marked points in the
bulk or on the boundary of $D$. We say that $(D, h, z_1, \ldots, z_n) \sim_\gamma(\tilde D, \tilde h, \tilde z_1, \ldots, \tilde z_n)$ if there exists a conformal map $\psi: \tilde{D} \rightarrow D$ such that $\tilde{h}=h \circ \psi+Q \log |\psi^{\prime}|$ and $\tilde z_i = \psi(z_i)$. A quantum surface with marked points in $\gamma$-LQG is an equivalence class under $\sim_\gamma$, denoted as $(D, h, z_1, \ldots, z_n)/{\sim_\gamma}$. We can similarly define a quantum surface decorated with curves. An embedding of a quantum surface is a choice of its representative.

We then introduce the quantum length of a quantum surface. In the above setting, the $\gamma$-LQG length measure on a straight line segment $L$ is defined by $\mathscr{L}_h^\gamma=\lim _{\epsilon \rightarrow 0} \epsilon^{\frac{\gamma^2}{4}} e^{\frac{\gamma}{2} h_\epsilon(z)} d z$, where $d z$ is the Lebesgue measure on $L$ and $h_\epsilon(z)$ is the average of $h$ over $\{w \in D:|w-z|=\epsilon\}$. For instance, if $\phi$ is a Liouville field on $\cC_\tau$, we can define $\mathscr{L}_\phi^\gamma=\lim _{\epsilon \rightarrow 0} \epsilon^{\frac{\gamma^2}{4}} e^{\frac{\gamma}{2} \phi_\epsilon(z)} d z$ on $\partial_1 \mathcal{C}_\tau$ and $\partial_2 \mathcal{C}_\tau$. The definition of quantum length is consistent with the equivalence relation $\sim_\gamma$.

Given a measure $\mathcal{M}$ on quantum surfaces, we can disintegrate $\mathcal{M}$ over the quantum lengths of the boundary arcs. For instance, for $\alpha \in \mathbb{R}$, one can disintegrate the measure $\LF^{(\alpha,i)}_{\mathbb{H}}$ according to its boundary length, i.e., 
 \begin{equation}\label{eq:disint}
     \LF^{(\alpha,i)}_{\mathbb{H}} = \int_0^\infty \LF^{(\alpha,i)}_{\mathbb{H}}(\ell) d\ell,
 \end{equation}
where $\LF^{(\alpha,i)}_{\mathbb{H}}(\ell)$ is supported on the set of quantum surfaces with one interior marked point and boundary length $\ell$. Then we have:

\begin{lemma}\label{lem:LFH-onepoint}
    For $\alpha> \frac{\gamma}{2}$, we have $|\LF^{(\alpha,i)}_\mathbb{H}(\ell)| = C \ell^{\frac{2(\alpha - Q)}{\gamma} - 1}$ for some constant $C>0$.
\end{lemma}
\begin{proof}
    This follows from \cite[Lemma 2.7]{ARS21}.
\end{proof}

The following result is from \cite[Theorem 1.6]{ARS22}. It provides a method for solving the annulus modulus of a quantum annulus described by Liouville fields using its joint boundary length distribution. This result will serve as a key input for our paper.

\begin{proposition}\label{prop:kpz-annulus}
    Let $\gamma \in (0,2)$ and $m(\mathrm{d} \tau)$ be a measure on $(0, \infty)$. For all $x \in \mathbb{R}$, we have
$$
\int_0^{\infty} e^{-\frac{\pi \gamma^2 x^2 \tau}{4}} m(\mathrm{d}\tau)=\frac{2 \sinh (\frac{\gamma^2}{4} \pi x)}{\pi \gamma x \Gamma(1+i x)} \iint \mathscr{L}_\phi^\gamma(\partial_1 \mathcal{C}_\tau) e^{-\mathscr{L}_\phi^\gamma(\partial_1 \mathcal{C}_\tau)} \mathscr{L}_\phi^\gamma(\partial_2 \mathcal{C}_\tau)^{ix} \operatorname{LF}_\tau(\mathrm{d}\phi) m(\mathrm{d}\tau).
$$
\end{proposition}

\subsection{Brownian surfaces}
\label{subsec:brownian-surface}

In this section we review the definition of Brownian surfaces. The \textit{Brownian sphere} is a random metric-measure space obtained in~\cite{legall-uniqueness, miermont-brownian-map} as the Gromov-Hausdorff-Prokhorov scaling limit of uniform quadrangulations. Let BS$_2(1)^\#$ be the probability measure of the Brownian sphere with two marked points and area 1. We define $\BS_2$ as the law on the Brownian sphere without conditioning on the total area.

\begin{definition}\label{def:BS}
For $A>0$, we rescale the area measure of a sample from $\BS_2(1)^{\#}$ by $A$ and let $\BS_2(A)^{\#}$ be the law of this new metric-measure space with two marked points and area $A$. Let $\BS_2=\int_0^\infty A^{-3/2}\BS_2(A)^{\#} d A$. We call a sample from $\BS_2$ a \emph{(free) Brownian sphere}.
\end{definition}

For $A,L>0$, the \textit{Brownian disk} with area $A$ and boundary length $L$ is a random metric-measure space obtained in \cite[Section 2.3]{bet-mier-disk} as the scaling limit of random quadrangulations with boundary. We denote its law by BD$_{0,1}(L;A)^\#$. The subscripts $\{0,1\}$ indicate that a sample from $\BD_{0,1}(L;A)^{\#}$ has no interior marked point and one boundary marked point. Next, we define the probability measure $\BD_{1,0}(L)^\#$ as the law on the Brownian disk with one interior marked point and boundary length $L$.

\begin{definition}\label{def:BD}
For $L>0$, let $\BD_{0,1}(L) = \int_0^\infty \frac{L^\frac{1}{2}}{\sqrt{2\pi A^5}} e^{-\frac{L^2}{2A}} \BD_{0,1}(L;A)^{\#} dA$. Given a sample from $\BD_{0,1}(L)$, we reweight its law by the area divided by the boundary length. We then forget the marked point on the boundary and add an interior marked point according to the probability measure proportional to the area measure. We denote the law of the resulting metric-measure space with one interior marked point by $\BD_{1,0}(L)$. Finally, $\BD_{1,0}(L)^{\#}$ is defined to be the probability measure proportional to $\BD_{1,0}(L)$.
\end{definition}   

The following lemma allows us to identify a Brownian disk within a Brownian sphere. It follows from~\cite[Theorem 3]{legall-disk-snake}. Whenever a subset of a metric space is viewed as a new metric space, we will use the internal metric on this subset. For the rest of the paper, we will use $B_d(x,r)$ to denote the $d$-metric ball of radius $r$ around $x$ and $B_d^\bullet(x,r)$ to denote the filled metric ball.

\begin{lemma}\label{lem:sphere-ball}
Sample $(\mathcal{S}, x, y, d, \mu)$ from $\BS_2$. On the event $d(x,y)>1$, let $D_y$ be the connected component of $\mathcal{S} \setminus B_d(x,1)$ containing $y$ and $\mathscr{L}$ be the boundary length of $D_y$. Conditioned on $d(x,y)>1$ and $\mathscr{L} = \ell$, the conditional law of $(D_y, y)$ is given by $\BD_{1,0}(\ell)^{\#}$.

\end{lemma}

Next, we introduce the \textit{Brownian annulus}, defined as a random metric-measure space obtained by removing a metric ball from a sample of $\BD_{1,0}$, following the definition in~\cite{ARS22}. It also describes the same object arising from the scaling limit of random quadrangulations of an annulus~\cite{bet-mier-surface}; see~\cite{legall-annulus}.

\begin{definition}\label{def:BA}
    Fix $a,b>0$. Sample $(\mathcal{B}, x, d, \mu)$ from $\BD_{1,0}(a)$. On the event $d(x, \partial \mathcal{B})>1$, let $\mathcal{A} := \mathcal{B} \setminus B_d^\bullet(x,1)$ be an annular domain, and $\mathscr{L}$ be the boundary length of $B_d^\bullet(x,1)$. Let ${\rm BA}(a,b)^\#$ be the conditional law of the metric-measure space $(\mathcal{A}, d, \mu)$, conditioned on $\mathscr{L} = b$. Let ${\rm BA} = \iint_0^\infty \frac{1}{2 \sqrt{ab}(a+b)} {\rm BA}(a,b)^\# dadb$. We call a sample from ${\rm BA}$ a \emph{(free) Brownian annulus}.  
\end{definition}

By Lemma~\ref{lem:sphere-ball}, the law ${\rm BA}(a,b)^\#$ can be obtained by removing two unit metric balls from a sample of ${\rm BS}_2$. From this, we see that ${\rm BA}(a,b)^\# = {\rm BA}(b,a)^\#$ for any $a,b>0$; see also~\cite[Lemma 6.4]{ARS22}.

Miller and Sheffield \cite{lqg-tbm1,lqg-tbm2,lqg-tbm3} provided the $\sqrt{8/3}$-LQG description of the conformally embedded Brownian surfaces. In this subsection, we will review the results for the disk and annulus. The sphere result will be reviewed in Section~\ref{sec:backbone}. Fix $\gamma = \sqrt{8/3}$. Recall from~\eqref{eq:disint} that $\{\LF^{(\gamma,i)}_{\mathbb{H}}(\ell)\}_{\ell>0}$ is the disintegration of $\LF^{(\gamma,i)}_{\mathbb{H}}$ with respect to the boundary length. Let $\LF^{(\gamma,i)}_{\mathbb{H}}(\ell)^\#$ be the probability measure proportional to $\LF^{(\gamma,i)}_{\mathbb{H}}(\ell)$. Sample $\phi$ from $\LF^{(\gamma,i)}_{\mathbb{H}}(\ell)^\#$. By~\cite{dddf-lfpp, gm-uniqueness}, we can choose a smooth regularization $\phi_\eps$ of $\phi$ and $c_\eps > 0$ such that as $\eps \to 0$, the Euclidean metric on $\mathbb{H}$ reweighted by $c_\eps e^{\frac{\gamma}{4} \phi_\eps}$ converges in probability to a metric $d_\phi$ on $\mathbb{H}$, which is called the $\sqrt{8/3}$-LQG metric of $\phi$. Moreover, $d_\phi$ agrees (modulo a multiplicative constant) with the metric constructed earlier in~\cite{lqg-tbm1,lqg-tbm2} via the quantum Loewner evolution. The following theorem summarizes the relation between $\LF^{(\gamma,i)}_\mathbb{H}$ and the Brownian disk.
\begin{theorem}\label{thm:bm-lqg}
There exist constants $c_1, c_2 >0 $ such that the following holds for all $\ell>0$. Sample $\phi$ from $\LF^{(\gamma,i)}_\mathbb{H}(\ell)^{\#}$. Let $(\cD,p, d,\mu)$ be the marked metric-measure space given by  $(\mathbb{H},i, c_1 d_\phi,c_2 \mu_\phi)$. Then $(\cD,p, d,\mu)$ is
a Brownian disk with an interior marked point and boundary length $\ell$; namely, its law is $\BD_{1,0}(\ell)^{\#}$. Moreover, the notion of $\sqrt{8/3}$-LQG length measure agrees with the boundary length measure for the Brownian disk. Finally, $(\cD,p, d,\mu)$ and the quantum surface $(\mathbb{H}, \phi,i)/{\sim_\gamma}$ are measurable with respect to each other. 
\end{theorem}

\begin{proof}
By~\cite[Theorem 3.4]{ARS21}, $(\H, \phi,i)/{\sim_\gamma}$ is a quantum disk with an interior marked point and boundary length $\ell$, as defined in the mating-of-trees framework~\cite{DMS14}. By~\cite{lqg-tbm1,lqg-tbm2, lqg-tbm3} and~\cite[Section 6]{legall-disk-snake}, it is possible to find appropriate constants $c_1$ and $c_2$ such that the law of $(\cD,p, d,\mu)$ is $\BD_{1,0}(\ell)^{\#}$ as a metric-measure space, and the two notions of boundary lengths agree. The measurability result was initially proved in~\cite{lqg-tbm3} and a more constructive proof was given in~\cite{gms-poisson-voronoi}. 
\end{proof}

\begin{remark}\label{rmk:equivalence}
In this paper, we set $c_1 = 1$ by choosing the normalizing sequence $\{c_\epsilon\}$ appropriately. The constant $c_2$ is canonically defined due to the different notions of area for Brownian surfaces and $\sqrt{8/3}$-LQG surfaces. As explained in~\cite[Remark 3.12]{ghs-metric-peano}, $c_2=\sqrt{3}$. In this paper, we do not require properties related to area measures, and thus, we treat $\LF^{(\gamma,i)}_\mathbb{H}(\ell)^{\#}$ and $\BD_{1,0}(\ell)^{\#}$ as the same law on random metric spaces with well-defined boundary length measures. We will use the same convention for other Brownian surfaces and quantum surfaces considered in the rest of the paper.
\end{remark}

The following lemma is derived from \cite[Lemma 6.8]{ARS22}, combined with Theorem~\ref{thm:bm-lqg}. Its proof is based on the enumeration and scaling limits of random quadrangulations; see Appendix B therein.

\begin{lemma}\label{lem:law-metric-ball}
Fix $a>0$ and let 
\begin{equation}\label{eq:def-f}
    \cf(b) = e^{-\frac{9}{2}b} \quad \mbox{for }b>0.
\end{equation}
Sample $(\mathbb{H}, \phi, i)$ from $\LF^{(\gamma,i)}_\mathbb{H}(a)^{\#}$ and let $d_\phi$ be the $\sqrt{8/3}$-LQG metric of $\phi$. On the event that $d_\phi(i, \partial \mathbb{H})>1$, let $\mathscr{L}$ be the quantum boundary length of $\mathcal{B}_{d_\phi}^{\bullet}(i, 1)$. Then the law of $\mathscr{L}$ conditioned on $d_\phi(i, \partial \mathbb{H})>1$ is
$$
\frac{1}{Z(a)}\frac{a^{3/2}
   b \cf(b)
   }{\sqrt{ab}(a+b)}1_{b>0} d b \quad \textrm{where }  Z(a)=\int_0^\infty   \frac{a^{3/2} b\cf(b) d b}{\sqrt{ab}(a+b)}.
$$
Moreover, $\LF^{(\gamma,i)}_\mathbb{H}(a)^{\#}[d_\phi(i, \partial \mathbb{H})>1]=Z(a) / Z(\infty)$ where $Z(\infty)=\int_0^{\infty} b^{\frac{1}{2}} \cf(b) \mathrm{d} b$.
\end{lemma}

In~\cite{ARS22}, the authors provided the $\sqrt{8/3}$-LQG description of the free Brownian annulus ${\rm BA}$. They also derived the distribution of its modulus, which is measurable with respect to the metric-measure space (see their Lemma 6.7). We first review an equivalent construction of ${\rm BA}$ using $\LF^{(\gamma,i)}_\mathbb{H}$ and then review the $\sqrt{8/3}$-LQG description of ${\rm BA}$. 
\begin{lemma}\label{lem:BA-equivalent}
    Let $\gamma = \sqrt{8/3}$ and choose constant $C_1$ such that $|\LF^{(\gamma,i)}_\mathbb{H}(\ell)| = C_1 \ell^{-\frac{3}{2}}$, in light of  Lemma~\ref{lem:LFH-onepoint}.  Let $C_2 = 2 C_1/ Z(\infty)$. Sample $(\mathbb{H}, i, \phi)$ from $\LF_\mathbb{H}^{(\gamma,i)}$ and let $d_\phi$ be the $\sqrt{8/3}$-LQG metric of $\phi$. On the event $E=\{d_\phi(i, \partial \mathbb{H}) > 1\}$, let $\mathcal{A} := \mathbb{H} \setminus B^\bullet_{d_\phi}(i, 1)$ be an annular domain and $\mathscr{L}$ be the quantum boundary length of $B_{d_\phi}^\bullet(i, 1)$. Then under the reweighted measure $M = \frac{1_E}{C_2 \mathscr{L} \cf (\mathscr{L})} \LF_\mathbb{H}^{(\gamma,i)}$, the law of $(\mathcal{A},\phi)/{\sim_\gamma}$ is ${\rm BA}$.
\end{lemma}

\begin{proof}
    Recall that $\mathscr{L}_\phi^\gamma(\partial \mathbb{H})$ is the boundary length of a sample $\phi$ from $\LF_\mathbb{H}^{(\gamma,i)}$. By definition, the joint law of $\mathscr{L}_\phi^\gamma(\partial \mathbb{H})$ and $\mathscr{L}$ under $M$ is given by
    $$
    |\LF^{(\gamma,i)}_\mathbb{H}(a)| \times \LF^{(\gamma,i)}_\mathbb{H}(a)^{\#}[E] \times p(\mathscr{L} = b|E) \times \frac{1}{C_2 b \cf(b)} 1_{a,b>0} da db,
    $$
    where $p(\mathscr{L} = b|E)$ is the probability density of $\mathscr{L}$ conditioned on the event $E$. By Lemmas~\ref{lem:LFH-onepoint} and \ref{lem:law-metric-ball}, this equals to
    $$
    C_1 a^{-3/2} \times \frac{Z(a)}{Z(\infty)} \times \frac{1}{Z(a)}\frac{a^{3/2}
    b \cf(b)}{\sqrt{ab}(a+b)} \times \frac{1}{C_2 b \cf(b)} 1_{a,b>0} da db = \frac{1}{2\sqrt{ab}(a+b) } 1_{a,b>0} da db.
    $$
    Combined with Definition~\ref{def:BA} and Theorem~\ref{thm:bm-lqg}, we conclude.
\end{proof}

We also have an equivalent construction of ${\rm LF}_\H^{(\gamma,i)}$ with $\gamma = \sqrt{8/3}$ from the Brownian sphere, and the proof follows from similar arguments.
\begin{lemma}\label{lem:law-sphere-metric-ball}
    Let $\gamma = \sqrt{8/3}$. Sample $(\mathcal{S}, x, y, d, \mu)$ from $\BS_2$. On the event $O = \{ d(x,y)>1 \}$, let $D_y$ be the connected component of $\mathcal{S} \setminus B_d(x,1)$ containing $y$, and let $\mathscr{L}$ be the quantum boundary length of $D_y$. Then, under the reweighted measure $\frac{1_O}{\mathscr{L} \cf(\mathscr{L})}\BS_2$, the law of $(D_y, y)$ is equivalent to ${\rm LF}_\H^{(\gamma,i)}$ up to a multiplicative constant.
\end{lemma}
\begin{proof}
    Similar to Lemma~\ref{lem:law-metric-ball}, we can show that the law of $\mathscr{L}$ conditioned on $d(x,y)>1$ is $C^{-1} a^{-1/2} \cf(a) 1_{a>0} da$, where $C = \int_0^\infty a^{-1/2} \cf(a) da$. Hence, under the reweighted measure, the law of $\mathscr{L}$ is $C' a^{-3/2} 1_{a>0} da$. Recall that $|{\rm LF}_\H^{(\gamma,i)}(a)| = C_1 a^{-\frac{3}{2}}$ for $a>0$. The lemma then follows by combining Lemma~\ref{lem:sphere-ball} and Theorem~\ref{thm:bm-lqg}.
\end{proof}

Now we recall the $\sqrt{8/3}$-LQG description of the free Brownian annulus from~\cite{ARS22}.

\begin{proposition}[Theorem 1.3 in \cite{ARS22}]\label{prop:BA-Liouville}
    We have ${\rm BA} = \frac{1}{\sqrt{2}} \eta(2 i \tau) {\rm LF}_\tau(d \phi) 1_{\tau>0} d \tau$. Namely, if we sample $\tau$ from $\frac{1}{\sqrt{2}} \eta(2 i \tau) 1_{\tau>0} d \tau$ and $\phi$ from ${\rm LF}_\tau(d \phi)$, then the law of $(\cC_\tau, \phi)/{\sim_\gamma}$ is ${\rm BA}$.
\end{proposition}

\section{Conformal welding results}\label{sec:welding}

We will review several quantum surfaces and their conformal welding results in Sections~\ref{subsec:quantum-disk} and \ref{subsec:conformal-welding-known}. Then in Section~\ref{subsec:conformal-welding-CLE} we derive the conformal welding results for the CLE loops $\mathcal{L}^o$ and $\mathcal{L}^*$ defined in Section~\ref{subsec:cle6}. The mains results that will be used in Sections~\ref{sec:proof-Cardy} and~\ref{sec:backbone} are Propositions~\ref{prop:conformal-welding-cardy} and \ref{prop:weld-backbone}. In this section, we consider $\gamma \in (\sqrt{2},2)$, $\kappa = 16/\gamma^2,$ and $\tilde \kappa = \gamma^2$. While our primary focus in this paper is on the case $\gamma = \sqrt{8/3}$, all the results in this section hold for any $\gamma \in (\sqrt{2},2)$. 

\subsection{Quantum disks, triangles and annuli}
\label{subsec:quantum-disk}

We first introduce quantum disks and triangles. The two-pointed (thick) quantum disk was introduced in \cite{DMS14}. Let $\mathcal{S}$ be the horizontal strip $\mathbb{R} \times (0, \pi)$ and $h_\mathcal{S}$ be the free boundary Gaussian free field on $\mathcal{S}$ with average zero over the vertical segment $\{ 0 \} \times (0, \pi)$. The lateral part of $h_{\mathcal{S}}$ is defined as the projection of $h_{\mathcal{S}}$ onto the subspace of functions that have mean zero on each vertical segment $\{ t \} \times (0,\pi)$ for all $t \in \mathbb{R}$ in the Hilbert space.

\begin{definition}\label{def:thick-disk}
    Fix $W \geq \frac{\gamma^2}{2}$, and let $\beta = \gamma + \frac{2-W}{\gamma} \leq Q$. Let $h = h^1 + h^2 + \bm{c}$, where $h^1, h^2$ are random generalized functions on $\mathcal{S}$, $\bm{c} \in \mathbb{R}$, and they are independently sampled as follows:
    \begin{enumerate}
        \item Let $h^1(z) = X_{{\rm Re} z}$ for $z \in \mathcal{S}$, where
        \begin{equation*}
        X_t \overset{d}{=} \begin{cases}
        B_{2t} - (Q -\beta)t \quad  \mbox{  if }t \geq 0 \\
        \widetilde B_{-2t} + (Q-\beta) t \quad \mbox{if }t < 0
        \end{cases},
        \end{equation*}
        and $(B_t)_{t \geq 0}$ and $(\widetilde B_t)_{t \geq 0}$ are two independent standard Brownian motions conditioned on $B_{2t} - (Q-\beta)t<0$ and $\widetilde B_{2t} - (Q-\beta)t<0$ for all $t > 0$;
        \item $h^2$ has the same law as the lateral part of $h_{\mathcal{S}}$;
        \item $\bm c$ is sampled from the infinite measure $\frac{\gamma}{2} e^{(\beta-Q)c} dc$.
    \end{enumerate}
    Let $\mathcal{M}^{\rm disk}_2(W)$ be the infinite measure describing the law of $(\mathcal{S},h,-\infty,\infty)/{\sim_\gamma}$. A sample from $\mathcal{M}^{\rm disk}_2(W)$ is referred to as a \emph{(two-pointed) quantum disk of weight $W$}.
\end{definition}

For $W \in (0,\frac{\gamma^2}{2})$, the two-pointed (thin) quantum disk was defined in \cite{AHS20} as follows.
\begin{definition}
\label{def:thin-disk}
    Fix $W \in (0,\frac{\gamma^2}{2})$. The infinite measure $\mathcal{M}^{\rm disk}_2(W)$ on two-pointed beaded quantum surfaces is defined as follows. First, we sample $T$ from the distribution $(1-\frac{2}{\gamma^2} W)^{-2} {\rm Leb}_{\mathbb{R}_+}$. Next, sample a Poisson point process $\{(u, \mathcal{D}_u)\}$ from the measure ${\rm Leb}_{\mathbb{R}_+} \times \mathcal{M}^{\rm disk}_2(\gamma^2-W)$. Finally, concatenate the disks $\{\mathcal{D}_u\}$ for $u\leq T$ in the order given by $u$. The left (resp.\ right) boundary length of $\mathcal{M}^{\rm disk}_2(W)$ is defined as the total sum of the left (resp.\ right) boundary lengths of all the $\mathcal{D}_u$'s before time $T$.
\end{definition}

The quantum triangle, parameterized by three weights $W_1,W_2,W_3>0$, was introduced in \cite{ASY22}. This quantum surface is defined using Liouville fields with three insertions and the thick-thin duality.

\begin{definition}[Thick quantum triangles]\label{def:thick-triangle}
Fix {$W_1, W_2, W_3>\frac{\gamma^2}{2}$}. Let $\beta_i = \gamma+\frac{2-W_i}{\gamma} < Q$ for $i=1,2,3$, and sample $\phi$ from $\frac{1}{(Q-\beta_1)(Q-\beta_2)(Q-\beta_3)}\textup{LF}_{\mathbb{H}}^{(\beta_1, \infty), (\beta_2, 0), (\beta_3, 1)}$. We define the infinite measure $\textup{QT}(W_1, W_2, W_3)$ to be the law of $(\bbH, \phi, \infty, 0,1)/{\sim_\gamma}$.
\end{definition}

\begin{definition}[Quantum triangles with thin vertices]\label{def:thin-triangle}
	Fix {$W_1, W_2, W_3\in (0,\frac{\gamma^2}{2})\cup(\frac{\gamma^2}{2}, \infty)$}. Let $I$ be the indices $i$ from $\{1,2,3\}$ where $W_i<\frac{\gamma^2}{2}$. Let $\tilde W_i = W_i$ if $i \not \in I$; and $\tilde W_i = \gamma^2 - W_i$ if $i \in I$. Sample $(S_0, (S_i)_{i \in I})$ from 
 \[\QT(\tilde W_1, \tilde W_2, \tilde W_3) \times \prod_{i\in I } (1-\frac{2W_i}{\gamma^2}) \Md_2(W_i).  \]
 Embed $S_0$ as $(D, \phi, a_1, a_2, a_3)$. For each $i \in I$, embed $S_i$ as $(\tilde D_i, \phi, \tilde a_i, a_i)$ such that the domains $\tilde D_i$ are disjoint and $\tilde D_i \cap D = a_i$. Set $\tilde a_i = a_i$ for each $i \not \in I$, and let $\tilde D = D \cup \bigcup_{i \in I} \tilde D_i$. Define $\QT(W_1, W_2, W_3)$ as the law of $(\tilde D, \phi, \tilde a_1, \tilde a_2, \tilde a_3)/{\sim_\gamma}$.
\end{definition} 

We now review the definition of quantum disks, either unmarked or marked by $m$ bulk points and $n$ boundary points, as described in \cite[Definition 2.2]{AHS21}.
\begin{definition}
\label{def:QD}
    Let $(\mathcal{S}, h, -\infty, \infty)/{\sim_\gamma}$ be a sample from $\mathcal{M}^{\rm disk}_2(2)$. Let $\mu_h$ and $\nu_h$ be its quantum area measure and length measure. We write $\QD$ for the law of $(\mathcal{S}, h) /{\sim_\gamma}$
    under the reweighted measure $\nu_h(\partial \mathcal{S})^{-2} \mathcal{M}^{\rm disk}_2(2)$.
    Fix integers $m,n \geq 0$, and let $(\mathcal{S}, h)$ be a sample from $\mu_h(\mathcal{S})^m \nu_h(\partial \mathcal{S})^n \QD$. Next, sample $z_1,\ldots, z_m$ and $w_1,\ldots w_n$ independently according to the probability measures $\mu_h^\#$ and $\nu_h^\#$, respectively. We write $\QD_{m,n}$ for the law of $$(\mathcal{S},h,z_1,\ldots,z_m,w_1,\ldots ,w_n)/{\sim_\gamma}.
    $$
\end{definition}

We now introduce the notion of a pinched quantum annulus, which is a measure on beaded quantum surfaces.

\begin{definition}\label{def:QA}
    For $W\in(0,\frac{\gamma^2}{2})$, define the measure $\wt\QA(W)$ on beaded surfaces as follows. First, sample $T$ from $(1-\frac{2W}{\gamma^2})^{-2}t^{-1}\mathds{1}_{t>0}dt$, and let $S_T:=\{z\in\bbC:|z| = \frac{T}{2\pi}\}$ be the circle with perimeter $T$. Next, sample a Poisson point process $\{(u,\cD_u)\}$ from the measure $\mathrm{Leb}_{S_T}\times\Md_2(\gamma^2-W)$. Finally, concatenate the disks $\{\cD_u\}$ to form a cyclic chain according to the ordering induced by $u$. The outer (resp.\ inner) boundary of $\wt\QA(W)$ consists of the left (resp.\ right) boundaries of all the $\cD_u$'s. A sample from $\wt\QA(W)$ is called a \emph{pinched quantum annulus of weight $W$}, and $T$ is referred to as its quantum cut point measure.
\end{definition}

\begin{lemma}\label{lem:QA-resample}
    Let $W\in(0,\frac{\gamma^2}{2})$. We write $\wt \QA_\bullet(W)$ as the law obtained by first sampling a quantum surface from $\wt \QA (W)$, reweighted by its outer boundary length, and then marking a boundary point independently according to the probability measure proportional to its quantum length measure on the outer boundary. We define $\Md_{2, \bullet}(\gamma^2 - W)$ similarly by adding a boundary marked point to the left boundary of a sample from $\Md_{2}(\gamma^2 - W)$. Then, we have
    $$
    \wt \QA_\bullet(W) = \Md_{2, \bullet}(\gamma^2 - W) \times \Md_2(W)
    $$
    where the right-hand side denotes the infinite measure on beaded quantum surfaces obtained by cyclically concatenating a pair of independent samples from $\Md_{2, \bullet}(\gamma^2 - W)$ and $ \Md_2(W)$.
\end{lemma}

\begin{proof}
    By Definition~\ref{def:QA} and \cite[Lemma 4.1]{PPY92}, $\wt \QA_\bullet(W)$ can be equivalently defined as follows. First, sample a quantum annulus from $\wt \QA(W)$ with quantum cut point measure $T$. Then, sample $u$ from ${\rm Leb}_{S_T}$ and add a quantum surface independently sampled from $\Md_{2, \bullet}(\gamma^2 - W)$ at point $u$. The lemma follows by integrating over $u$.
\end{proof}

Next, we revisit the notions of forest lines and generalized quantum surfaces as considered in \cite{DMS14, MSW21-nonsimple, AHSY23}. For $\kappa\in(4,8)$, the forested line is defined in~\cite{DMS14} using the \emph{$\frac{\kappa}{4}$-stable looptrees} discussed in~\cite{CK13looptree}. Let $(X_t)_{t\geq 0}$ be a stable L\'{e}vy process starting from 0 with index $\frac{\kappa}{4}\in(1,2)$ that only has upward jumps. The graph $\{(t, X_t) : t \geq 0\}$ is formed by identifying pairs of points $(s, X_s)$ and $(t, X_t)$ if and only if $X_s = X_t$ and the horizontal segment connecting $(s, X_s)$ and $(t, X_t)$ does not intersect $\{(u, X_u): s < u < t\}$. Additionally, if $t$ is an upward jump time of $X$, then $(t, X_t)$ is identified with $(t, X_{t-})$. This construction results in a tree of topological disks, each corresponding to a positive jump in $(X_t)_{t \geq 0}$. Each topological disk is then replaced with an independent sample from $\QD(\ell)^\#$ where $\ell$ is the size of the corresponding jump. This process defines the forested line.

\begin{definition}[Forested line]\label{def:forested-line}
For $\gamma \in (\sqrt{2}, 2)$, let $(X_t)_{t\geq 0}$ be a stable L\'evy process of index $\frac{4}{\gamma^2}>1$ that only has positive jumps and starts from 0. For $t>0$, let $Y_t=\inf\{s>0:X_s\le -t\}$, and fix the multiplicative constant of $X$ such that $\bbE[e^{-Y_1}] = e^{-1}$. Define the forested line as described above. The unique point corresponding to $(0,0)$ on the graph of $X$ is called the \emph{root}. The \emph{forested boundary arc} refers to the closure of the collection of the points on the boundaries of the quantum disks, while the \emph{line boundary arc} refers to the set of the points corresponding to the running infimum of $(X_t)_{t\geq 0}$.

The line boundary arc is parametrized by quantum length, while the forested boundary arc is parametrized by \emph{generalized quantum length}, which is the length of the corresponding interval of $(X_t)$. For instance, for a point on the line boundary arc with quantum length $t$ to the root, the segment of the forested boundary arc between this point and the root has a generalized quantum length $Y_t$. 

\end{definition}
 
As described in~\cite{AHSY23}, a truncation operation can be defined on forested lines. Specifically, for $t>0$ and a sample of a forested line obtained from the L\'evy process $X$, we consider the surface generated by $(X_s)_{0\leq s\leq Y_t}$. This beaded quantum surface is called a forested line segment, which has a line boundary arc with quantum length $t$ and a forested boundary arc with generalized quantum length $Y_t$.

\begin{definition}\label{def:forested-disk}
   Let $\mathcal{S}$ be a measure on quantum surfaces. Sample $S$ from the measure $\mathcal{S}$, and for each boundary arc of $S$, sample independent forested line segments such that their quantum lengths match the boundary arc lengths of $S$. Then, glue them to $S$ accordingly. Let $S^\fd$ be the resulting beaded quantum surface. This procedure is called \emph{foresting the boundary} of $S$. Let $\mathcal{S}^\fd$ be the law of $S^\fd$, which will be called a measure on \emph{generalized quantum surfaces}. 
\end{definition}

\begin{remark}
For each boundary arc of $S$, we select a specific endpoint and forest the boundary arc with this point as the root. According to \cite[Proposition 3.11]{AHSY23}, foresting the boundary arc is equivalent to attaching a Poisson point process of generalized quantum disks to the boundary arc, with intensity proportional to the quantum boundary length. Therefore, the choice of the endpoint does not affect the resulting law of the generalized quantum surfaces. When there are no marked points on the boundary arc (i.e., it is an entire boundary component), we sample a point according to the probability measure proportional to the quantum length measure and consider this random point as the root. 
\end{remark}

For $W_1,W_2,W_3>0$, we denote by $\QT^\fd(W_1,W_2,W_3)$ the law of the generalized quantum surface obtained by foresting the three boundary arcs of a quantum triangle sampled from $\QT(W_1,W_2,W_3)$. For $W \in (0, \frac{\gamma^2}{2})$, we denote by $\wt\QA^\fd(W)$ the law of the generalized quantum annulus obtained by foresting the outer and inner boundary arcs of a quantum annulus sampled from $\wt\QA(W)$. Other generalized quantum surfaces are defined in a similar manner.

\begin{lemma}\label{lem:QD-forested-resample}
    Let $\QD_{1,0}^\fd$ and $\QD_{1,1}^\fd$ be the measures on generalized quantum surfaces obtained from $\QD_{1,0}$ and $\QD_{1,1}$ as described in Definition~\ref{def:forested-disk}. Let $\QD_{1,0,\bullet}^\fd$ be the measure on generalized quantum surface obtained by first sampling a generalized quantum surface from $\QD_{1,0}^\fd$ reweighted by the generalized quantum length and then choosing a boundary marked point according to the probability measure proportional to the generalized quantum length measure. Then for some constant $C>0$, we have
    $$
    \QD_{1,0,\bullet}^\fd = C \cdot \mathcal{M}_{2}^\fd(\gamma^2-2) \times \QD_{1,1}^\fd
    $$
    where the right-hand side denotes the measure obtained by attaching the top marked point of a sample from $\mathcal{M}_{2}^\fd(\gamma^2-2)$ to the boundary marked point of an independent sample from $\QD_{1,1}^\fd$.
\end{lemma}
\begin{proof}
    This result follows from \cite[Lemma 4.1]{PPY92} and \cite[Proposition 3.11]{AHSY23}. The proof is similar to that of Lemma~\ref{lem:QA-resample}, so we omit it here.
\end{proof}

The following lemma is from Lemma 6.12 in~\cite{ASY22}.

\begin{lemma}\label{lem:QT(W,2,W)}
    For $W>0$, let $\Md_{2, \bullet}(W)$ be the measure obtained by adding a boundary marked point to the left boundary of a sample from $\Md_{2}(W)$ according to the quantum length measure. Then, we have $\Md_{2, \bullet}(W) = C \cdot {\rm QT}(W, 2, W)$ for some $C>0$.
\end{lemma}

An analogue result for $\Mfd_{2, \bullet}(W)$ is proved in Lemma 4.1 of~\cite{ASYZ24}.

\begin{lemma}\label{lem:QT(W,2,W)-forested}
    For $W>0$, let $\Mfd_{2, \bullet}(W)$ be the measure obtained by adding a boundary marked point to the left forested boundary of a sample from $\Mfd_{2}(W)$ according to the generalized quantum length measure. Then, we have $\Mfd_{2, \bullet}(W) = C \cdot {\rm QT}^\fd(W, \gamma^2-2, W)$ for some $C>0$.
\end{lemma}

\subsection{Reminder of existing relevant conformal welding results}
\label{subsec:conformal-welding-known}

In this section, we introduce the notion of conformal welding and review several relevant results. 

Let $\mathcal{M}^1$ and $\mathcal{M}^2$ be measures on quantum surfaces with boundary marked points. For $i=1,2$, we fix a boundary arc $e_i$ of a sample from $\mathcal{M}^i$ and define the measures $\{\mathcal{M}^i(\ell)\}_{\ell >0}$ through the disintegration $\mathcal{M}^i = \int_0^\infty \mathcal{M}^i(\ell )d\ell$, as described in Section~\ref{subsec:KPZ}. For $\ell>0$, given a pair of quantum surfaces sampled from the product measure $\mathcal{M}^1(\ell) \otimes \mathcal{M}^2(\ell)$, we can conformally weld them together according to the quantum length. This process yields a single quantum surface decorated with a curve, known as the welding interface. We denote the law of the resulting quantum surface as ${\rm Weld}(\mathcal{M}^1(\ell), \mathcal{M}^2(\ell))$, and let ${\rm Weld}(\mathcal{M}^1, \mathcal{M}^2) = \int_0^\infty {\rm Weld}(\mathcal{M}^1(\ell), \mathcal{M}^2(\ell)) d \ell$ be the \textit{conformal welding} of $\mathcal{M}^1$ and $\mathcal{M}^2$ along the boundary arcs $e_1$ and $e_2$.

In the case that there are no marked points on the boundary arcs $e_1$ and $e_2$, i.e., each $e_i$ is an entire boundary component of a sample from $\mathcal{M}^i$, we define ${\rm Weld}(\mathcal{M}^1(\ell), \mathcal{M}^2(\ell))$ as the curve-decorated quantum surface obtained as follows. First, sample a point $p_i \in e_i$ from the probability measure proportional to the quantum length measure on $e_i$ for $i \in \{1,2\}$. Next, we conformally welding $\mathcal{M}^1$ and $\mathcal{M}^2$ along $e_1$ and $e_2$ by identifying $p_1$ with $p_2$. We define ${\rm Weld}(\mathcal{M}^1, \mathcal{M}^2) = \int_0^\infty {\rm Weld}(\mathcal{M}^1(\ell), \mathcal{M}^2(\ell)) \ell d \ell$ as the \textit{uniform conformal welding} of $\mathcal{M}^1$ and $\mathcal{M}^2$ along the boundary components $e_1$ and $e_2$, where the additional factor $\ell$ accounts for the freedom of welding.

For measures on generalized quantum surfaces, we can also define the disintegration over the generalized quantum length of the forested boundary arcs. The operations of conformal welding and uniform conformal welding can be defined similarly as before. Recall that for $\gamma\in(\sqrt{2},2)$, we set $\kappa = \frac{16}{\gamma^2}$ and $\tilde \kappa = \gamma^2$. We first recall Theorem 1.4 in~\cite{AHSY23}.

\begin{lemma}\label{lem:weld-fd-disk}
     Let {$W_-,W_+>0$} and $\rho_\pm = \frac{4}{\gamma^2}(2-\gamma^2+W_\pm)$. Let $W = W_++W_-+2-\frac{\gamma^2}{2}$. Then for some constant $C>0$,
     $$
     \Mfd_2(W)\otimes\SLE_\kappa(\rho_-;\rho_+) = C \int_0^\infty \Wd(\Mfd_2(W_-;\ell),\Mfd_2(W_+;\ell))d\ell.
     $$
\end{lemma}

We also have the conformal welding result for the forested quantum triangle.

\begin{lemma}\label{lem:weld-fd-triangle}
   Let $W,W_1,W_2,W_3>0$ with $W_2+W_3 = W_1+\gamma^2-2$. Let $\rho_-=\frac{4}{\gamma^2}(W+2-\gamma^2)$, $\rho_+ = \frac{4}{\gamma^2}(W_2+2-\gamma^2)$, and $\rho_1 = \frac{4}{\gamma^2}(W_3+2-\gamma^2)$. Then for some constant $C>0$, we have
    \begin{equation}\label{eq:disk+QT-f}
        \QT^\fd(W+W_1+2-\frac{\gamma^2}{2}, W+W_2+2-\frac{\gamma^2}{2}, W_3)\otimes \SLE_{\kappa}(\rho_-;\rho_+,\rho_1) = C \int_0^\infty \Mfd_2(W;\ell)\times \QT^\fd(W_1,W_2,W_3;\ell)\, d\ell
    \end{equation}
    where the right forested boundary arc of a sample from $\Mfd_2(W;\ell)$ is conformally welded to the forested boundary arc between the $W_1$- and $W_2$-weighted vertices of a sample from $\QT^\fd(W_1,W_2,W_3;\ell)$ conditioned on having the same generalized boundary length.
\end{lemma}

\begin{proof}
    The lemma follows by combining \cite[Theorem 1.2]{ASY22} with Lemma~\ref{lem:weld-fd-disk}. The law of the interface can be identified using imaginary geometry~\cite{ig1}.
\end{proof}

Next, we recall the conformal welding result for radial SLE$_\kappa(\kappa - 6)$ from~\cite{ASYZ24}, which is a crucial input for Section~\ref{subsec:conformal-welding-CLE}. Let $\eta:= \eta^o$ be a radial $\SLE_\kappa(\kappa-6)$ in $\mathbb{D}$ from 1 targeted at 0 with force point $1 e^{i0^-}$. Let $\sigma_1$ be the first time that $\eta$ forms a loop around 0. Let $\tilde{\mathsf{m}}^\circlearrowleft$ (resp.\ $\tilde{\mathsf{m}}^\circlearrowright$) be the law of $\eta([0,\sigma_1])$ restricted to the event that the loop is counterclockwise (resp.\ clockwise). 

\begin{proposition}[Proposition 4.3 of~\cite{ASYZ24}]\label{prop:weld-wind} Let $\gamma \in (\sqrt2, 2)$. For some constant $C$ depending only on $\gamma$, we have
\begin{equation}
\begin{aligned}
&\QD_{1,1}^\fd \otimes \tilde{\mathsf{m}}^\circlearrowleft = C \iint_{\ell_1>\ell_2>0} {\rm Weld}( {\rm QT}^\fd (\frac{3\gamma^2}{2}-2,\gamma^2-2,2-\frac{\gamma^2}{2};\ell_1,\ell_2), \QD_{1,1}^\fd(\ell_1-\ell_2)) d\ell_1d\ell_2\,;\\
&\QD_{1,1}^\fd \otimes \tilde{\mathsf{m}}^\circlearrowright = C \iint_{\ell_2>\ell_1>0} {\rm Weld}( {\rm QT}^\fd (\frac{3\gamma^2}{2}-2,\gamma^2-2,2-\frac{\gamma^2}{2};\ell_1,\ell_2), \QD_{1,1}^\fd(\ell_2-\ell_1)) d\ell_1d\ell_2.
\end{aligned}
\end{equation}
The disintegration ${\rm QT}^\fd (\frac{3\gamma^2}{2}-2,\gamma^2-2,2-\frac{\gamma^2}{2};\ell_1,\ell_2)$ is supported on the set of forested quantum triangles with generalized quantum length $\ell_1$ (resp.\ $\ell_2$) from the weight $\frac{3\gamma^2}{2}-2$ vertex to the weight $\gamma^2-2$ vertex (resp.\ the weight $2-\frac{\gamma^2}{2}$ vertex). Here, for the forested quantum triangle, we conformally weld the two forested boundary arcs adjacent to the weight $\frac{3\gamma^2}{2}-2$ vertex, starting by identifying the weight $\gamma^2-2$ vertex with the weight $2-\frac{\gamma^2}2$ vertex, and conformally welding until the shorter boundary arc is completely welded to the longer boundary arc. Then, the remaining segment of the longer boundary arc is conformally welded to the forested quantum disk, identifying the weight $\frac{3\gamma^2}2-2$ vertex with the boundary marked point of the forested quantum disk.
\end{proposition}

Now we define the forested line segment and forested circle.

\begin{definition}
    Define $\Mfl$ as the law of the surface obtained by first sampling $T$ from ${\rm Leb}_{\mathbb{R}_+}$, then considering a line segment assigned with a quantum length $T$, and finally foresting this line segment as described in Definition~\ref{def:forested-disk}. We call a sample from $\Mfl$ a \emph{forested line segment}. For $\ell_1, \ell_2>0$, let $\Mfl(\ell_1, \ell_2)$ be the disintegration of $\Mfl$, which is a measure supported on the set of forested line segments with quantum length $\ell_1$ and generalized quantum length $\ell_2$.
\end{definition}

\begin{definition}
    We define a measure $\Mfr$ as follows. Sample $T$ from the distribution $t^{-1} 1_{t>0} dt$ and consider the unit circle $\cC$, which is assigned with a quantum length $T$. Then, forest the cicle $\cC$ as described in Definition~\ref{def:forested-disk}. Let $\Mfr$ be the law of the resulting beaded quantum surface. We call a sample from $\Mfr$ a \emph{forested circle}. For $\ell_1, \ell_2>0$, let $\Mfr(\ell_1, \ell_2)$ be disintegration of $\Mfr$, which is the measure supported on the set of forested circles with quantum length $\ell_1$ and generalized quantum length $\ell_2$.
\end{definition}

Using this definition, we can interpret foresting a boundary arc of a sample from a quantum surface measure $\mathcal{S}$ as performing the conformal welding for $\mathcal{S}$ and $\Mfl$. Similarly, foresting the entire boundary arc of a sample from a quantum surface measure $\mathcal{S}$ can be interpreted as performing the uniform conformal welding for $\mathcal{S}$ and $\Mfr$. For instance, we have
\begin{equation}\label{eq:weld-forest-circle}
\begin{aligned}
\Mfd_2(W) &= \int_{\mathbb{R}_+^4} {\rm Weld}(\Mfd_2(W;\ell_1, \ell_2), \Mfl(\ell_1, \ell_1'), \Mfl(\ell_2, \ell_2')) d \ell_1 d \ell_1' d\ell_2  d \ell_2';\\
\QD_{1,0}^\fd &= \iint_0^\infty {\rm Weld}(\QD_{1,0}(\ell_1), \Mfr(\ell_1, \ell_2)) \ell_1 d \ell_1 d\ell_2.
\end{aligned}
\end{equation}

\begin{lemma}\label{lem:GQD+LT}
    There exists a probability measure $\mathsf{l}_1$ on non-simple boundary-filling loops defined on a simply connected domain and a constant $C>0$ such that the following conformal welding holds: 
    $$
    {\rm QD}_{1,0} \otimes \mathsf{l}_1 = C \iint_0^\infty {\rm Weld}(\QD_{1,0}^\fd(\ell_2), \Mfr(\ell_1, \ell_2)) \ell_2 d \ell_1 d\ell_2.
    $$
    The right-hand side stands for the uniform conformal welding along the forested boundary arcs of $\QD_{1,0}^\fd$ and $\Mfr$, conditioned on having the same generalized quantum length.
\end{lemma}

\begin{proof}
    Let $\mu_1$ be the SLE$_\kappa$ loop measure on $\mathbb{C}$ restricted to those loops that separate $0$ and $\infty$, and let $\mu_2$ be the law of the outer boundary of a sample from $\mu_1$. It was shown in~\cite{ACSW24} that $\mu_2$ is equal to the SLE$_{\tilde \kappa}$ loop measure restricted to loops that separate $0$ and $\infty$ (modulo a multiplicative constant). Let $\mathsf{l}_1$ be the conditional law of a loop sampled from $\mu_1$ conditioned on its outer boundary. By~\cite[Proposition 6.15]{ACSW24}, we have ${\rm QS}_2 \otimes \mu_1 = C \int_0^\infty {\rm Weld}(\QD_{1,0}^\fd(\ell), \QD_{1,0}^\fd(\ell)) \ell d\ell$, where ${\rm QS}_2$ is a measure on the quantum sphere with two marked points (see Definition~\ref{def:QS}). By~\eqref{eq:weld-forest-circle}, we further have ${\rm QS}_2 \otimes \mu_2 \otimes \mathsf{l}_1 = C \iint_0^\infty {\rm Weld}(\QD_{1,0}(\ell_1), \Mfr(\ell_1, \ell_2), \QD_{1,0}^\fd(\ell_2)) \ell_1 \ell_2 d\ell_1 d\ell_2$. By Proposition 6.5 in~\cite{ACSW24}, we have ${\rm QS}_2 \otimes \mu_2 = C' \iint_0^\infty {\rm Weld}(\QD_{1,0}(\ell), \QD_{1,0}(\ell)) \ell d\ell$ for some constant $C'>0$. Comparing these two conformal welding results, we see that for a.e.\ all $\ell_1>0$, $\QD_{1,0}(\ell_1) \otimes \mathsf{l}_1 = \frac{C}{C'} \iint_0^\infty {\rm Weld}(\Mfr(\ell_1, \ell_2), \QD_{1,0}^\fd(\ell_2)) \ell_2 d\ell_2$. Integrating over $\ell_1$ yields the lemma.
\end{proof}

Recall from Definition~\ref{def:LF} the law ${\rm LF}_\H^{(\gamma, i), (0, 0)}$ on quantum surfaces with one bulk and one boundary marked points. We let ${\rm LF}_\H^{(\gamma, i), (0, 0), f}$ be the law on beaded quantum surfaces obtained by foresting the boundary arc of a sample from ${\rm LF}_\H^{(\gamma, i), (0, 0)}$. The following conformal welding result follows from \cite[Proposition 3.25]{AHSY23} combined with \cite[Theorem 1.1]{Wu23}.

\begin{lemma}\label{lem:weld-bubble}
    There exists a probability measure $\mathsf{l}_2$ on non-simple loops defined on a simply connected domain rooted at a boundary marked point surrounding a bulk marked point and a consant $C>0$ such that the following conformal welding holds: 
    $$
    {\rm LF}_\H^{(\gamma, i), (0, 0), f} \otimes \mathsf{l}_2 = C \int_0^\infty {\rm Weld}(\QD_{1,1}^\fd(\ell), \Mfd_2(\gamma^2-2; \ell)) d \ell.
    $$
    On the right-hand side, we conformally weld the forested boundary arc of $\QD_{1,1}^\fd$ to the left forested boundary arc of $\Mfd_2(\gamma^2-2)$, conditioned on having the same generalized quantum length.
\end{lemma}
\begin{proof}
    By~\eqref{eq:weld-forest-circle}, the conformal welding on the right-hand side is equal to
    $$
    \int_{\mathbb{R}_+^5} {\rm Weld}(\QD_{1,1}(\ell_1), \Mfl(\ell_1, \ell_2), \Mfl(\ell_3,\ell_2), \Md_2(\gamma^2-2; \ell_3, \ell_4), \Mfl(\ell_4, \ell_5)) d \ell_1 d\ell_2 d\ell_3 d\ell_4 d\ell_5
    $$
    as we can view foresting the boundary arcs of samples from $\QD^\fd_{1,1}$ or $\Mfd_2(\gamma^2-2)$ as conformally welding them with samples from the forested line segment. Then, by \cite[Proposition 3.25]{AHSY23}, we can first conformally weld samples from $\Mfl(\ell_1, \ell_2)$ and $\Mfl(\ell_3,\ell_2)$ and integrate over $\ell_2$, which yields $\Md_2(2 - \frac{\gamma^2}{2}; \ell_1, \ell_3)$ decorated with an independent SLE$_{\kappa}(\frac{\kappa}{2} - 4; \frac{\kappa}{2} - 4)$ curve. Then, we apply~\cite[Theorem 2.2]{AHS20} to conformally weld samples from $\Md_2(2 - \frac{\gamma^2}{2}; \ell_1, \ell_3)$ and $\Md_2(\gamma^2-2; \ell_3, \ell_4)$ and integrate over $\ell_3$, which yields $\Md_2(\frac{\gamma^2}{2}; \ell_1, \ell_4)$ decorated with an independent SLE$_{\tilde \kappa}( - \frac{\tilde \kappa}{2}; \tilde \kappa - 4)$. Next, we apply~\cite[Theorem 1.1]{Wu23} to conformally weld samples from $\QD_{1,1}(\ell_1)$ and $\Md_2(\frac{\gamma^2}{2}; \ell_1, \ell_4)$, which yields ${\rm LF}_\H^{(\gamma, i), (0, 0)}(\ell_4)$ decorated with an independent curve sampled from the SLE$_{\tilde \kappa}(\frac{\tilde \kappa}{2} - 2)$ bubble measure conditioned on surrounding the bulk marked point. Finally, applying~\eqref{eq:weld-forest-circle} yields the desired result. The law of $\mathsf{l}_2$ can be explicitly written out by taking all these independent SLE curves into account.
\end{proof}

\subsection{Proof of the key conformal welding  results}
\label{subsec:conformal-welding-CLE}

In this section, we  prove several conformal welding results for CLE loops needed for the proof of Theorem~\ref{thm:cardy-formula} and~\ref{thm:backbone-crossing}. We refer to Section~\ref{subsec:cle6} for the construction of CLE$_\kappa$ via the exploration tree in~\cite{Sheffield09}.
We first discuss the conformal welding result concerning the outermost loop in CLE$_\kappa$ that surrounds a given point. For a CLE$_\kappa$ $\Gamma$ on $\mathbb{D}$, let $\Loop$ be the outermost loop in $\Gamma$ that surrounds $0$, and $\mathsf{m}$ be the law of the inner boundary of $\Loop$, which is the boundary of the connected component of $\mathbb{D} \setminus \Loop$ containing $0$. Let $T = \{ \Loop \cap \partial \mathbb{D} \neq \emptyset \}$ and $\mathsf{m}_T$ be the law of the inner boundary of $\Loop$ restricted to the event $T$.

\begin{proposition}\label{prop:conformal-welding-cardy}
    There exist measures $\mathcal{QA}$ and $\mathcal{QA}_T$ on the space of quantum surfaces with annular topology such that:
    \begin{equation}\label{eq:conformal-welding-cardy}
    \begin{split}
     \QD_{1,0} \otimes \mathsf{m} &= \int_0^\infty {\rm Weld}(\mathcal{QA}(\ell), \QD_{1,0}(\ell)) \ell d\ell \,;\\
     \QD_{1,0} \otimes \mathsf{m}_T &= \int_0^\infty {\rm Weld}(\mathcal{QA}_T(\ell), \QD_{1,0}(\ell)) \ell d\ell.
    \end{split}
    \end{equation}
    The right-hand side represents the uniform conformal welding along the inner boundary of a sample from $\mathcal{QA}$ (resp.\ $\mathcal{QA}_T$)  and the entire boundary component of a sample from $\QD_{1,0}$ conditioned on having the same quantum length. 
\end{proposition}

This result follows essentially from Proposition~\ref{prop:weld-wind}. Recall the setting of Proposition~\ref{prop:weld-wind}. By \cite[Proposition 2.2]{ASYZ24}, the event $T$ occurs if and only if $\eta([0,\sigma_1])$ forms a counterclockwise loop. On this event, let $z$ be the left-most intersection point of $\eta([0,\sigma_1]) \cap \partial \mathbb{D}$. Let $\theta$ be the last time $\eta$ visits $z$. Draw an independent chordal SLE$_\kappa$ curve $\eta^\circlearrowleft$ in the connected component of $\mathbb{D} \setminus \eta([0, \sigma_1])$ containing $z$, from $\eta(\sigma_1)$ to $z$. Let $\mu^{\circlearrowleft}$ be the law of the concatenation of $\eta([\theta, \sigma_1]) \cup \eta^\circlearrowleft$. See Figure~\ref{fig:event-T} (left). Then the loop $\Loop$, restricted to the event $T$, agrees in law with $\mu^\circlearrowleft$. Moreover, by~\cite{Sheffield09, ig2, ig3}, $\mu^{\circlearrowleft}$ is independent of the starting point 1. 

If $\eta([0, \sigma_1])$ forms a clockwise loop, we define a loop measure $\mu^{\circlearrowright}$ analogously. Specifically, let $z$ be as defined before. Let $z'$ be the right-most intersection point of $\eta([0,\sigma_1]) \cap \partial \mathbb{D}$ and $\theta'$ be the last time $\eta$ visits $z'$. Draw an independent chordal SLE$_\kappa(\kappa - 6)$ curve $\eta^\circlearrowright$ in the connected component of $\mathbb{D} \setminus \eta([0, \sigma_1])$ containing $z'$, from $\eta(\sigma_1)$ to $z
'$, with the force point at $z$. Let $\mu^{\circlearrowright}$ be the law of the concatenation of $\eta([\theta', \sigma_1]) \cup \eta^\circlearrowright$. See Figure~\ref{fig:event-T} (right). Similar arguments to those in \cite[Proposition 5.1]{Sheffield09} and~\cite{ig2, ig3} show that the law $\mu^{\circlearrowright}$ is independent of the starting point 1. 

\begin{figure}[h]
\centering
\includegraphics[scale = 0.8]{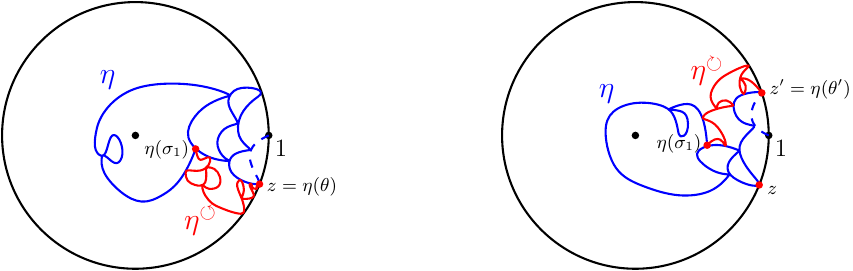}
\caption{\textbf{Left:} $\eta([0,\sigma_1])$ forms a counterclockwise loop, and the event $T$ occurs. $\mu^\circlearrowleft$ is the law of the loop formed by the solid blue and red curves, which is equivalent to the law of $\Loop$ restricted to the event $T$. \textbf{Right:} $\eta([0,\sigma_1])$ forms a clockwise loop, and the event $T$ does not occur. $\mu^\circlearrowright$ is the law of the loop formed by the solid blue and red curves.}
\label{fig:event-T}
\end{figure}

We now show the conformal welding results for $\mu^\circlearrowleft$ and $\mu^\circlearrowright$, from which Proposition~\ref{prop:conformal-welding-cardy} follows.

\begin{lemma}\label{lem:radial-sle-counterclock}
    For some constants $C, C'>0$, we have
    \begin{align}
   &\QD_{1,0}^\fd \otimes \mu^{\circlearrowleft} =  C \int_0^\infty {\rm Weld}(\widetilde{\QA}^\fd(\gamma^2-2;\ell),\QD_{1,0}^\fd(\ell))\ell 
   d\ell\,;\label{eq:welding-lem4.17-1}\\
   &\QD_{1,0}^\fd \otimes \mu^{\circlearrowright} =  C' \int_0^\infty {\rm Weld}(\widetilde{\QA}^\fd(2 - \frac{\gamma^2}{2};\ell),\QD_{1,0}^\fd(\ell))\ell 
   d\ell.\label{eq:welding-lem4.17-2}
   \end{align}
   The right-hand sides represent the uniform conformal welding along the inner forested boundary of a sample from $\widetilde{\QA}^\fd(\gamma^2-2)$ or $\widetilde{\QA}^\fd(2 - \frac{\gamma^2}{2})$ and the entire forested boundary arc of a sample from $\QD_{1,0}^\fd$ conditioned on having the same generalized quantum length. 
\end{lemma}

\begin{proof}
We will only present the proof for $\mu^\circlearrowleft$, and the case for $\mu^\circlearrowright$ follows from the same argument. In this proof, the constant $C$ may change from line to line but only depends on $\gamma$. By Definition~\ref{def:thin-triangle}, we have ${\rm QT}^\fd (\frac{3\gamma^2}{2}-2,\gamma^2-2,2 - \frac{\gamma^2}{2}) = C \Mfd_2(2 - \frac{\gamma^2}{2}) \times {\rm QT}^\fd (\frac{3\gamma^2}{2}-2,\gamma^2-2,\frac{3\gamma^2}{2} - 2)$ for some $C>0$. Thus, the first welding equation in Proposition~\ref{prop:weld-wind} implies that
$$
\QD_{1,1}^\fd \otimes \tilde{\mathsf{m}}^\circlearrowleft = C \iiint_{\ell_1>\ell_2 + \ell_3} {\rm Weld}( {\rm QT}^\fd (\frac{3\gamma^2}{2}-2,\gamma^2-2,\frac{3\gamma^2}{2} - 2;\ell_1,\ell_2), \Mfd_2(2 - \frac{\gamma^2}{2};\ell_3), \QD_{1,1}^\fd(\ell_1-\ell_2-\ell_3)) d\ell_1d\ell_2 d\ell_3.
$$
By drawing the independent SLE$_\kappa$ $\eta^\circlearrowleft$ as defined earlier on both sides (on the right-hand side, this corresponds to drawing an SLE$_\kappa$ on top of a sample from ${\rm QT}^\fd (\frac{3\gamma^2}{2}-2,\gamma^2-2,\frac{3\gamma^2}{2} - 2)$ between the two $(\frac{3\gamma^2}{2}-2)$-weighted vertices) and applying Lemma~\ref{lem:weld-fd-triangle} with $W = W_1 = W_2 = W_3 = \gamma^2-2$, we further have
\begin{equation}\label{eq:weld-lem4.17-1}
\begin{aligned}
\QD_{1,1}^\fd \otimes  \tilde{\mathsf{m}}^\circlearrowleft \otimes \mbox{SLE$_\kappa$}(d\eta^\circlearrowleft) &= C \iiiint_{\ell_1>\ell_2 + \ell_3} {\rm Weld}\Big( {\rm QT}^\fd (\gamma^2-2,\gamma^2-2,\gamma^2-2;\ell_1,\ell_4), \Mfd_2(2 - \frac{\gamma^2}{2};\ell_4, \ell_2), \\
&\qquad \qquad \qquad\quad \qquad\qquad\qquad \Mfd_2(2 - \frac{\gamma^2}{2};\ell_3), \QD_{1,1}^\fd(\ell_1-\ell_2-\ell_3)\Big) d\ell_1d\ell_2 d\ell_3 d \ell_4.
\end{aligned}
\end{equation}
For the conformal welding on the right-hand side of~\eqref{eq:weld-lem4.17-1}, we first attach $\Mfd_2(2 - \frac{\gamma^2}{2};\ell_4, \ell_2)$ with $\QD_{1,1}^\fd(\ell_1-\ell_2-\ell_3)$. By Lemma~\ref{lem:QD-forested-resample}, for fixed $\ell_1>\ell_3$ and $\ell_4>0$, we have $\int_0^{\ell_1 - \ell_3} \Mfd_2(2 - \frac{\gamma^2}{2};\ell_4, \ell_2) \times \QD_{1,1}^\fd(\ell_1-\ell_2-\ell_3) d\ell_2 = C \QD_{1,0,\bullet}^\fd(\ell_4, \ell_1 - \ell_3)$. Hence, the right-hand side of~\eqref{eq:weld-lem4.17-1} becomes
$$
C \iiint_{\ell_1> \ell_3} {\rm Weld}( {\rm QT}^\fd (\gamma^2-2,\gamma^2-2,\gamma^2-2;\ell_1,\ell_4), \Mfd_2(2 - \frac{\gamma^2}{2};\ell_3), \QD_{1,0,\bullet}^\fd(\ell_4, \ell_1 - \ell_3))d\ell_1d\ell_3 d \ell_4.
$$
Next, we change the variable $\ell_1' = \ell_1 + \ell_4$ and integrate $\ell_4$ from 0 to $\ell_1' - \ell_3$ for fixed $\ell_1' > \ell_3$. By Lemma~\ref{lem:QT(W,2,W)-forested}, ${\rm QT}^\fd (\gamma^2-2,\gamma^2-2,\gamma^2-2) = C \Mfd_{2, \bullet}(\gamma^2-2)$, so we can view ${\rm QT}^\fd (\gamma^2-2,\gamma^2-2,\gamma^2-2)$ as $\Mfd_2(\gamma^2-2)$ with an additional boundary marked point on the left boundary, which corresponds to $\eta(\sigma_1)$. Furthermore, integrating $\ell_4$ from 0 to $\ell_1' - \ell_3$ corresponds to forgetting this additional marked point. Therefore, we have
\begin{equation}\label{eq:weld-lem4.17-2}
\QD_{1,1}^\fd \otimes  \tilde{\mathsf{m}}^\circlearrowleft \otimes \mbox{SLE$_\kappa$}(d\eta^\circlearrowleft) = C \iint_{\ell_1'> \ell_3} {\rm Weld}( \Mfd_{2}(\gamma^2-2; \ell_1'), \Mfd_2(2 - \frac{\gamma^2}{2};\ell_3), \QD_{1,0,\bullet}^\fd(\ell_1' - \ell_3))d\ell_1' d\ell_3.
\end{equation}
Now we mark the point on the left boundary of $\Mfd_2(2 - \frac{\gamma^2}{2}; \ell_1')$ that has distance $\ell_3$ to the top vertex and view a sample from $\Mfd_2(2 - \frac{\gamma^2}{2})$ with this additional boundary marked point as a sample from ${\rm QT}^\fd (\gamma^2-2,\gamma^2-2,\gamma^2-2)$. We further change the variable $\ell_1'' = \ell_1' - \ell_3$, apply Lemma~\ref{lem:weld-fd-triangle} with $W = 2 - \frac{\gamma^2}{2}$ and $W_1 = W_2 = W_3 = \gamma^2-2$, and forget the curve $\eta([0, \theta])$. This yields that
\begin{equation}\label{eq:weld-lem4.17-3}
\QD_{1,1}^\fd \otimes \mu^\circlearrowleft =  C \int_0^\infty {\rm Weld}( {\rm QT}^\fd (2,2,\gamma^2-2; \ell_1''), \QD_{1,0,\bullet}^\fd(\ell_1''))d\ell_1''.
\end{equation}
Then, we forget the boundary marked point on $\QD_{1,0,\bullet}^\fd(\ell_1'')$, which turns the conformal welding on the right-hand side of~\eqref{eq:weld-lem4.17-3} to a uniform conformal welding. Therefore,
\begin{equation}\label{eq:weld-lem4.17-4}
\QD_{1,1}^\fd \otimes \mu^\circlearrowleft =  C \int_0^\infty {\rm Weld}( {\rm QT}^\fd (2,2,\gamma^2-2; \ell_1''), \QD_{1,0}^\fd(\ell_1'')) \ell_1'' d\ell_1''.
\end{equation}
By Definition~\ref{def:thin-triangle} and Lemma~\ref{lem:QT(W,2,W)}, we have ${\rm QT} (2,2,\gamma^2-2) = C {\rm QT}(2,2,2) \times \Md_2(\gamma^2-2) = C' \Md_{2, \bullet}(2) \times \Md_2(\gamma^2-2)$. Combined with Lemma~\ref{lem:QA-resample}, we see that ${\rm QT} (2,2,\gamma^2-2) = \wt{\QA}_\bullet(\gamma^2 -2)$. Hence, we can view ${\rm QT}^\fd(2,2,\gamma^2-2)$ as $\wt{\QA}^\fd(\gamma^2 -2)$ with an additional boundary marked point on the unforested outer boundary. Note that the law of $\mu^\circlearrowleft$ does not depend on the starting point 1, which corresponds to this boundary marked point. By forgetting this marked point on both sides of~\eqref{eq:weld-lem4.17-4}, we arrive at the lemma.
\end{proof}

Next, we complete the proof of Proposition~\ref{prop:conformal-welding-cardy} using Lemma~\ref{lem:radial-sle-counterclock}.

\begin{proof}[Proof of Proposition~\ref{prop:conformal-welding-cardy}]
By Lemma~\ref{lem:radial-sle-counterclock} and~\eqref{eq:weld-forest-circle}, we have
\begin{equation}\label{eq:proof-prop4.17-1}
\QD_{1,0}^\fd \otimes \mu^{\circlearrowleft} =  C \iint_{\mathbb{R}_+^2} {\rm Weld}(\widetilde{\QA}^\fd(\gamma^2-2;\ell_2),\Mfr(\ell_1, \ell_2), \QD_{1,0}(\ell_1))\ell_1 \ell_2 
   d\ell_1 d\ell_2
\end{equation}
where the right-hand side stands for two uniform conformal weldings. By~\eqref{eq:weld-forest-circle}, we can view $\widetilde{\QA}^\fd(\gamma^2-2)$ as the measure on generalized quantum surfaces obtained by performing uniform conformal weldings of a sample from $\widetilde{\QA}(\gamma^2-2)$ with two independent samples from $\Mfr$ conditioned on having the same quantum boundary lengths. Then, by forgetting the looptrees attached to the outer boundary of both sides of~\eqref{eq:proof-prop4.17-1}, we obtain
\begin{equation}\label{eq:proof-prop4.17-2}
\QD_{1,0} \otimes \mu^{\circlearrowleft} =  C \iiint_{\mathbb{R}_+^3} {\rm Weld}(\widetilde{\QA}(\gamma^2-2;\ell_3), \Mfr(\ell_3, \ell_2), \Mfr(\ell_1, \ell_2), \QD_{1,0}(\ell_1))\ell_1 \ell_2 \ell_3 
   d\ell_1 d\ell_2 d\ell_3.
\end{equation}
Recall that $\mathsf{m}_T$ is the law of the inner boundary of a sample from $\mu^\circlearrowleft$. Hence, the conformal welding for $\mathsf{m}_T$ in Proposition~\ref{prop:conformal-welding-cardy} follows by taking $\mathcal{QA}_T = C \iiint_{\mathbb{R}_+^3} {\rm Weld}(\widetilde{\QA}(\gamma^2-2;\ell_3), \Mfr(\ell_3, \ell_2), \Mfr(\ell_1, \ell_2)) \ell_2 \ell_3 d\ell_1 d\ell_2 d\ell_3$ and forgetting the interfaces between them.

The conformal welding for $\mathsf{m}$ in Proposition~\ref{prop:conformal-welding-cardy} also follows from Lemma~\ref{lem:radial-sle-counterclock} since $\mathsf{m}$ can be represented using the iteration of $\mu^\circlearrowleft$ and $\mu^\circlearrowright$, as we now elaborate. Recall the setting of Lemma~\ref{lem:radial-sle-counterclock}. On the event that $\eta([0, \sigma_1])$ forms a counterclockwise loop, $\Loop$ has the same law as $\mu^\circlearrowleft$. On the event that $\eta([0, \sigma_1])$ forms a clockwise loop, let $D$ be the simply connected component of $\mathbb{D} \setminus \eta([0, \sigma_1])$ containing 0. Then the boundary of $D$ has the same law as the inner boundary of a sample from $\mu^\circlearrowright$. Moreover, by the Markov property of radial SLE$_\kappa(\kappa-6)$ and CLE$_\kappa$ as in \cite[Proposition 2.3]{MSWCLEgasket}, on this event, $\Loop$ has the same law as the outermost CLE$_\kappa$ loop surrounding 0 sampled in $D$. Therefore, the law of $\Loop$ can be obtained by iteratively sampling from the combined probability measure $\mu = \mu^\circlearrowleft + \mu^\circlearrowright$ until a sample from $\mu^\circlearrowright$ is obtained. This procedure also yields $\mathsf{m}$, which is the law of the inner boundary of $\Loop$. By Lemma~\ref{lem:radial-sle-counterclock}, we have an analog of~\eqref{eq:proof-prop4.17-2} for $\mu^{\circlearrowright}$ with $\widetilde{\QA}(2 - \frac{\gamma^2}{2})$ in place of $\widetilde{\QA}(\gamma^2 - 2)$. Combining these results yields the conformal welding for $\mathsf{m}$ in Proposition~\ref{prop:conformal-welding-cardy}. The measure of $\mathcal{QA}$ can be represented using the uniform conformal welding of $\widetilde{\QA}(\gamma^2-2), \widetilde{\QA}(2 - \frac{\gamma^2}{2}),$ and $\Mfr$. \qedhere
\end{proof}

Next, we present the conformal welding results for the outermost loop in CLE$_\kappa$ whose outer boundary surrounds a given point. For a CLE$_\kappa$ $\Gamma$ on $\mathbb{D}$, let $\mathcal{L}^*$ be the outermost loop in $\Gamma$ whose outer boundary surrounds $0$, and let $\ep{\mathcal{L}^*}$ be the outer boundary of $\mathcal{L}^*$. Let $T^* = \{ \mathcal{L}^* \cap \partial \mathbb{D} \neq \emptyset \}$, and $\mathsf{n}_{T^*}$ (resp.\ $\mathsf{n}_{(T^*)^c}$) be the law of $\ep{\mathcal{L}^*}$ restricted to the event $T^*$ (resp.\ $(T^*)^c$).

\begin{proposition}\label{prop:weld-backbone}
    There exist a constant $C>0$ and a measure $\mathcal{QA}^*$ on the space of quantum surfaces with annular topology such that:
    \begin{align}
     \QD_{1,0} \otimes \mathsf{n}_{T^*} &= C \int_0^\infty {\rm Weld}(\wt{\QA}(\gamma^2-2) (\ell), \QD_{1,0}(\ell)) \ell d\ell \,;\label{eq:conformal-welding-backbone-1}\\
     \QD_{1,0} \otimes \mathsf{n}_{(T^*)^c} &= \int_0^\infty {\rm Weld}(\mathcal{QA}^*(\ell), \QD_{1,0}(\ell)) \ell d\ell.\label{eq:conformal-welding-backbone-2}
    \end{align}
    The right-hand sides stand for the uniform conformal welding along the inner boundary of a sample from $\wt{\QA}(\gamma^2-2)$ or $\mathcal{QA}^*$ and the entire boundary arc of a sample from $\QD_{1,0}$ conditioned on having the same quantum length. 
\end{proposition}

Proposition~\ref{prop:weld-backbone} also follows from Proposition~\ref{prop:weld-wind} combined with several other conformal welding results in Section~\ref{subsec:conformal-welding-known}. Recall the setting of Proposition~\ref{prop:weld-wind}: $\eta$ is a radial SLE$_\kappa(\kappa-6)$ in $\mathbb{D}$ from 1 to 0, and let $\sigma_1$ be the first time $\eta$ forms a loop around 0. To determine if the event $T^*$ occurs, there are three cases to consider:
\begin{enumerate}[(1)]
    \item $\eta([0,\sigma_1])$ forms a counterclockwise loop; see Figure~\ref{fig:event-T} (left). In this case, we have $\Loop = \mathcal{L}^*$ and both the events $T$ and $T^*$ occur. We denote by $\mathsf{n}_T$ the law of $\ep{\mathcal{L}^*}$ restricted to the event $T$. Then, using Lemmas~\ref{lem:GQD+LT} and \ref{lem:radial-sle-counterclock}, we can deduce that~\eqref{eq:conformal-welding-backbone-1} holds with $\mathsf{n}_T$ in place of $\mathsf{n}_{T^*}$.

    \item $\eta([0,\sigma_1])$ forms a clockwise loop. In this case, the event $T$ does not occur, but whether the event $T^*$ occurs depends on the behavior of $\eta$. Specifically, let $\tau$ be the time before $\sigma_1$ that $\eta$ visits $\eta(\sigma_1)$, which is almost surely unique. There are two subcases:

    \begin{enumerate}[(i)]
        \item $\eta([\tau, \sigma_1])$ does not hit $\partial \mathbb{D}$; see Figure~\ref{fig:event-Tout} (left). In this case, $T^*$ occurs. Let $z$ be the left-most intersection point of $\eta([0,\tau]) \cap \partial \mathbb{D}$, and let $\theta$ be the last time $\eta$ visits $z$. Draw an independent chordal SLE$_\kappa$ curve $\eta^*$ in the connected component of $\mathbb{D} \setminus \eta([0, \sigma_1])$ containing $z$, from $\eta(\sigma_1)$ to $z$. See Figure~\ref{fig:event-Tout} (left). Let $\mu^*$ be the law of the concatenation of $\eta([\theta, \sigma_1]) \cup \eta^*$, and let $\mathsf{n}_{T^* \setminus T}$ be the law of the outer boundary of a sample from $\mu^*$. In the next subcase, the event $T^*$ does not occur. Therefore, $\mu^*$ agrees with the law of $\mathcal{L}^*$ restricted to the event $T^* \setminus T$, and we also have $\mathsf{n}_{T^*} = \mathsf{n}_T + \mathsf{n}_{T^* \setminus T}$. \label{case:outermost}

        \item $\eta([\tau, \sigma_1])$ hits $\partial \mathbb{D}$; see Figure~\ref{fig:event-Tout} (right). In this case, the event $T^*$ does not occur. We will perform the same operation as in Figure~\ref{fig:event-T} (right). Specifically, let $z$ be as defined before. Let $z'$ be the right-most intersection point of $\eta([0,\sigma_1]) \cap \partial \mathbb{D}$, and let $\theta'$ be the last time $\eta$ visits $z'$. Draw an independent chordal SLE$_\kappa(\kappa - 6)$ curve $\eta^\circlearrowright$ in the connected component of $\mathbb{D} \setminus \eta([0, \sigma_1])$ containing $z'$, from $\eta(\sigma_1)$ to $z'$, with the force point at $z$. Let $\mu^{\circlearrowright,*}$ be the law of the concatenation of $\eta([\theta', \sigma_1]) \cup \eta^\circlearrowright$ on this event, and let $\mathsf{m}^*$ be the law of the inner boundary of a sample from $\mu^{\circlearrowright,*}$.
    \end{enumerate}
    
\end{enumerate}

\begin{figure}[h]
\centering
\includegraphics[scale = 0.8]{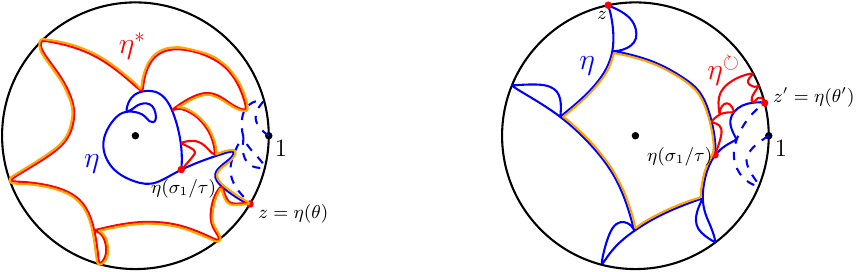}
\caption{\textbf{Left:} $\eta([0, \sigma_1])$ forms a clockwise loop, and $\eta([\tau, \sigma_1])$ does not intersect $\partial \mathbb{D}$. In this case, the event $T^* \setminus T$ occurs. $\mu^*$ is the law of the loop formed by the solid blue and red curves, which is equivalent to the law of $\mathcal{L}^*$ restricted to the event $T^* \setminus T$. The measure $\mathsf{n}_{T^* \setminus T}$ is the law of the orange loop, which will be the welding interface in Lemma~\ref{lem:weld4.21}. \textbf{Right:} $\eta([0, \sigma_1])$ forms a clockwise loop, and $\eta([\tau, \sigma_1])$ intersects $\partial \mathbb{D}$. In this case, the event $T^*$ does not occur. $\mu^{\circlearrowright,*}$ is the law of the loop formed by the solid blue and red curves, and $\mathsf{m}^*$ is the law of the orange loop, which will be the welding interface in Lemma~\ref{lem:weld4.23}.}
\label{fig:event-Tout}
\end{figure}

We first show the conformal welding result for $\mathsf{n}_{T^* \setminus T}$.
\begin{lemma}\label{lem:weld4.21}
    For some constant $C>0$, we have
   \begin{equation}\label{eq:weld-lem4.21}
   \QD_{1,0} \otimes \mathsf{n}_{T^* \setminus T} =  C \int_0^\infty {\rm Weld}(\widetilde{\QA}(\gamma^2-2;\ell),\QD_{1,0}(\ell))\ell 
   d\ell.
   \end{equation}
\end{lemma}

\begin{proof}
    The proof is similar to that of Lemma~\ref{lem:radial-sle-counterclock}. In this proof, the constant $C$ may change from line to line but only depend on $\gamma$. By Definition~\ref{def:thin-triangle}, we have ${\rm QT}^\fd (\frac{3\gamma^2}{2}-2,\gamma^2-2,2 - \frac{\gamma^2}{2}) = C \Md_2(2 - \frac{\gamma^2}{2}) \times {\rm QT}^\fd (\frac{3\gamma^2}{2}-2,\gamma^2-2,\frac{3\gamma^2}{2} - 2)$ for some $C>0$. Thus, the second welding equation in Proposition~\ref{prop:weld-wind} implies that
$$
\QD_{1,1}^\fd \otimes \tilde{\mathsf{m}}^\circlearrowright = C \iiint_{\ell_2 + \ell_3 > \ell_1} {\rm Weld}( {\rm QT}^\fd (\frac{3\gamma^2}{2}-2,\gamma^2-2,\frac{3\gamma^2}{2} - 2;\ell_1,\ell_2), \Mfd_2(2 - \frac{\gamma^2}{2};\ell_3), \QD_{1,1}^\fd(\ell_2 + \ell_3 - \ell_1)) d\ell_1d\ell_2 d\ell_3.
$$
Now we restrict both sides to the event $T^* \setminus T$. We claim that on the right-hand side, this event corresponds to $\ell_1 > \ell_3$. This is due to the fact that on the event $T^* \setminus T$ (namely, the case~\eqref{case:outermost}), $\eta(\tau) = \eta(\sigma_1)$ corresponds to the first $(\frac{3\gamma^2}{2}-2)$-weighted vertex in ${\rm QT}^\fd (\frac{3\gamma^2}{2}-2,\gamma^2-2,\frac{3\gamma^2}{2} - 2)$ as well as a point on the forested boundary arc between two $(\frac{3\gamma^2}{2}-2)$-weighted vertices with generalized quantum distance $\ell_2 + \ell_3 - \ell_1$ to the first $(\frac{3\gamma^2}{2}-2)$-weighted vertex. Moreover, $\eta([\tau, \sigma_1])$ corresponds to the forested boundary arc from this point to the first $(\frac{3\gamma^2}{2}-2)$-weighted vertex which does not intersect $\partial \mathbb{D}$. Hence, the event $T^* \setminus T$ happens if and only if this point exists, which is equivalent to $\ell_2 + \ell_3 - \ell_1 < \ell_2$, namely $\ell_1 > \ell_3$. Therefore, we have
\begin{align*}
\QD_{1,1}^\fd \otimes \tilde{\mathsf{m}}^\circlearrowright 1_{T^* \setminus T} &= C \iiint_{\ell_2 + \ell_3 > \ell_1, \ell_1 > \ell_3} {\rm Weld}( {\rm QT}^\fd (\frac{3\gamma^2}{2}-2,\gamma^2-2,\frac{3\gamma^2}{2} - 2;\ell_1,\ell_2), \\
&\qquad \qquad \qquad \qquad \qquad \qquad \qquad \Mfd_2(2 - \frac{\gamma^2}{2};\ell_3), \QD_{1,1}^\fd(\ell_2 + \ell_3 - \ell_1)) d\ell_1d\ell_2 d\ell_3.
\end{align*}
Then, we draw the independent chordal SLE$_\kappa$ curve $\eta^*$ as defined in case~\eqref{case:outermost} (on the right-hand side, this corresponds to drawing an SLE$_\kappa$ on top of a sample from ${\rm QT}^\fd (\frac{3\gamma^2}{2}-2,\gamma^2-2,\frac{3\gamma^2}{2} - 2)$ between the two $(\frac{3\gamma^2}{2}-2)$-weighted vertices). Next, we apply Lemma~\ref{lem:weld-fd-triangle} with $W = W_1 = W_2 = W_3 = \gamma^2-2$. This yields that
\begin{equation}\label{eq:weld-lem4.22-1}
\begin{aligned}
\QD_{1,1}^\fd \otimes  \tilde{\mathsf{m}}^\circlearrowright 1_{T^* \setminus T} \otimes {\rm SLE}_\kappa(d\eta^*) &= C \iiiint_{\ell_2 + \ell_3 > \ell_1 > \ell_3} {\rm Weld}\Big( {\rm QT}^\fd (\gamma^2-2,\gamma^2-2,\gamma^2-2;\ell_1,\ell_4), \\
&\qquad \Mfd_2(2 - \frac{\gamma^2}{2};\ell_4, \ell_2), \Mfd_2(2 - \frac{\gamma^2}{2};\ell_3), \QD_{1,1}^\fd(\ell_2 + \ell_3 - \ell_1)\Big) d\ell_1d\ell_2 d\ell_3 d \ell_4.
\end{aligned}
\end{equation}
Now we add a boundary marked point to $\Mfd_2(2 - \frac{\gamma^2}{2}; \ell_4, \ell_2)$ on the right boundary with generalized quantum length $\ell_2 + \ell_3 - \ell_1 < \ell_2$ to the top vertex, which yields the quantum surface $\Mfd_{2, \bullet}(2 - \frac{\gamma^2}{2};\ell_4, \ell_2 + \ell_3 - \ell_1, \ell_1 - \ell_3)$. By Lemma~\ref{lem:QT(W,2,W)-forested}, we can identify it in another way: we can first sample from $\Mfd_2(2 - \frac{\gamma^2}{2};\ell_4 + \ell_1 - \ell_3, \ell_2 + \ell_3 - \ell_1)$ and then add a boundary marked point on the left boundary with generalized quantum length $\ell_4$ to the top vertex. Combined with Lemma~\ref{lem:weld-bubble}, we see that for fixed $\ell_1 > \ell_3$ and $\ell_4>0$, $\int_{\ell_1 - \ell_3}^\infty {\rm Weld}(\Mfd_{2, \bullet}(2 - \frac{\gamma^2}{2};\ell_4, \ell_2 + \ell_3 - \ell_1, \ell_1 - \ell_3) ,\QD_{1,1}^\fd(\ell_2 + \ell_3 -\ell_1)) d\ell_2 = C {\rm LF}_{\H, \bullet}^{(\gamma, i), (0, 0), f}(\ell_4, \ell_1 - \ell_3)$, which is equivalent to $C\QD_{1,0,\bullet}^\fd(\ell_4, \ell_1 - \ell_3)$ by~\cite[Theorem 3.4]{ARS21}. We forget the welding interface in this conformal welding on both sides of~\eqref{eq:weld-lem4.22-1} (which corresponds to $\eta([\tau, \sigma_1])$ on the left-hand side), and further forget the marked point. This yields that
\begin{align*}
&\quad \QD_{1,1}^\fd \otimes  \widetilde{\mathsf{m}} \otimes {\rm SLE}_\kappa(d\eta^*) \\
&= C \iiint_{\ell_1> \ell_3} {\rm Weld}( {\rm QT}^\fd (\gamma^2-2,\gamma^2-2,\gamma^2-2;\ell_1,\ell_4), \Mfd_2(2 - \frac{\gamma^2}{2};\ell_3), \QD_{1,0,\bullet}^\fd(\ell_4, \ell_1 - \ell_3))d\ell_1d\ell_3 d \ell_4,
\end{align*}
where $\widetilde{\mathsf{m}}$ denotes the law of $\eta([0, \tau])$ on the event $T^* \setminus T$. Note that the right-hand side is essentially the right-hand side of~\eqref{eq:weld-lem4.17-2}. Using the same argument in the proof of Lemma~\ref{lem:radial-sle-counterclock}, we obtain that 
$$
\QD_{1,0}^\fd \otimes  \widetilde{\mathsf{m}} \otimes {\rm SLE}_\kappa(d\eta^*) = C \int_0^\infty {\rm Weld}(\widetilde{\QA}^\fd(\gamma^2-2;\ell),\QD_{1,0}^\fd(\ell)) \ell d \ell.
$$
Using~\eqref{eq:weld-forest-circle} and Lemma~\ref{lem:GQD+LT}, we obtain the desired result.
\end{proof}

Next, we show the conformal welding result for $\mu^{\circlearrowright,*}$.

\begin{lemma}\label{lem:weld4.22}
    For some constant $C>0$, we have
    \begin{equation}\label{eq:weld-lem4.22}
    \begin{aligned}
   \QD_{1,0}^\fd \otimes \mu^{\circlearrowright,*} =  C \int_0^\infty {\rm Weld}(\widetilde{\QA}^\fd(2 - \frac{\gamma^2}{2};\ell),\QD_{1,0}^\fd(\ell))1_{(T^*)^c} \ell 
   d\ell.
   \end{aligned}
   \end{equation}
   The right-hand side stands for the uniform conformal welding restricted to the event $(T^*)^c$. Note that this event is measurable with respect to the welding interface. Specifically, the welding interface between $\widetilde{\QA}^\fd(2 - \frac{\gamma^2}{2})$ and $\QD_{1,0}^\fd$ can be decomposed into excursions based on its hitting time of the outer unforested boundary arc of $\widetilde{\QA}^\fd(2 - \frac{\gamma^2}{2})$, each corresponding to part of a CLE loop that touches the boundary. $(T^*)^c$ corresponds to the event that there is no excursion that separates the unforested boundary arc of $\QD_{1,0}^\fd$ from the outer unforested boundary arc of $\widetilde{\QA}^\fd(2 - \frac{\gamma^2}{2})$.
\end{lemma}

\begin{proof}
    This lemma directly follows from~\eqref{eq:welding-lem4.17-2} in Lemma~\ref{lem:radial-sle-counterclock} by restricting to the event $(T^*)^c$.
\end{proof}

As a consequence of Lemma~\ref{lem:weld4.22} and~\eqref{eq:weld-forest-circle}, we have the following lemma.

\begin{lemma}\label{lem:weld4.23}
    There exists a measure $\mathcal{QA}^{**}$ on the space of quantum surfaces with annular topology such that:
    \begin{equation}
     \QD_{1,0} \otimes \mathsf{m}^* = \int_0^\infty {\rm Weld}(\mathcal{QA}^{**}(\ell), \QD_{1,0}(\ell)) \ell d\ell.
    \end{equation}
\end{lemma}

Finally we finish the proof of Proposition~\ref{prop:weld-backbone} by combining Lemmas~\ref{lem:weld4.21} and~\ref{lem:weld4.23}.

\begin{proof}[Proof of Proposition~\ref{prop:weld-backbone}]
    By Lemmas~\ref{lem:radial-sle-counterclock} and \ref{lem:GQD+LT}, we have
    $$
    \QD_{1,0} \otimes \mathsf{n}_{T} =  C \int_0^\infty {\rm Weld}(\widetilde{\QA}(\gamma^2-2;\ell),\QD_{1,0}(\ell))\ell 
    d\ell.
    $$
    Since $\mathsf{n}_{T^*} = \mathsf{n}_T + \mathsf{n}_{T^* \setminus T}$, the conformal welding~\eqref{eq:conformal-welding-backbone-1} follows by combining the above welding equation with Lemma~\ref{lem:weld4.21}. 
    
    Next, we prove~\eqref{eq:conformal-welding-backbone-2}. The key observation is that the law $\mathsf{n}_{(T^*)^c}$ can be represented using the iteration of $\mathsf{n}_{T^*}$ and $\mathsf{m}^*$, similarly to how $\mathsf{m}$ can be represented using $\mu^\circlearrowleft$ and $\mu^\circlearrowright$. This is due to the fact that on the event $(T^*)^c$, $\mathcal{L}^*$ has the same law as the outermost loop whose outer boundary surrounds 0 in an independent CLE$_\kappa$ sampled in the domain enclosed by a loop from $\mathsf{m}^*$. Therefore, the welding equation~\eqref{eq:conformal-welding-backbone-2} follows from Lemma~\ref{lem:weld4.23} and~\eqref{eq:conformal-welding-backbone-1}. The measure of $\mathcal{QA}^*$ can be represented using the uniform conformal welding of $\wt{\QA}(\gamma^2-2)$ and $\mathcal{QA}^{**}$.
\end{proof}

\section{Derivation of $p_B(\tau)$ and $p_{BW}(\tau)$}
\label{sec:proof-Cardy}

In Section~\ref{subsec:decorate-12arm}, we introduce the Brownian annuli decorated with the complement of the one-arm event and the polychromatic two-arm event. Using the conformal welding results from Section~\ref{subsec:conformal-welding-CLE}, we can compute the joint boundary length distributions of the decorated Brownian annuli. The proof of Theorem~\ref{thm:cardy-formula} is then completed in Section~\ref{subsec:proof-cardy} using Proposition~\ref{prop:kpz-annulus}.

\subsection{The Brownian annuli decorated with the arm events}
\label{subsec:decorate-12arm}

First, we introduce the Brownian annulus decorated with the complement of the one-arm event. We sample a field from ${\rm LF}_\mathbb{H}^{(\gamma,i)}$ and an independent CLE$_6$ on top of it. Then, we apply the metric ball construction of ${\rm BA}$ from Lemma~\ref{lem:BA-equivalent}, restricted to the event that the metric ball of radius 1 centered at $i$ is enclosed by the outermost CLE$_6$ loop surrounding $i$. The resulting quantum surface, BA$_{B^c}$ defined below, along with the CLE$_6$ on it, can be viewed as the joint scaling limits of random quadrangulations and critical percolation on it restricted to the complement of the one-arm event.

\begin{lemma}\label{lem:decorate-BA-1arm}
Let $\gamma=\sqrt{8/3}$. Sample $(\mathbb{H}, i, \phi)$ from $\LF_\H^{(\gamma,i)}$ and let $d_\phi$ be the $\sqrt{8/3}$-LQG metric of $\phi$. On the event $E=\{d_\phi(i, \partial \H) > 1\}$, let $\mathcal{A} := \H \setminus B^\bullet_{d_\phi} (i,1)$ be an annular domain, and $\mathscr{L}$ be the quantum boundary length of $B^\bullet_{d_\phi} (i,1)$. Independently sample a ${\rm CLE}_6$ $\Gamma$ on $\mathbb{H}$, and let $\mathcal{L}^i$ be the outermost loop in $\Gamma$ that surrounds $i$, with $D_{\mathcal{L}^i}$ being the connected component of $\mathbb{H} \setminus \mathcal{L}^i$ that contains $i$. Let the event $F = \{ B^\bullet_{d_\phi} (i,1) \subset D_{\mathcal{L}^i} \}$; see Figure~\ref{fig:2} (left). Let ${\rm BA}_{B^c}$ denote the law of $(\mathcal{A}, \phi) /{\sim_\gamma}$ under the reweighted measure $\frac{1_{E \cap F}}{C_2 \mathscr{L} \cf(\mathscr{L})} \LF_\H^{(\gamma,i)} \otimes {\rm CLE}_6(d \Gamma)$. Then, we have ${\rm BA}_{B^c} = \frac{1}{\sqrt{2}}  (1 - p_B(\tau))\eta(2i\tau) \LF_\tau(d \phi) 1_{\tau>0} d\tau$.
\end{lemma}

\begin{proof}
    By Lemma~\ref{lem:BA-equivalent} and Proposition~\ref{prop:BA-Liouville}, the law of $(\mathcal{A}, \phi) /{\sim_\gamma}$ under the reweighted measure $\frac{1_{E}}{C_2 \mathscr{L} \cf(\mathscr{L})} \LF_\H^{(\gamma,i)}$ is $\frac{1}{\sqrt{2}}\eta(2i\tau) \LF_\tau(d \phi) 1_{\tau>0} d\tau$. Note that the boundary of $B^\bullet_{d_\phi} (i,1)$ is a Jordan loop almost surely. By Lemma~\ref{lem:CLE-12arm}, conditioning on $\phi$, the event $F$ happens with probability $1 - p_B(\tau)$ almost surely, where $\tau$ is the modulus of $\mathcal{A}$. Combining these two results yields the lemma.
\end{proof}

Similarly, we can define the Brownian annulus decorated with the polychromatic two-arm event. The resulting quantum surface, BA$_{BW}$ defined below, along with the CLE$_6$ on it, can be seen as the joint scaling limits of random quadrangulations and critical percolation restricted to the polychromatic two-arm event.

\begin{lemma}\label{lem:decorate-BA-2arm}
In the same setting as Lemma~\ref{lem:decorate-BA-1arm}, let the event $G = \{ \mathcal{L}^i \cap \partial \mathbb{H} \neq \emptyset \mbox{ and } B^\bullet_{d_\phi}(i,1) \not \subset D_{\mathcal{L}^i}\}$; see Figure~\ref{fig:2} (right). Let ${\rm BA}_{BW}$ denote the law of $(\mathcal{A}, \phi) /{\sim_\gamma}$ under the reweighted measure $\frac{1_{E \cap G}}{C_2 \mathscr{L} \cf(\mathscr{L})} \LF_\H^{(\gamma,i)} \otimes {\rm CLE}_6(d \Gamma)$. Then, we have ${\rm BA}_{BW} = \frac{1}{2 \sqrt{2}} \cdot p_{BW}(\tau) \eta(2i\tau) \LF_\tau(d \phi) 1_{\tau>0} d\tau$.
\end{lemma}

\begin{proof}
    The proof follows identically to that of Lemma~\ref{lem:decorate-BA-1arm}, using the second equation in Lemma~\ref{lem:CLE-12arm}.
\end{proof}

\begin{figure}[h]
\centering
\includegraphics[scale = 0.8]{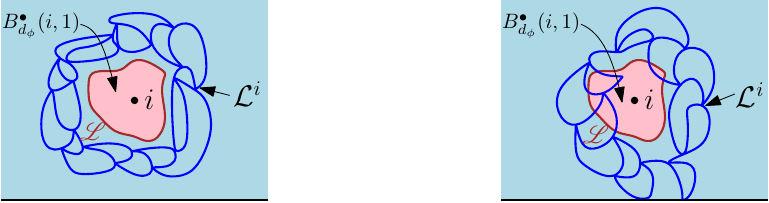}
\caption{Illustration of the events in Lemmas~\ref{lem:decorate-BA-1arm} and \ref{lem:decorate-BA-2arm}. Sample $\phi$ from $\LF_\H^{(\gamma,i)}$ and an independent CLE$_6$ on $\mathbb{H}$. In both figures, we have $d_\phi(i ,\partial \mathbb{H}) > 1$. On this event, the pink region is $B^\bullet_{d_\phi} (i,1)$ whose quantum boundary length is $\mathscr{L}$, and the light blue region is $\mathcal{A}$. The blue loop $\mathcal{L}^i$ is the outermost CLE loop that surrounds $i$. \textbf{Left:} The event $F$ occurs, i.e., the pink region is surrounded by $\mathcal{L}^i$. \textbf{Right:} The event $G$ occurs, i.e., $\mathcal{L}^i$ touches the boundary of $\mathbb{H}$ and moreover the pink region is not surrounded by $\mathcal{L}^i$.} 
\label{fig:2}
\end{figure}

In light of Proposition~\ref{prop:kpz-annulus}, to derive the formulae for $p_B(\tau)$ and $p_{BW}(\tau)$, it suffices to calculate the joint boundary length distributions of ${\rm BA}_{B^c}$ and ${\rm BA}_{BW}$. We will use the conformal welding results for $\mathcal{L}^i$ from Section~\ref{subsec:conformal-welding-CLE} and Proposition~\ref{prop:formula-CR} to derive the length distribution of the inner boundary of $\mathcal{L}^i$ and its distribution conditioned on the event $\{ \mathcal{L}^i \cap \partial \mathbb{H} \neq \emptyset \}$. This, combined with Lemma~\ref{lem:law-metric-ball}, gives the joint boundary length distributions of the desired decorated Brownian annuli; see Section~\ref{subsec:proof-cardy} for details.

\subsection{Proof of Theorem~\ref{thm:cardy-formula}}
\label{subsec:proof-cardy}

We begin by deriving the joint boundary length distributions of the quantum annuli $\mathcal{QA}$ and $\mathcal{QA}_T$ (defined in Proposition~\ref{prop:conformal-welding-cardy}) using Proposition~\ref{prop:formula-CR}. Let $\{\mathcal{QA}(\ell_1, \ell_2)\}_{\ell_1,\ell_2>0}$ be the disintegration of $\mathcal{QA}$ over the quantum lengths of the outer and inner boundaries, as in Section~\ref{subsec:KPZ}. Here, $\mathcal{QA}(\ell_1, \ell_2)$ is supported on the set of quantum annuli with outer boundary length $\ell_1$ and inner boundary length $\ell_2$. Similarly, $\{ \mathcal{QA}_T(\ell_1,\ell_2) \}_{\ell_1, \ell_2>0}$ denotes the disintegration of $\mathcal{QA}_T$.

\begin{lemma}\label{lem:distribution-QA}
Fix $\gamma = \sqrt{8/3}$. For any $\ell_1>0$ and $x \in \mathbb{R}$, we have
\begin{align}
    &\int_0^\infty |\mathcal{QA}(\ell_1,\ell_2)| \ell_2^{ix} d\ell_2 = \frac{\cos(\pi(1 - \frac{\gamma^2}{4}))}{\cosh( \frac{ \pi \gamma^2 x}{4})} \ell_1^{ix - 1};\label{eq:lem-QA-1}\\
    &\int_0^\infty |\mathcal{QA}_T(\ell_1,\ell_2)| \ell_2^{ix} d\ell_2 =  
   \frac{2\cos(\pi(1 - \frac{\gamma^2}{4})) \sinh(\pi(1 - \frac{\gamma^2}{4})x)}{\sinh(\pi x)} \ell_1^{ix - 1} .\label{eq:lem-QA-2}
\end{align}
\end{lemma}
\begin{proof}
    We begin with the proof of~\eqref{eq:lem-QA-1}. Recall the measure $\mathsf{m}$ from Proposition~\ref{prop:conformal-welding-cardy}. For a loop $\eta$ sampled from $\mathsf{m}$, let $D_\eta$ be the connected component of $\mathbb{H} \setminus \eta$ containing $i$, and let $\psi_\eta: \mathbb{H} \rightarrow D_\eta $ be a conformal map with $\psi_\eta(i) = i$. For $\alpha \in \mathbb{R}$, we write $\Delta_\alpha = \frac{\alpha}{2}(Q - \frac{\alpha}{2})$. Define the measure $\mathsf{m}^\alpha(\eta)$ by $\frac{d \mathsf{m}^\alpha(\eta)}{d \mathsf{m}(\eta)} = |\psi_\eta'(i)|^{2 \Delta_\alpha - 2}$. Using Proposition~\ref{prop:conformal-welding-cardy}, we will show that
    \begin{equation}\label{eq:lem-distribution-QA-1}
     \LF_\H^{(\alpha,i)} \otimes \mathsf{m}^\alpha(\eta) = \int_0^\infty  {\rm Weld}(\mathcal{QA}(\ell), \LF_\H^{(\alpha,i)}(\ell)) \ell d\ell.
    \end{equation}
    By \cite[Theorem 3.4]{ARS22}, there exists some $C>0$ such that $\QD_{1,0} = C \cdot \LF_\H^{(\gamma,i)}$ when viewed as a measure on quantum surfaces. Hence, by Proposition~\ref{prop:conformal-welding-cardy}, Equation~\eqref{eq:lem-distribution-QA-1} holds with $\alpha = \gamma$. For $\alpha \neq \gamma$, let $(\phi,\eta)$ be a sample from the left-hand side of~\eqref{eq:lem-distribution-QA-1} with $\alpha = \gamma$. Let $p$ be a point sampled from the harmonic measure on $\partial D_\eta$ viewed from $i$, and we fix the conformal map $\psi_\eta: \mathbb{H} \rightarrow D_\eta $ by requiring that $\psi_\eta(0) = p$. Define $X = \phi\circ\psi_\eta+Q\log|\psi_\eta'|$. The claim follows by weighting the law of $(\phi, \eta)$ by $\e^{\frac{\alpha^2-\gamma^2}{2}}e^{(\alpha-\gamma)X_\epsilon(i)}$ and sending $\epsilon \rightarrow 0$, where $X_\epsilon(i)$ is the average of the field $X$ around $i$. The proof is identical to that of~\cite[Theorem 4.6]{ARS21}, so we omit the details.

    By~\eqref{eq:CR-CLE} in Proposition~\ref{prop:formula-CR}, we have $|\mathsf{m}^\alpha| = \mathsf{m}[|\psi_\eta'(i)|^{2 \Delta_\alpha - 2}] = \frac{\cos(\pi(1 - \frac{\gamma^2}{4}))}{ \cos(\frac{\pi\gamma(Q - \alpha)}{2})}$. Applying the conformal welding from~\eqref{eq:lem-distribution-QA-1} and Lemma~\ref{lem:LFH-onepoint}, we get that for any $\ell_1>0$ and $\alpha> \frac{\gamma}{2}$,
    \begin{equation}\label{eq:lem-distribution-QA-2}
    \int_0^\infty |\mathcal{QA}(\ell_1,\ell_2)| \ell_2^{\frac{2(\alpha - Q)}{\gamma}} d\ell_2 = \frac{\cos(\pi(1 - \frac{\gamma^2}{4}))}{ \cos(\frac{\pi\gamma(Q - \alpha)}{2})} \ell_1^{\frac{2(\alpha - Q)}{\gamma} - 1}.
    \end{equation}
    Hence, as a function of $z$, $\int_0^\infty |\mathcal{QA}(\ell_1,\ell_2)| \ell_2^{z} d\ell_2$ is analytic for $\{z \in \mathbb{C} : {\rm Re} z > -\frac{4}{\gamma^2} \}$. Applying analytic continuation to both sides of~\eqref{eq:lem-distribution-QA-2} and setting $z = ix$ yields~\eqref{eq:lem-QA-1}. Equation~\eqref{eq:lem-QA-2} can be derived from a similar argument using Proposition~\ref{prop:conformal-welding-cardy} and Equation~\eqref{eq:CR-CLE-touch} in Proposition~\ref{prop:formula-CR}. \qedhere
\end{proof}

We prove Theorem~\ref{thm:cardy-formula} using the following integral identities, whose proof is postponed to Section~\ref{sec:calculation}.
\begin{lemma}\label{lem:laplace-calculation}
The following holds for all $x \in \mathbb{R}$:
    \begin{enumerate}[(1)]
        \item \label{eq:eta-integral-1}$\int_0^\infty e^{-\frac{2\pi x^2 \tau}{3}}(1-\sqrt{\frac{3}{2}} \frac{\eta(6i \tau) \eta(\frac{3}{2}i \tau)}{\eta(2 i \tau)\eta(3 i \tau)}) \eta(2 i \tau)d\tau = \frac{\sqrt{3}\tanh(\frac{2 \pi x}{3}) }{4x \cosh(\pi x)}$.
        \item \label{eq:eta-integral-2}$ \int_0^\infty e^{-\frac{2\pi x^2 \tau}{3}}  \frac{ \sqrt{3} \eta(i\tau) \eta(6i\tau)^2}{\eta(3i\tau) \eta(2i\tau)^2} \eta(2 i \tau)d\tau = \frac{\sqrt{3} \sinh(\frac{\pi x}{3})^2}{x \sinh(\pi x)} $.
    \end{enumerate}
\end{lemma}

\begin{proof}[Proof of Theorem~\ref{thm:cardy-formula} given Lemma~\ref{lem:laplace-calculation}]
Fix $\gamma = \sqrt{8/3}$. Recall the decorated Brownian annuli ${\rm BA}_{B^c}$ and ${\rm BA}_{BW}$ from Lemmas~\ref{lem:decorate-BA-1arm} and \ref{lem:decorate-BA-2arm}. Their moduli include the terms $(1 - p_B(\tau))$ and $p_{BW}(\tau)$, respectively. We will show that for all $x \in \mathbb{R}$,
\begin{equation}\label{eq:proof-thm1.1-1}
{\rm BA}_{B^c}[\mathscr{L}_1 e^{-\mathscr{L}_1} \mathscr{L}_2^{ix}] = \frac{\pi \Gamma( 1 + ix)}{4 \cosh(\pi x)  \cosh(\frac{2 \pi x}{3})} \quad \mbox{and} \quad {\rm BA}_{BW}[\mathscr{L}_1 e^{-\mathscr{L}_1} \mathscr{L}_2^{ix}] = \frac{\pi \sinh(\frac{\pi x}{3})^2 \Gamma(1+ix)}{2 \sinh(\pi x) \sinh(\frac{2 \pi x}{3})}
\end{equation}
where $\mathscr{L}_1$ and $\mathscr{L}_2$ are the quantum lengths of the outer and inner boundaries of a sample from ${\rm BA}_{B^c}$ or ${\rm BA}_{BW}$. Theorem~\ref{thm:cardy-formula} follows by combining these formulae with Proposition~\ref{prop:kpz-annulus} and Lemma~\ref{lem:laplace-calculation}. Note that by Lemma~\ref{lem:percolation-limit}, $p_B(r,R)$ and $p_{BW}(r,R)$ defined in~\eqref{eq:def-p(r,R)} are equal to $p_B(\tau)$ and $p_{BW}(\tau)$, respectively, with $\tau = \frac{1}{2 \pi} \log(\frac{R}{r})$ being the modulus of $A(r,R)$.

We first prove the formula for ${\rm BA}_{B^c}$ in~\eqref{eq:proof-thm1.1-1}. Let ${\rm BA}_{B^c}(\ell_1, \ell_2)$ be the disintegration of ${\rm BA}_{B^c}$ over the quantum lengths of the outer and inner boundaries, as in Section~\ref{subsec:KPZ}. Sample $(\mathbb{H}, \phi, i)$ from $\LF_\H^{(\gamma,i)}$ and an independent CLE$_6$ $\Gamma$ on top of it. Let $\eta$ be the inner boundary of the outermost loop in $\Gamma$ that surrounds $i$, and let $D_\eta$ be the connected component of $\mathbb{H} \setminus \eta$ containing $i$. Let $\mathscr{L}_1$ (resp.\ $\mathscr{L}_0$ and $\mathscr{L}_2$) be the quantum length of $\partial \mathbb{H}$ (resp.\ $\partial D_\eta$ and $\partial B^\bullet_\phi(i,1)$). By the conformal welding from Proposition~\ref{prop:conformal-welding-cardy}, the joint density of $(\mathscr{L}_1, \mathscr{L}_0)$ is given by $p(\ell_1, \ell_0):= \ell_0 |\mathcal{QA}(\ell_1, \ell_0)| \cdot |\LF_\H^{(\gamma,i)}(\ell_0)| 1_{\ell_0,\ell_1>0} $. Moreover, conditioned on $\mathscr{L}_1 = \ell_1$ and $\mathscr{L}_0 = \ell_0$, the conditional law of $(D_\eta, \phi, i)/{\sim_\gamma}$ follows $\LF_\H^{(\gamma,i)}(\ell_0)^\#$. It is evident that the event $E \cap F$ defined in Lemma~\ref{lem:decorate-BA-1arm} is equivalent to $\{ d_\phi(i, \partial D_\eta) > 1 \}$. Therefore, by Lemma~\ref{lem:law-metric-ball}, we have for $\ell_1, \ell_2>0$
\begin{equation}\label{eq:proof-thm1.1-2}
|{\rm BA}_{B^c}(\ell_1, \ell_2)| = \int_0^\infty p(\ell_0, \ell_1) \times \frac{Z(\ell_0)}{Z(\infty)} \frac{1}{Z(\ell_0)} \frac{\ell_0^{3/2} \ell_2 \cf(\ell_2)}{\sqrt{\ell_0\ell_2}(\ell_0 + \ell_2)} \times \frac{1}{C_2 \ell_2\cf(\ell_2)} d \ell_0,
\end{equation}
where the term $\frac{1}{C_2 \ell_2\cf(\ell_2)}$ comes from the weighted measure in the definition of ${\rm BA}_{B^c}$ (see Lemma~\ref{lem:decorate-BA-1arm}). Recall from Lemma~\ref{lem:BA-equivalent} that $|\LF_\H^{(\gamma,i)}(\ell_0)| = C_1 \ell_0^{-3/2}$ and $C_2 = 2C_1/Z(\infty)$. Therefore, we further have
$$
{\rm BA}_{B^c}[\mathscr{L}_1 e^{-\mathscr{L}_1} \mathscr{L}_2^{ix}] = \iiint_0^\infty \ell_1 e^{-\ell_1} \frac{\ell_2^{ix} \sqrt{\ell_0}}{2\sqrt{\ell_2}(\ell_0+\ell_2)} | \mathcal{QA}(\ell_1, \ell_0) | d \ell_0 d\ell_1 d\ell_2.
$$
By substituting $\ell_2 = \ell_0 t$, the above integral becomes
\begin{align*}
{\rm BA}_{B^c}[\mathscr{L}_1 e^{-\mathscr{L}_1} \mathscr{L}_2^{ix}] &= \iint_0^\infty  \ell_1 e^{-\ell_1} \ell_0^{ix} | \mathcal{QA}(\ell_1, \ell_0) |  d \ell_0 d\ell_1 \cdot \int_0^\infty \frac{t^{ix}}{2\sqrt{t}(1+t)} dt \\
& = \frac{\pi}{2 \cosh(\pi x)} \iint_0^\infty  \ell_1 e^{-\ell_1} \ell_0^{ix} | \mathcal{QA}(\ell_1, \ell_0) |  d \ell_0 d\ell_1,
\end{align*}
where we used $\int_0^\infty \frac{t^{ix}}{\sqrt{t}(1+t)} dt = \frac{\pi}{\cosh(\pi x)}$ for $x \in \mathbb{R}$. Applying~\eqref{eq:lem-QA-1} from Lemma~\ref{lem:distribution-QA}, we obtain:
\begin{align*}
    {\rm BA}_{B^c}[\mathscr{L}_1 e^{-\mathscr{L}_1} \mathscr{L}_2^{ix}] = \frac{\pi}{4 \cosh(\pi x) \cosh(\frac{2 \pi x}{3})} \int_0^\infty \ell_1^{ix} e^{-\ell_1} d\ell_1 = \frac{\pi \Gamma(1 + ix)}{4 \cosh(\pi x) \cosh(\frac{2 \pi x}{3})}.
\end{align*}
By first calculating the case of $x=0$, we can check that the above integrals are absolutely integrable, allowing us to exchange the order of integration. This completes the proof of the first equation in~\eqref{eq:proof-thm1.1-1}.

Next, we derive the formula for ${\rm BA}_{BW}$ in~\eqref{eq:proof-thm1.1-1}. Let ${\rm BA}_{BW}(\ell_1, \ell_2)$ be the disintegration of ${\rm BA}_{BW}$ over the quantum lengths of the outer and inner boundaries. We use the same setting and notation as before: $(\mathbb{H}, \phi, i)$ is sampled from $\LF_\H^{(\gamma,i)}$, with an independent CLE$_6$ $\Gamma$ on top of it. Recall from Lemmas~\ref{lem:decorate-BA-1arm} and \ref{lem:decorate-BA-2arm} that $\mathcal{L}^i$ is the outermost loop in $\Gamma$ that surrounds $i$, and by definition, $D_{\mathcal{L}^i} = D_\eta$. Let $T = \{ \mathcal{L}^i \cap \partial \mathbb{H} \neq \emptyset \}$, and recall the event $G = \{ \mathcal{L}^i \cap \partial \mathbb{H} \neq \emptyset, B^\bullet_{d_\phi}(i,1) \not \subset D_{\mathcal{L}^i} \}$. Then we have $G \subset T$ and $T \setminus G = \{\mathcal{L}^i \cap \partial \mathbb{H} \neq \emptyset, B^\bullet_{d_\phi}(i,1) \subset D_\eta \}$. 

We will use the relation $1_{E \cap G} = 1_{E \cap T} - 1_{E \cap (T \setminus G)}$ to compute $|{\rm BA}_{BW}(\ell_1, \ell_2)|$. Since $\phi$ and $\Gamma$ are independently sampled, the joint density of $(\mathscr{L}_1, \mathscr{L}_2)$ under the reweighted measure $\frac{1_{E \cap T}}{C_2 \mathscr{L} \cf(\mathscr{L})} \LF_\H^{(\gamma,i)} \otimes {\rm CLE}_6(d \Gamma)$ is given by
$$
\mathbb{P}[T] \times \frac{1}{2\sqrt{\ell_1 \ell_2}(\ell_1+\ell_2) } 1_{\ell_1,\ell_2>0} d\ell_1 d\ell_2 = \frac{1}{4\sqrt{\ell_1 \ell_2}(\ell_1+\ell_2) } 1_{\ell_1,\ell_2>0} d\ell_1 d\ell_2,
$$
where we used the identity $\mathbb{P}[T] = 1/2$ from Proposition~\ref{prop:formula-CR}. Similar to~\eqref{eq:proof-thm1.1-2}, we can derive the joint density of $(\mathscr{L}_1, \mathscr{L}_2)$ under the reweighted measure $\frac{1_{E \cap (T \setminus G)}}{C_2 \mathscr{L} \cf(\mathscr{L})} \LF_\H^{(\gamma,i)} \otimes {\rm CLE}_6(d \Gamma) $. Therefore, for $\ell_1, \ell_2>0$, we have:
$$
|{\rm BA}_{BW}(\ell_1, \ell_2)| = \frac{1}{4\sqrt{\ell_1 \ell_2}(\ell_1+\ell_2) } - \int_0^\infty \frac{\sqrt{\ell_0}}{2\sqrt{\ell_2}(\ell_0+\ell_2)}| \mathcal{QA}_T(\ell_1, \ell_0) | d \ell_0.
$$
Using similar calculations as before and \eqref{eq:lem-QA-2} from Lemma~\ref{lem:distribution-QA}, we conclude that 
$$
{\rm BA}_{BW}[\mathscr{L}_1 e^{-\mathscr{L}_1} \mathscr{L}_2^{ix}] = \frac{\pi \Gamma(1+ix)}{4 \cosh(\pi x)} - \frac{\pi \sinh(\frac{\pi x}{3}) \Gamma(1+ix)}{2 \cosh(\pi x) \sinh(\pi x)} = \frac{\pi \sinh(\frac{\pi x}{3})^2 \Gamma(1+ix)}{2 \sinh(\pi x) \sinh(\frac{2 \pi x}{3})}. \qedhere
$$

\end{proof}
\section{Derivation of $p_{BB}(\tau)$}
\label{sec:backbone}

In Section~\ref{subsec:backbone-annulus}, we introduce the Brownian annulus decorated with the monochromatic two-arm event. In Section~\ref{subsec:proof-lem5.2}, we will prove Lemma~\ref{lem:decorate-BA-backbone}, showing that the modulus of this quantum surface is indeed weighted by the crossing probability $p_{BB}(\tau)$, by combining various results. The proof of Theorem~\ref{thm:backbone-crossing} is completed in Section~\ref{subsec:finalproof-backbone}. Although the high-level strategy is similar to the derivation of $p_{B}(\tau)$ and $p_{BW}(\tau)$, the arguments are more tricky in several places.

\subsection{The Brownian annulus decorated with the monochromatic two-arm event}
\label{subsec:backbone-annulus}

We begin by introducing the Brownian annulus decorated with the monochromatic two-arm event. Similar to the cases of BA$_{B^c}$ and BA$_{BW}$, it can be viewed as the joint scaling limits of random quadrangulations and critical percolation restricted to the monochromatic two-arm event.

\begin{definition}\label{def:decorate-BA-backbone}
    Set $\gamma = \sqrt{8/3}$. Sample $(\mathbb{H}, i, \phi)$ from $\LF_\H^{(\gamma,i)}$ and an independent ${\rm CLE}_6$ $\Gamma$ on top of it. Let $\mathcal{L}^*$ be the outermost loop in $\Gamma$ whose outer boundary surrounds $i$. Let $\ep{\mathcal{L}^*}$ be the outer boundary of $\mathcal{L}^*$ and $\mathscr{L}$ be the quantum length of $\ep{\mathcal{L}^*}$. Let $T^* = \{ \mathcal{L}^* \cap \partial \mathbb{H} \neq \emptyset\}$. On the event $(T^*)^c$, let $\mathcal{A}$ be the annular connected component of $\mathbb{H} \setminus \mathcal{L}^*$; see Figure~\ref{fig:3} (left). We write ${\rm BA}_{BB}$ as the law of $(\mathcal{A}, \phi)/{\sim_\gamma}$ under the reweighted measure $\sqrt{\mathscr{L}} 1_{(T^*)^c} \LF_\H^{(\gamma,i)} \otimes {\rm CLE}_6 (d\Gamma)$. 

\end{definition}

\begin{figure}[h]
\centering
\includegraphics[scale = 0.8]{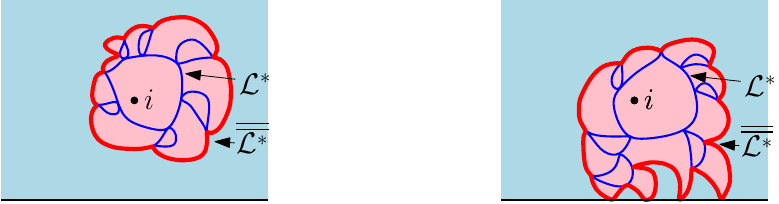}
\caption{Illustration of Definition~\ref{def:decorate-BA-backbone}. The blue loop $\mathcal{L}^*$ is the outermost CLE$_6$ loop whose outer boundary surrounds $i$. The red loop $\ep{\mathcal{L}^*}$ is its outer boundary which has quantum length $\mathscr{L}$. \textbf{Left:} The event $(T^*)^c$ occurs, i.e., $\mathcal{L}^*$ does not touch $\partial \mathbb{H}$. On this event, the light blue region corresponds to $\mathcal{A}$. \textbf{Right:} The event $T^*$ occurs, i.e., $\mathcal{L}^*$ touches $\partial \mathbb{H}$. On this event, as shown by the conformal welding result in Proposition~\ref{prop:weld-backbone}, the law of the quantum surface corresponding to the light blue region can be described by $\wt{\QA}(\gamma^2-2)$.} 
\label{fig:3}
\end{figure}

The following lemma states that the modulus of ${\rm BA}_{BB}$ encodes the crossing probability $p_{BB}(\tau)$. Its proof will be provided in Section~\ref{subsec:proof-lem5.2}.

\begin{lemma}\label{lem:decorate-BA-backbone}
    There exists some constant $C_3>0$ such that ${\rm BA}_{BB} = C_3 \cdot p_{BB}(\tau) \eta(2i\tau) \LF_\tau(d\phi) 1_{\tau>0} d\tau $.
\end{lemma}
We now give an outline of the proof of Theorem~\ref{thm:backbone-crossing} based on Lemma~\ref{lem:decorate-BA-backbone}.
In light of Proposition~\ref{prop:kpz-annulus}, to derive $p_{BB}(\tau)$, it suffices to calculate the joint boundary length distribution of ${\rm BA}_{BB}$. By Proposition~\ref{prop:weld-backbone}, ${\rm BA}_{BB}$ is simply the quantum annulus $\mathcal{QA}^*$ defined there modulo a multiplicative constant. We can use a similar argument as in Lemma~\ref{lem:distribution-QA} to derive the joint boundary length distribution of $\mathcal{QA}^*$. However, this involves the law of the conformal radius of $\ep{\mathcal{L}^*}$ restricted to the event $(T^*)^c$, which is not known. To overcome this difficulty, we also consider the law of the quantum surface between $\partial \mathbb{H}$ and $\ep{\mathcal{L}^*}$ on the event $T^*$, which can be described by a pinched quantum annulus $\wt{\QA}(\gamma^2-2)$; see~\eqref{eq:conformal-welding-backbone-1} of Proposition~\ref{prop:weld-backbone} and Figure~\ref{fig:3} (right). The joint boundary length distribution of this pinched quantum annulus can be derived from the boundary structure constants of Liouville conformal field theory~\cite{RZ22}. Hence, we can derive the law of the conformal radius of $\ep{\mathcal{L}^*}$ restricted to the event $(T^*)^c$ up to some constant. This, together with Proposition~\ref{prop:formula-CR}, yields the boundary length distribution of $\mathcal{QA}^*$ up to two unknown constants (see Lemma~\ref{lem:length-distribution-backbone}). Finally, these two constants can be fixed using Proposition~\ref{prop:kpz-annulus} and the fact that $\lim_{\tau \to 0} p_{BB}(\tau) = 1$. We refer to Section~\ref{subsec:finalproof-backbone} for details.

\subsection{Proof of Lemma~\ref{lem:decorate-BA-backbone}}
\label{subsec:proof-lem5.2}
We first explain the discrete intuition behind Lemma~\ref{lem:decorate-BA-backbone}. Sample a random quadrangulation of the disk with one bulk marked point and an independent  critical Bernoulli percolation on top of it. Consider the outermost filled percolation interface (obtained by filling passages with width 1) that surrounds the marked point which corresponds to $\ep{\mathcal{L}^*}$ in the scaling limit. This filled interface divides the random quadrangulation into two independent parts conditioned on the number of vertices along it. The outer part may have pinched point depending on whether the filled percolation interface touches the boundary. The key observation is that if it has annular topology, then the percolation on it is decorated with the monochromatic two-arm event and thus the modulus encodes $p_{BB}(\tau)$ in the continuum. Instead of using the convergence result to make this intuition rigorous, we will directly prove Lemma~\ref{lem:decorate-BA-backbone} in the continuum. 

Now we outline the proof of Lemma~\ref{lem:decorate-BA-backbone}. The idea is to give an equivalent description of ${\rm BA}_{BB}$ starting from the Brownian sphere ${\rm BS}_2$. Sample a quantum surface from ${\rm BS}_2$ and embed it as $(\widehat{\mathbb{C}}, 0, \infty)$, and then sample a whole-plane critical percolation on $\widehat{\mathbb{C}}$. Let $\mathscr{C}_1$ be the boundary of the filled unit metric ball centered at $\infty$, and let $\mathscr{C}_2$ be the outer boundary of a percolation interface sampled from the counting measure. We will show that the quantum annulus bounded by $\mathscr{C}_1$ and $\mathscr{C}_2$ has the same law as both ${\rm BA}_{BB}$ and the Brownian annulus ${\rm BA}$ decorated with the monochromatic two-arm event, provided that the monochromatic two-arm event occurs in the annular region and the measure is reweighted by quantum boundary lengths; see Lemma~\ref{lem:BABB-equivalent} for the precise statement and Figure~\ref{fig:4} for an illustration. On the one hand, we can sample $\mathscr{C}_1$ first, and by Lemma~\ref{lem:law-sphere-metric-ball}, the quantum surface enclosed by it can be described by ${\rm LF}_\H^{(\gamma,i)}$. The rest of the construction of the quantum annulus is exactly the same as in Definition~\ref{def:decorate-BA-backbone}, and thus, we obtain ${\rm BA}_{BB}$. On the other hand, we can sample $\mathscr{C}_2$ first and consider the region outside of it. In fact, this quantum surface can also be described by ${\rm LF}_\H^{(\gamma,i)}$. This follows from the fact that the marginal law of $\mathscr{C}_2$ is the same as the SLE$_{8/3}$ loop measure on $\widehat{\mathbb{C}}$ by~\cite{werner-loop}, combined with the conformal welding result for the SLE$_{8/3}$ loop measure and ${\rm QS}_2$ (which is equivalent to ${\rm BS}_2$ by~\cite{lqg-tbm1, lqg-tbm2, lqg-tbm3}) from~\cite{AHS-loop}; see Sections~\ref{subsec:QS} and~\ref{subsec:percolation-whole-plane} for details. Then, by Lemma~\ref{lem:BA-equivalent}, sampling $\mathscr{C}_1$ yields the Brownian annulus ${\rm BA}$ and the restricted event is equivalent to the monochromatic two-arm event.

\subsubsection{Quantum sphere and the SLE$_{8/3}$ loop}
\label{subsec:QS}
In this section, we review the $\sqrt{8/3}$-LQG description of the free Brownian sphere ${\rm BS}_2$ given by Miller and Sheffield~\cite{lqg-tbm1, lqg-tbm2, lqg-tbm3}. The two-pointed quantum sphere was introduced in~\cite{DMS14}, which we now recall.

\begin{definition}\label{def:QS}
    Fix $\gamma=\sqrt{8/3}$ and let $Q=\frac{2}{\gamma}+\frac{\gamma}{2}$. Let $\mathcal{C}=\mathbb{R} \times [0,2\pi]/{\sim}$ be a horizontal cylinder. Set $h=h^1+h^2_\cC+\mathbf c$, where $h^1,h^2$ are random generalized functions on $\cC$, $\mathbf c \in \mathbb{R}$, and they are independently sampled as follows. Let $h^1(z) = Y_{\mathrm{Re} z}$ for each $z \in \cC$, where \[Y_t =
	\left\{
	\begin{array}{ll}
	B_{t} - (Q -\gamma)t  & \mbox{if } t \geq 0 \\
	\wt B_{-t} +(Q-\gamma) t & \mbox{if } t < 0
	\end{array}
	\right. ,\]
   and $(B_s)_{s \geq 0}$ and $(\widetilde B_s)_{s \geq 0}$ are two independent standard Brownian motions conditioned on $B_s - (Q-\gamma)s<0$ and $\widetilde B_s - (Q-\gamma)s<0$ for all $s>0$. Let $h^2_\cC$ be the lateral part of the Gaussian free field on $\cC$. Let $\mathbf c\in \R$ be sampled from $ \frac\gamma2 e^{2(\gamma-Q)c}dc$. Let ${\rm QS}_2$ be the infinite measure describing the law of $(\mathcal{C},h,-\infty,\infty)/{\sim_\gamma}$. We call a sample from ${\rm QS}_2$ a two-pointed quantum sphere.
\end{definition}

The following theorem summarizes the relation between ${\rm QS}_2$ and ${\rm BS}_2$.

\begin{theorem}\label{thm:bm-lqg-sphere}
Recall the constants $c_1 = 1$ and $c_2$ from Theorem~\ref{thm:bm-lqg}. Sample $h$ from ${\rm QS}_2$, and let $(\mathcal{S}, x, y, d,\mu)$ be the marked metric-measure space given by $(\mathcal{C}, -\infty, \infty, c_1 d_h,c_2 \mu_h)$. Then the law of $(\mathcal{S}, x, y, d,\mu)$ is $C \cdot {\rm BS}_2$ for some constant $C>0$. Moreover, $(\mathcal{S}, x, y, d,\mu)$  and the quantum surface $(\mathcal{C}, h, -\infty, \infty)/{\sim_\gamma}$ are measurable with respect to each other. 
\end{theorem}

Next, we recall a conformal welding result from~\cite{AHS-loop, ACSW24} about the quantum sphere. We first introduce the SLE$_{8/3}$ loop measure, as defined in~\cite{werner-loop}. We say that a family of measures defined on simple loops on any Riemann surface satisfies the conformal restriction property if (1). for any conformal map $f: S \rightarrow S'$, $f \circ \mu_S = \mu_{S'}$ (2). for any $S \subset S'$, $\mu_S$ is equal to the measure $\mu_{S'}$ restricted to those loops that are contained in $S$. 

\begin{theorem}[Theorem 1 of \cite{werner-loop}]\label{thm:sle-loop}
    There exists a unique family of measures $(\mu_S)$ (modulo a multiplicative constant) that satisfies the conformal restriction property.
\end{theorem}

We say that a family of measures satisfies the weak conformal restriction property if we restrict to simply connected domains. As shown in Section 3 of~\cite{werner-loop}, there exists a unique family of measures that satisfies the weak conformal restriction property, which can be obtained by restricting $(\mu_S)$ to simply connected domains.

Let $\mu_{\mathbb{C}}$ be the SLE$_{8/3}$ loop measure on $\mathbb{C}$. The original construction in~\cite{werner-loop} is via the Brownian loop measures. We also refer to~\cite{zhan-sle-loop} for a construction via SLE. Let $\mu^{\rm sur}_{\mathbb{C}}$ be the measure $\mu_{\mathbb{C}}$ restricted to those loops that surround 0. In~\cite[Proposition 6.5]{ACSW24}, based on~\cite{AHS-loop}, the following conformal welding result is proved for the quantum sphere and the SLE$_{8/3}$ loop measure.

\begin{lemma}\label{lem:QS-weld}
    We have ${\rm QS}_2 \otimes \mu^{\rm sur}_{\mathbb{C}} = C \cdot \int_0^\infty {\rm Weld}({\rm QD}_{1,0}(\ell), {\rm QD}_{1,0}(\ell)) \ell d \ell$ for some $C>0$. Here, the right-hand side stands for the uniform conformal welding of two independent samples from ${\rm QD}_{1,0}$ conditioned on having the same quantum boundary length.
\end{lemma}

\subsubsection{Critical percolation on $\mathbb{C}$}
\label{subsec:percolation-whole-plane}

In this section, we provide some properties of critical percolation in $\mathbb{C}$. As explained in Section~\ref{sec:percolation-cle}, for any open domain $D \subset \mathbb{C}$, we can use a law on the quad-crossing configuration $\omega \in \mathcal{H}(D)$ to describe the scaling limit of critical percolation on $D$. Specifically, the quad-crossing configuration induced by Bernoulli $\frac{1}{2}$-site percolation on $\delta \mathbb{T}$ is known to converge in law to $\omega$ for the metric $d_{\mathcal{H}(D)}$ as $\delta \rightarrow 0$. Furthermore, when $D$ is a simply connected domain or doubly connected domain with piecewise smooth boundaries, the quad-crossing configuration and the percolation interfaces for the discrete percolation jointly converge in law to $(\omega, \Gamma)$ for the metrics $d_{\mathcal{H}(D)}$ and $d_\ell$. Note that when defining the percolation interfaces, we always assume that the boundary conditions are black. In addition, $\omega$ and $\Gamma$ are measurable with respect to each other, and their joint law is denoted by $\mathbb{P}_D$. The joint law of $(\omega, \Gamma)$ can also be extended to general simply connected domain or doubly connected domains by conformal mappings.

We have the following Markov property for critical percolation on $\mathbb{C}$. It can be deduced from the independence of discrete percolation in disjoint quads and the scaling limit argument.

\begin{lemma}\label{lem:percolation-markov}
    Sample $(\omega, \Gamma)$ from $\mathbb{P}_{\mathbb{C}}$ and $\eta$ from the counting measure on those percolation interfaces in $\Gamma$ that have black color on the outside. Let $D_\eta$ be the infinite connected component of $\mathbb{C} \setminus \eta$. Then, conditioned on $\ep{\eta}$ (the boundary of $D_\eta$), $\omega|_{D_\eta}$ has the same law as a quad-crossing configuration sampled from $\mathbb{P}_{D_\eta}$.
\end{lemma}

Now we describe the law of $\ep{\eta}$ for the loop $\eta$ as sampled in Lemma~\ref{lem:percolation-markov}. For any simply connected domain $D \subset \mathbb{C}$, sample a critical percolation on $D$ and $\eta$ from the counting measure on the percolation interfaces that have black color on the outside and do not touch $\partial D$, and let $\tilde{\mu}_D$ be the law of $\ep{\eta}$. Then, by Lemma~\ref{lem:local}, the family of measures $(\mu_D)_{D \subset \mathbb{C}}$ satisfies the weak conformal restriction property. As a consequence of Theorem~\ref{thm:sle-loop}, we have the following lemma, which is first stated in Section 8 of~\cite{werner-loop}.

\begin{lemma}\label{lem:percolation-sle-loop}
    $\tilde{\mu}_{\mathbb{C}}$ and $\mu_{\mathbb{C}}$ have the same law up to a multiplicative constant.
\end{lemma}

\subsubsection{Equivalent construction of ${\rm BA}_{BB}$}

In the following lemma, we define a quantum annulus by starting with a sample from ${\rm QS}_2$ and an independent whole-plane percolation on top of it. Then we will show by combining results from Sections~\ref{subsec:QS} and \ref{subsec:percolation-whole-plane} that the resulting quantum annulus has the same law as ${\rm BA}_{BB}$ up to a multiplicative constant. Moreover, its modulus encodes the crossing probability $p_{BB}(\tau)$. We refer to the beginning of Section~\ref{subsec:proof-lem5.2} for the proof outline. As an immediate consequence, we obtain Lemma~\ref{lem:decorate-BA-backbone}. Let $\widehat{\mathbb{C}} = \mathbb{C} \cup \{ \infty \}$ be the Riemann sphere.

\begin{lemma}\label{lem:BABB-equivalent}
    Let $(\widehat{\mathbb{C}}, h, 0, \infty)$ be an embedding of a sample from ${\rm QS}_2$ and $d_h$ be the $\sqrt{8/3}$-LQG metric of $h$. Sample a whole-plane percolation $(\omega, \Gamma)$ from $\mathbb{P}_\mathbb{C}$ on top of it, and a loop $\eta$ from the counting measure on those percolation interfaces in $\Gamma$ that have black color on the outside. Define the event 
    $$J = \{ d_h(\infty, \ep{\eta}) > 1, \ep{\eta} \mbox{ separates 0 and $\infty$} \}.$$ 
    On the event $J$, let $\mathcal{A}$ be the annular connected component of $\widehat{\mathbb{C}} \setminus (B_{d_h}(\infty,1) \cup \eta)$. Consider the restriction of $\omega$ on the domain $\mathcal{A}$, namely $\omega|_\mathcal{A}$, and let $\Gamma_\mathcal{A}$ be the annulus ${\rm CLE}_6$ induced by $\omega|_{\mathcal{A}}$ on $\mathcal{A}$ with black boundary conditions as in Lemma~\ref{lem:CLE-double-domain}. Let
    $$
    K = \{\mbox{there does not exist $\gamma \in \Gamma_\mathcal{A}$ such that $\ep{\gamma}$ separates $0$ and $\infty$} \}.
    $$
    Let $\mathscr{L}_1$ be the quantum length of $\partial B^\bullet_{d_h}(\infty,1)$ and $\mathscr{L}_2$ be the quantum length of $\ep \eta$. We write $\widetilde {\rm BA}$ as the law of $(\mathcal{A}, h)/{\sim_\gamma}$ under the reweighted measure $\frac{1_{J \cap K} \sqrt{\mathscr{L}_2}}{\mathscr{L}_1 \cf(\mathscr{L}_1)} {\rm QS}_2 \otimes \mathbb{P}_\mathbb{C} (d\Gamma) \otimes {\rm Count}_\Gamma(d\eta)$. Then there exist constants $C, C'>0$ such that the following holds:
    \begin{enumerate}[(i)]
        \item $\widetilde {\rm BA} = C \cdot {\rm BA}_{BB}$.\label{claim:BA-equivalent}
        \item $\widetilde {\rm BA} = C' \cdot p_{BB}(\tau) \eta(2i\tau) \LF_\tau(d\phi) 1_{\tau>0} d\tau $.\label{claim:BA-modulus}
    \end{enumerate}
    \begin{figure}[h]
\centering
\includegraphics[scale = 0.8]{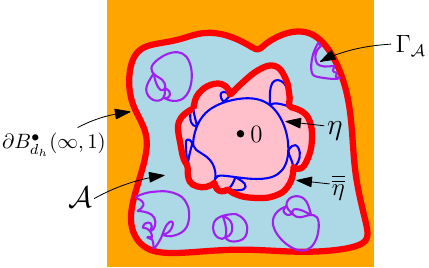}
\caption{Illustration of the construction of $\widetilde {\rm BA}$. In this figure, both the events $J$ and $K$ occur. Specifically, the orange region corresponds to $B^\bullet_{d_h}(\infty,1)$, and the light-blue region corresponds to $\mathcal{A}$. The purple loops in $\mathcal{A}$ correspond to $\Gamma_\mathcal{A}$, which are the percolation interfaces induced by $\omega|_{\mathcal{A}}$ with black boundary conditions. Note that the boundary touching loops in $\Gamma_\mathcal{A}$ may not appear in $\Gamma$. $K$ is the event that there are no purple loops whose outer boundaries separate $0$ and $\infty$. The two red loops have quantum length $\mathscr{L}_1$ and $\mathscr{L}_2$, respectively.} 
\label{fig:4}
\end{figure}
\end{lemma}

\begin{proof}
    Let $(\widehat{\mathbb{C}}, h, 0, \infty)$ be an embedding of a sample from ${\rm QS}_2$, and sample a whole-plane percolation $(\omega, \Gamma)$. We begin with the proof of Claim~\eqref{claim:BA-equivalent}. On the event $O := \{d_h(0, \infty) > 1 \} \supset J$, let $D_0$ be the connected component of $\widehat{\mathbb{C}} \setminus B_{d_h}(\infty, 1)$ containing 0. Let $\Gamma_0$ be the disk CLE$_6$ induced by $\omega|_{D_0}$ with black boundary conditions. By definition, those percolation interfaces in $\Gamma_0$ that have black color on the outside and do not touch $\partial D_0$ are the same as those percolation interfaces in $\Gamma$ that have black color on the outside and are contained in $D_0$. Therefore, if a loop $\eta \in \Gamma$ satisfies the event $J \cap K$, then it must be contained in $D_0$ and do not touch $\partial D_0$, and thus it is also a loop in $\Gamma_0$. Moreover, this loop must be the outermost loop in $\Gamma_0$ whose outer boundary surrounds 0. (This is because the event $J$ implies that $\ep{\eta}$ surrounds 0, and the event $K$ implies that $\eta$ is the outermost such loop; otherwise, if there exists another loop in $\Gamma_0$ whose outer boundary surrounds 0 and $\eta$, then this loop also appears in $\Gamma_\mathcal{A}$ and so the event $K$ cannot occur.) On the other hand, if the outermost loop in $\Gamma_0$ whose outer boundary surrounds 0 does not touch $\partial D_0$, then it is in $\Gamma$ and satisfies the event $J \cap K$. By Lemma~\ref{lem:local}, given $h$ and on the event $O$, $(\omega|_{D_0}, \Gamma_0)$ has the same law as critical percolation on $D_0$, namely following the law $\mathbb{P}_{D_0}$. Therefore, after first sampling $h$ from $\frac{1_{O}}{\mathscr{L}_1 \cf(\mathscr{L}_1)} {\rm QS}_2$ and restricting to the quantum surface $(D_0, h)/{\sim_\gamma}$, the rest of the construction of $\widetilde {\rm BA}$ and ${\rm BA}_{BB}$ as defined in Definition~\ref{def:decorate-BA-backbone} are the same. Furthermore, by Lemma~\ref{lem:law-sphere-metric-ball} and Theorem~\ref{thm:bm-lqg-sphere}, the law of $(D_0, h)/{\sim_\gamma}$ under the reweighted measure $\frac{1_{O}}{\mathscr{L}_1 \cf(\mathscr{L}_1)} {\rm QS}_2$ is equivalent to ${\rm LF}_\H^{(\gamma,i)}$ up to a multiplicative constant. Hence, we obtain Claim~\eqref{claim:BA-equivalent}.

    Next, we prove Claim~\eqref{claim:BA-modulus}. Let $D_\eta$ be the infinite connected component of $\mathbb{C} \setminus \eta$ and $\tilde \mu_\mathbb{C}$ be the law of $\ep{\eta}$. By Lemmas~\ref{lem:percolation-markov} and \ref{lem:local}, conditioned on $\ep{\eta}$ and $h$ and on the event $J$, $(\omega|_\mathcal{A}, \Gamma_\mathcal{A})$ has the same law as critical percolation on $\mathcal{A}$, namely following the law $\mathbb{P}_\mathcal{A}$. We claim that $(\mathcal{A}, h)/{\sim_\gamma}$, under the reweighted measure $\frac{1_J \sqrt{\mathscr{L}_2}}{\mathscr{L}_1 \cf(\mathscr{L}_1)} {\rm QS}_2 \otimes \tilde \mu_\mathbb{C}$, has the law of $\eta(2i\tau) \LF_\tau(d\phi) 1_{\tau>0} d\tau$ up to a multiplicative constant. Combining this with Lemma~\ref{lem:CLE-backbone} yields Claim~\eqref{claim:BA-modulus}. Finally we prove the claim. Let $J' := \{\ep{\eta} \mbox{ separates $0$ and $\infty$}\} \supset J$. By Lemma~\ref{lem:percolation-sle-loop}, $\tilde \mu_\mathbb{C}$ restricted to the event $J'$ is equivalent to $\mu_\mathbb{C}^{\rm sur}$ up to a multiplicative constant. Therefore, by Lemma~\ref{lem:QS-weld} and the fact that ${\rm QD}_{1,0} = C \cdot {\rm LF}_\H^{(\gamma,i)}$ and $|{\rm LF}_\H^{(\gamma,i)}(\ell)| = C_1 \ell^{-\frac{3}{2}}$, the law of $(D_\eta, h)/{\sim_\gamma}$ under the reweighted measure $1_{J'} \sqrt{\mathscr{L}_2} {\rm QS}_2 \otimes \tilde \mu_\mathbb{C}$ is equivalent to ${\rm LF}_\H^{(\gamma,i)}$ up to a multiplicative constant. Then applying Lemma~\ref{lem:law-metric-ball} and Proposition~\ref{prop:BA-Liouville} yields the claim.
\end{proof}

\begin{proof}[Proof of Lemma~\ref{lem:decorate-BA-backbone}]
    Lemma~\ref{lem:decorate-BA-backbone} follows directly from Lemma~\ref{lem:BABB-equivalent}.
\end{proof}

\subsection{Proof of Theorem~\ref{thm:backbone-crossing}}\label{subsec:finalproof-backbone}

Recall the pinched quantum annulus $\wt{\QA}(\gamma^2-2)$ from Definition~\ref{def:QA}. We will first provide the joint boundary length distribution of $\wt{\QA}(\gamma^2-2)$.

\begin{lemma}\label{lem:QA(W)-length}
    Fix $\gamma \in (\sqrt{2},2)$. Let $\{\wt{\QA}(\gamma^2-2; \ell_1, \ell_2) \}_{\ell_1, \ell_2 >0}$ be the disintegration of $\wt{\QA}(\gamma^2-2)$ over the quantum lengths of the outer and inner boundaries. There exists a constant $C>0$ depending on $\gamma$ such that for any $\ell_1>0$ and $\alpha \in (\frac{\gamma}{2}, Q)$, 
    $$
    \int_0^\infty |\wt{\QA}(\gamma^2-2; \ell_1, \ell_2)|  \ell_2^{\frac{2(\alpha - Q)}{\gamma}} d\ell_2 = C \cdot \frac{\sin(\pi ( 1 - \frac{\gamma^2}{4}) \cdot \frac{2(\alpha - Q)}{\gamma} )}{\sin(\frac{\pi \gamma^2 }{4}\cdot \frac{2(\alpha - Q)}{\gamma})} \ell_1^{\frac{2(\alpha - Q)}{\gamma} - 1}.
    $$
\end{lemma}
\begin{proof}
    This result follows from Definition~\ref{def:QA} and \cite[Proposition 3.5]{AHS21}. 
\end{proof}

Next, we derive the joint boundary length distribution of BA$_{BB}$ using the conformal welding results in Proposition~\ref{prop:weld-backbone}, the formula for the CLE conformal radius from Proposition~\ref{prop:formula-CR}, and Lemma~\ref{lem:QA(W)-length}.

\begin{lemma}\label{lem:length-distribution-backbone}
    Fix $\gamma = \sqrt{8/3}$. There exist constants $C_4, C_5>0$ such that for any $\ell_1>0$ and $x \in \mathbb{R}$,
    \begin{equation}\label{eq:lem-BA-length-1}
    \int_0^\infty |{\rm BA}_{BB}(\ell_1,\ell_2)| \ell_2^{ix} d\ell_2 = \Big(C_4 \cdot \frac{\sinh(\pi x)}{\sinh(\frac{\pi \gamma^2 x}{2}) - x\sin(\frac{\pi \gamma^2}{2}) } - C_5 \cdot \frac{\sinh(\pi ( 1 - \frac{\gamma^2}{4}) x)}{\sinh(\frac{\pi \gamma^2 x}{4})} \Big)\ell_1^{ix - 1}.
    \end{equation}
\end{lemma}

\begin{proof}
    In this proof, the constant $C$ may change from line to line. Using~\eqref{eq:conformal-welding-backbone-2} from Proposition~\ref{prop:weld-backbone} and the fact that $\QD_{1,0} = C \LF_\H^{(\gamma,i)}$ from \cite[Theorem 3.4]{ARS21}, we obtain 
    \begin{equation}\label{eq:lem6.12-1}
    \LF_\H^{(\gamma,i)} \otimes \mathsf{n}_{(T^*)^c} = \int_0^\infty {\rm Weld}(\mathcal{QA}^*(\ell), \LF_\H^{(\gamma,i)}(\ell)) \ell d\ell
    \end{equation}
    where $\mathsf{n}_{(T^*)^c}$ is the law of $\ep{\mathcal{L}^*}$ restricted to the event $(T^*)^c$. Using the fact that $|\LF_\H^{(\gamma,i)}(\ell)| = C_1 \ell^{-\frac{3}{2}}$, we see from Definition~\ref{def:decorate-BA-backbone} that ${\rm BA}_{BB} = C \cdot \mathcal{QA}^*$ for some $C>0$. For a loop $\eta$ sampled from $\mathsf{n}_{(T^*)^c}$, let $D_\eta$ be the connected component of $\mathbb{H} \setminus \eta$ containing $i$, and let $\psi_\eta: \mathbb{H} \rightarrow D_\eta $ be the conformal map with $\psi_\eta(i) = i$. Define the measure $\mathsf{n}_{(T^*)^c}^\alpha$ by $\frac{d \mathsf{n}_{(T^*)^c}^\alpha(\eta)}{d \mathsf{n}_{(T^*)^c}(\eta)} = |\psi_\eta'(i)|^{2 \Delta_\alpha - 2}$. Similar to the proof in Lemma~\ref{lem:distribution-QA}, we can derive from~\eqref{eq:lem6.12-1} that
    $$
    \LF_\H^{(\alpha,i)} \otimes \mathsf{n}_{(T^*)^c}^\alpha = \int_0^\infty {\rm Weld}(\mathcal{QA}^*(\ell), \LF_\H^{(\alpha,i)}(\ell)) \ell d\ell.
    $$
    Applying Lemma~\ref{lem:LFH-onepoint}, we get that for any $\ell_1>0$ and $\alpha > \frac{\gamma}{2}$:
    \begin{equation}\label{eq:lem6.12-2}
    \int_0^\infty |\mathcal{QA}^*(\ell_1,\ell_2)| \ell_2^{\frac{2(\alpha - Q)}{\gamma}} d\ell_2 = \mathsf{n}_{(T^*)^c}[|\psi_\eta'(i)|^{2 \Delta_\alpha - 2}].
    \end{equation}
    Similarly, we can use~\eqref{eq:conformal-welding-backbone-2} of Proposition~\ref{prop:weld-backbone} to derive that
    \begin{equation}\label{eq:lem6.12-3}
    \int_0^\infty |\wt{\QA}(\gamma^2-2; \ell_1,\ell_2)| \ell_2^{\frac{2(\alpha - Q)}{\gamma}} d\ell_2 = \mathsf{n}_{T^*}[|\psi_\eta'(i)|^{2 \Delta_\alpha - 2}],
    \end{equation}
    where $\psi_\eta$ is defined similarly as before. Moreover, by~\eqref{eq:CR-backbone} from Proposition~\ref{prop:formula-CR}, we have
    \begin{equation}\label{eq:lem6.12-4}
    \mathsf{n}_{(T^*)^c}[|\psi_\eta'(i)|^{2 \Delta_\alpha - 2}] + \mathsf{n}_{T^*}[|\psi_\eta'(i)|^{2 \Delta_\alpha - 2}] = \frac{4 \sin(\frac{\pi \gamma^2}{2})}{\gamma^2 \sin(\frac{4 \pi}{\gamma^2})} \times \frac{\sin(\pi \frac{2(\alpha - Q)}{\gamma})}{\sin(\frac{\pi \gamma^2}{2}\cdot \frac{2(\alpha - Q)}{\gamma}) - \frac{2(\alpha - Q)}{\gamma} \sin(\frac{\pi \gamma^2}{2})}.
    \end{equation}
    Combining~\eqref{eq:lem6.12-2}--\eqref{eq:lem6.12-4}, Lemma~\ref{lem:QA(W)-length}, and using the fact that ${\rm BA}_{BB} = C \cdot \mathcal{QA}^*$, we obtain that for $\frac{\gamma}{2} < \alpha < Q$,
    \begin{equation}\label{eq:lem6.12-5}
    \int_0^\infty |{\rm BA}_{BB}(\ell_1,\ell_2)| \ell_2^{\frac{2(\alpha - Q)}{\gamma}} d\ell_2 = \Big( C_4 \cdot \frac{\sin(\pi \frac{2(\alpha - Q)}{\gamma})}{\sin(\frac{\pi \gamma^2}{2}\cdot \frac{2(\alpha - Q)}{\gamma}) - \frac{2(\alpha - Q)}{\gamma} \sin(\frac{\pi \gamma^2}{2})} - C_5 \cdot \frac{\sin(\pi ( 1 - \frac{\gamma^2}{4}) \cdot \frac{2(\alpha - Q)}{\gamma} )}{\sin(\frac{\pi \gamma^2 }{4}\cdot \frac{2(\alpha - Q)}{\gamma})} \Big) \ell_1^{\frac{2(\alpha - Q)}{\gamma} - 1}
    \end{equation}
    for some constants $C_4,C_5>0$. As a function of $z$, $ \int_0^\infty |{\rm BA}_{BB}(\ell_1,\ell_2)| \ell_2^z d\ell_2$ is analytic for $\{z \in \mathbb{C}: -\frac{4}{\gamma^2} < {\rm Re} z < 0\}$. Moreover, by sending $\alpha$ to $Q$, we can deduce from~\eqref{eq:lem6.12-5} that $\int_0^\infty |{\rm BA}_{BB}(\ell_1,\ell_2)| d\ell_2 < \infty$. Applying analytic continuation to both sides of~\eqref{eq:lem6.12-5} and sending $z$ to $ix$ yields the lemma. Note that the function converges as $z$ tends to $ix$, as $\int_0^\infty |{\rm BA}_{BB}(\ell_1,\ell_2)| d\ell_2 < \infty$.
\end{proof}

Now we prove Theorem~\ref{thm:backbone-crossing}, assuming Lemma~\ref{lem:laplace-calculation-1}, whose proof is postponed to Section~\ref{sec:calculation}.

\begin{lemma}\label{lem:laplace-calculation-1}
    For $\tau>0$, let $y(\tau)$ be the right-hand side of~\eqref{eq:thm-backbone-2} or~\eqref{eq:thm-backbone-1}   in Theorem~\ref{thm:backbone-crossing}. Then
    \begin{equation}\label{eq:lem7.3-1}
    \int_0^\infty e^{-\frac{2\pi x^2 \tau}{3}} y(\tau) \eta(2 i \tau)d\tau = \frac{\sqrt{3}}{x} \bigg( \frac{\sinh(\frac{2}{3}\pi x) \sinh(\pi x)}{\sinh(\frac{4}{3} \pi x) + \frac{\sqrt{3}}{2} x} - \sinh(\frac{1}{3} \pi x) \bigg) \quad \mbox{for all }x \in \mathbb{R}.
    \end{equation}
    Moreover, $\lim_{\tau \to 0} y(\tau) = 1$.
\end{lemma}

\begin{proof}[Proof of Theorem~\ref{thm:backbone-crossing} given Lemma~\ref{lem:laplace-calculation-1}]
    Combining Proposition~\ref{prop:kpz-annulus}, Lemmas~\ref{lem:decorate-BA-backbone} and \ref{lem:length-distribution-backbone}, we obtain that for all $x \in \mathbb{R}$,
    \begin{equation}\label{eq:proof-1.2-1}
    \begin{aligned}
    \int_0^\infty e^{-\frac{\pi \gamma^2 x^2 \tau}{4}} \times C_3 \cdot p_{BB}(\tau) \eta(2i\tau) d \tau &= \frac{2 \sinh (\frac{\gamma^2}{4} \pi x)}{\pi \gamma x \Gamma(1+i x)} \iint_0^\infty \ell_1 e^{-\ell_1} \ell_2^{ix} \times |{\rm BA}_{BB}(\ell_1, \ell_2)| d \ell_1 d \ell_2 \\
    &=\frac{2 \sinh (\frac{\gamma^2}{4} \pi x)}{\pi \gamma x} \Big(C_4 \cdot \frac{\sinh(\pi x)}{\sinh(\frac{\pi \gamma^2 x}{2}) - x\sin(\frac{\pi \gamma^2}{2}) } - C_5 \cdot \frac{\sinh(\pi ( 1 - \frac{\gamma^2}{4}) x)}{\sinh(\frac{\pi \gamma^2 x}{4})} \Big).
    \end{aligned}
    \end{equation}
    Since $0 \leq p_{BB}(\tau) \leq 1$ and $\eta(2i\tau)$ has an exponential tail as $\tau$ tends to $\infty$, the left-hand side of~\eqref{eq:proof-1.2-1} converges to 0 as $x$ tends to infinity, and so must the right-hand side. This implies that $C_4 = C_5$. Therefore, by Lemma~\ref{lem:laplace-calculation-1}, $p_{BB}(\tau)$ agrees with the right-hand side of~\eqref{eq:thm-backbone-2} and~\eqref{eq:thm-backbone-1} modulo a multiplicative constant,  which can be fixed using the fact $\lim_{\tau \rightarrow 0} p_{BB}(\tau) = 1$. This concludes the proof.
\end{proof}

\section{Integral identities: proof of Lemmas~\ref{lem:laplace-calculation} and \ref{lem:laplace-calculation-1}}\label{sec:calculation}

For $\tau>0$, let $q = e^{-\pi/\tau}$ and $\tilde q = e^{-2\pi\tau}$. Recall that the Dedekind eta function $\eta(z) = e^{\frac{i \pi z}{12}} \prod_{n=1}^\infty (1-e^{2 n i \pi z})$ for $z \in \mathbb{C}$ with ${\rm Im} z >0 $. It satisfies the relation $\eta(\frac{i}{\tau}) = \sqrt{\tau} \eta(i \tau)$ for any $\tau>0$. 

\begin{proof}[Proof of Lemma~\ref{lem:laplace-calculation}]
    It suffices to only prove the case for $x>0$. To prove Claim~\eqref{eq:eta-integral-1}, we first show that
    \begin{equation}\label{eq:lema.1-1}
        \Big(1-\sqrt{\frac{3}{2}} \frac{\eta(6i \tau) \eta(\frac{3}{2}i \tau)}{\eta(2 i \tau)\eta(3 i \tau)}\Big) \eta(2 i \tau) = \frac{1}{\sqrt{2 \tau}} q^{\frac{1}{24}} \sum_{k \in \mathbb{Z}} (-1)^k \Big( q^{\frac{3k^2-k}{2}} - q^{\frac{2k^2+k}{3}} \Big).
    \end{equation}
    Using the identity $\eta(2 i \tau) = \frac{1}{\sqrt{2\tau}} \eta(\frac{i}{2\tau})$ and the Euler identity $\prod_{n=1}^\infty(1-q^n) = \sum_{k \in \mathbb{Z}} (-1)^k q^{\frac{3k^2-k}{2}}$, we see that $\eta(2 i \tau) = \frac{1}{\sqrt{2 \tau}} q^{\frac{1}{24}} \sum_{k \in \mathbb{Z}} (-1)^k q^{\frac{3k^2-k}{2}}$. The Jacobi triple product identity states that for any $z,y \in \mathbb{C}$ with $|z|<1$ and $y \neq 0$,
\begin{equation}\label{eq:jacobi-triple}
\prod_{n=1}^\infty(1-z^{2n})(1+z^{2n-1}y)(1+z^{2n-1} y^{-1}) = \sum_{k=-\infty}^\infty z^{k^2} y^k.
\end{equation}
Taking $z = q^{\frac{2}{3}}$ and $y = - q^{\frac{1}{3}}$ in~\eqref{eq:jacobi-triple} and using $\eta(\frac{i}{\tau}) = \sqrt{\tau} \eta(i \tau)$ with $\frac{3}{2}\tau, 3 \tau, 6 \tau$ in place of $\tau$, we obtain
$$
\sum_{k=-\infty}^\infty (-1)^k q^{\frac{2k^2+k}{3}} = \prod_{n=1}^\infty(1-q^\frac{4n}{3})(1-q^\frac{4n-1}{3})(1-q^\frac{4n-3}{3}) = q^{-\frac{1}{24}} \frac{\eta(\frac{i}{6 \tau}) \eta(\frac{2i}{3\tau})}{\eta(\frac{i}{3\tau})} = \sqrt{3 \tau} q^{-\frac{1}{24}} \frac{\eta(6i \tau) \eta(\frac{3}{2}i \tau)}{\eta(3 i \tau)}.
$$
Combining this with the preceding identity for $\eta(2 i \tau)$ yields~\eqref{eq:lema.1-1}. Then, using \eqref{eq:lema.1-1} and the identity $\int_0^\infty \frac{1}{\sqrt{2 \tau}} e^{-\frac{2\pi x^2 \tau}{3}} q^a d\tau = \frac{\sqrt{3}}{2 x} e^{-\sqrt{\frac{8 a}{3}} \pi x}$, which holds for any $a>0$, we obtain
    $$
    \int_0^\infty e^{-\frac{2\pi x^2 \tau}{3}}\Big(1-\sqrt{\frac{3}{2}} \frac{\eta(6i \tau) \eta(\frac{3}{2}i \tau)}{\eta(2 i \tau)\eta(3 i \tau)}\Big) \eta(2 i \tau)d\tau = \frac{\sqrt{3}}{2x} \sum_{k \in \mathbb{Z}} (-1)^k \Big(e^{-\frac{|6j-1|}{3} \pi x} - e^{-\frac{|4j+1|}{3} \pi x} \Big).
    $$
After simplifying, we conclude Claim~\eqref{eq:eta-integral-1}.

Claim~\eqref{eq:eta-integral-2} can be proved similarly. Using $\eta(i \tau) = \frac{1}{\sqrt{\tau}} \eta(\frac{i}{\tau})$ with $\tau, 2\tau, 3\tau, 6\tau$ in place of $\tau$, we obtain
\begin{align*}
\frac{\sqrt{3} \eta(i\tau) \eta(6i\tau)^2}{\eta(3i\tau) \eta(2i\tau)} &= \frac{1}{\sqrt{2 \tau}} \frac{\eta(\frac{i}{\tau}) \eta(\frac{i}{6 \tau})^2}{\eta(\frac{i}{3\tau})\eta(\frac{i}{2\tau})} =  \frac{1}{\sqrt{2 \tau}} q^{\frac{1}{24}} \prod_{n=1}^\infty \frac{(1-q^{2n}) (1-q^{\frac{n}{3}})^2}{(1-q^{\frac{2n}{3}}) (1-q^n)}\\
&= \frac{1}{\sqrt{2 \tau}} q^{\frac{1}{24}} \prod_{n=1}^\infty (1 + e^{\frac{2 \pi i}{3}} q^{\frac{n}{3}}) (1 + e^{-\frac{2 \pi i}{3}} q^{\frac{n}{3}})(1 - q^{\frac{n}{3}}) \\
&= \frac{1}{\sqrt{2 \tau}} q^{\frac{1}{24}}\sum_{k \in \mathbb{Z}} \Big( q^{6k^2+k} + q^{6k^2+5k+1}-2q^{6k^2+3k+\frac{1}{3}}  \Big),
\end{align*}
where the last equation follows by taking $z = q^{\frac{1}{6}}$ and $y = e^{\frac{2 \pi i}{3}} q^{\frac{1}{6}}$ in~\eqref{eq:jacobi-triple} and dividing both sides by $(1 + e^{-\frac{2 \pi i}{3}})$. Then, we can obtain Claim~\eqref{eq:eta-integral-2} using the identity $\int_0^\infty \frac{1}{\sqrt{2 \tau}} e^{-\frac{2\pi x^2 \tau}{3}} q^a d\tau = \frac{\sqrt{3}}{2 |x|} e^{-\sqrt{\frac{8 a}{3}} \pi x}$.
\end{proof}

\begin{remark}\label{rmk:expansion}

Here we explain the derivation of~\eqref{eq:expand-tilde-q} and~\eqref{eq:expand-q} from Section~\ref{subsec:main-result}. The expressions in terms of $q$ have been established in the proof of Lemma~\ref{lem:laplace-calculation}, while the expressions in terms of $\tilde q$ follow from the Poisson summation formula. 

\end{remark}

Next, we prove Lemma~\ref{lem:laplace-calculation-1}.

\begin{proof}[Proof of Lemma~\ref{lem:laplace-calculation-1}]
    We prove Lemma~\ref{lem:laplace-calculation-1} by using the inverse Laplace transform of the right-hand side of~\eqref{eq:lem7.3-1} to solve $y(\tau)$ and show that it equals the right-hand side of~\eqref{eq:thm-backbone-2} or~\eqref{eq:thm-backbone-1}. The relation $\lim_{\tau \to 0} y(\tau) = 1$ can be seen from the right-hand side of~\eqref{eq:thm-backbone-1}. First, we derive the expression in terms of $q$. For sufficiently large $x$, the right-hand side of~\eqref{eq:lem7.3-1} is
    \begin{equation}\label{eq:lem7.3-1-1}
    \begin{aligned}
    &=\frac{\sqrt{3}}{x}  \bigg( \sinh(\frac{2}{3}\pi x) \sinh(\pi x) \sum_{n=0}^\infty (-\frac{\sqrt{3}}{2}x )^n \sinh(\frac{4}{3} \pi x)^{-n-1} - \sinh(\frac{1}{3} \pi x) \bigg) \\
    &=\sqrt{3} \bigg( \sum_{m,n \geq 0} \frac{1}{2} (-\sqrt{3})^n \binom{n+m}{m} x^{n-1} (e^{\frac{5}{3} \pi x} - e^{\frac{1}{3} \pi x} - e^{-\frac{1}{3} \pi x} + e^{-\frac{5}{3} \pi x}) e^{-\frac{4}{3} (n+2m+1) \pi x } - \frac{1}{x}\sinh(\frac{1}{3} \pi x) \bigg).
    \end{aligned}
    \end{equation}
    In the second equality, we used the relation $\sinh(\frac{4}{3} \pi x)^{-n-1} = 2^{n+1} e^{-\frac{4}{3} (n+1) \pi x} \sum_{m=0}^\infty \binom{n+m}{m} e^{-\frac{8}{3} m \pi x }$. 
    
    For any integer $n \geq 0$, the following identity holds:
    $$
    x^{n-1} e^{-2x} = \frac{1}{\sqrt{\pi}} \int_0^\infty e^{-x^2 \tau - \frac{1}{\tau}} \sum_{k=0}^\infty (-\frac{1}{2})^k \binom{n}{2k} (2k-1)!! \times \tau^{k-n-\frac{1}{2}} d\tau.
    $$
    For any $u>0$, replacing $x$ with $\frac{u \pi}{2} x$ and $\tau$ with $\frac{8}{3 u^2 \pi} \tau$ in the above identity, we get
    \begin{equation}\label{eq:lem7.3-1-4}
    \begin{aligned}
        &x^{n-1} e^{- u \pi x} = \int_0^\infty e^{-\frac{2\pi x^2 \tau}{3}} g(\tau; n, u) d \tau \\
        \mbox{with}\quad &g(\tau; n,u)=q^{\frac{3}{8} u^2} \sum_{k=0}^\infty (-\frac{1}{2 \pi})^k \binom{n}{2k} (2k-1)!! \times 2^{3k-2n+\frac{1}{2}} 3^{n-k-\frac{1}{2}} u^{n-2k} \tau^{k-n-\frac{1}{2}}.
    \end{aligned}
    \end{equation}
    Applying the above inverse Laplace transform to~\eqref{eq:lem7.3-1-1} term by term, we formally have $y(\tau) \eta(2 i \tau)$ is
    \begin{equation}\label{eq:lem7.3-1-2}
    \begin{aligned}
    &= \sqrt{3} \sum_{m,n \geq 0} \frac{1}{2} (-\sqrt{3})^n \binom{n+m}{m} \sum_{k=0}^\infty (-\frac{1}{2 \pi})^k \binom{n}{2k} (2k-1)!! \times 2^{3k-2n+\frac{1}{2}} 3^{n-k-\frac{1}{2}} \tau^{k-n-\frac{1}{2}}\\
    &\qquad \quad \times\bigg((\frac{4}{3}(n+2m+1) - \frac{5}{3})^{n-2k}q^{\frac{3}{8}(\frac{4}{3}(n+2m+1) - \frac{5}{3})^2} - (\frac{4}{3}(n+2m+1) - \frac{1}{3})^{n-2k}q^{\frac{3}{8}(\frac{4}{3}(n+2m+1) - \frac{1}{3})^2} 
    \\
    &\qquad \qquad - (\frac{4}{3}(n+2m+1)+ \frac{1}{3})^{n-2k}q^{\frac{3}{8}(\frac{4}{3}(n+2m+1) + \frac{1}{3})^2} + (\frac{4}{3}(n+2m+1)+ \frac{5}{3})^{n-2k}q^{\frac{3}{8}(\frac{4}{3}(n+2m+1) + \frac{5}{3})^2} \bigg).
    \end{aligned}
    \end{equation}
    We justify this equation by showing that the Laplace transform of the right-hand side matches the right-hand side of~\eqref{eq:lem7.3-1} for sufficiently large $x$. Then, by the uniqueness of the Laplace transform, we see that~\eqref{eq:lem7.3-1-2} holds for any $\tau$. We first verify that the infinite summation on the right-hand side of~\eqref{eq:lem7.3-1-2} is absolutely convergent and the absolute summation can be bounded by $Ce^{C\tau}$ for some large $C>0$. When $\tau \leq 1$, the power of $q$ in~\eqref{eq:lem7.3-1-2} dominates other terms, so the absolute summation is bounded by a universal constant. When $\tau>1$, there exists $C>0$ such that the absolute summation for~\eqref{eq:lem7.3-1-2} is upper-bounded by
    \begin{equation}\label{eq:lem7.3-1-3}
    C \times \sum_{\substack{m,n \geq 0\\ 0 \leq k \leq \lfloor n/2 \rfloor}} C^{n+m} k^k \tau^{k-n} (n+2m)^{n-2k} e^{-\frac{(n+2m)^2}{3 \tau}} = C \times \sum_{\substack{m,n \geq 0\\ 0 \leq k \leq \lfloor n/2 \rfloor}} C^{n+m}(\frac{n+2m}{\tau})^n (\frac{k \tau}{(n+2m)^2})^k e^{-\frac{(n+2m)^2}{3 \tau}}.
    \end{equation}
    Let $A > 100$ be a constant to be chosen. When $n + 2m \geq A \tau$, we have $k \tau \leq (n+2m)^2$. Hence, the summand term on the right-hand side of~\eqref{eq:lem7.3-1-3} is upper-bounded by $C^{n+m}\exp((n+2m)( \log(\frac{n+2m}{\tau}) - \frac{n+2m}{3\tau}))$. We can fix $A$ sufficiently large such that this term is upper-bounded by $e^{-n-2m}$, and thus the summation over $n + 2m \geq A \tau$ is upper-bounded by some universal constant. Next consider the range where $n+2m < A \tau$. Using $k \leq n + 2m$, the summand term on the right-hand side of~\eqref{eq:lem7.3-1-3} can be bounded by $C^{n+m} \cdot A^n \cdot  \max\{1, \frac{\tau}{n+2m} \}^{n+2m} \leq C'^\tau$ for some $C'>0$. Therefore, the right-hand side of~\eqref{eq:lem7.3-1-3} can be bounded by $C e^{C\tau}$ for $\tau>1$, where we enlarge the value of $C$.
    
    Therefore, we can integrate the right-hand side of~\eqref{eq:lem7.3-1-2} against $e^{-\frac{2\pi x^2 \tau}{3}} dx$ and interchange the integral for all sufficiently large $x$. By~\eqref{eq:lem7.3-1-1} and~\eqref{eq:lem7.3-1-4}, we see that~\eqref{eq:lem7.3-1} holds with $y(\tau) \eta(2 i \tau)$ replaced by the right-hand side of~\eqref{eq:lem7.3-1-2}. Using the uniqueness of the Laplace transform, we obtain that~\eqref{eq:lem7.3-1-2} holds. Recall that $\eta(2 i \tau) = \frac{1}{\sqrt{2 \tau}} \eta(\frac{i}{2 \tau}) = \frac{1}{\sqrt{2 \tau}} q^{\frac{1}{24}} \prod_{n=1}^\infty(1-q^n)$. After simplifying~\eqref{eq:lem7.3-1-2}, we obtain the right-hand side of~\eqref{eq:thm-backbone-1}.

    Next, we derive the expression in terms of $\tilde q$, namely the right-hand side of~\eqref{eq:thm-backbone-2}. Replacing $\frac{2\pi x^2}{3}$ in~\eqref{eq:lem7.3-1} with $t$, it is equivalent to solving the Laplace transform $\int_0^\infty e^{-t\tau} y(\tau) \eta(2 i \tau) d\tau = g(t)$ for any $t>0$, where $g(t)$ is given by
    \begin{equation}\label{eq:lem7.3-2-1}
    g(t) = \frac{\sqrt{3}}{x} \bigg( \frac{\sinh(\frac{2}{3}\pi x) \sinh(\pi x)}{\sinh(\frac{4}{3} \pi x) + \frac{\sqrt{3}}{2} x} - \sinh(\frac{1}{3} \pi x) \bigg) \quad \mbox{with} \quad x = \sqrt{\frac{3 t}{2 \pi}}.
    \end{equation}
    Define $\sqrt{z}$ via arguments in $(-\frac{\pi}{2}, \frac{\pi}{2}]$ for $z \in \mathbb{C}$. Then, we can verify that $g(t)$ is a meromorphic function in $\mathbb{C}$. In particular, it is analytic on $(-\infty,0)$ except for one pole smaller than $-\frac{2}{3} \pi$, determined by the solution to $\sinh(4 \pi \sqrt{\frac{t}{6\pi}}) + \frac{3}{2}\sqrt{\frac{t}{2\pi}} = 0$. 
    %Since $|y(\tau)| \leq 1$ and $|\eta(2 i \tau)| \leq e^{-\frac{\pi}{6} \tau}$, the Laplace transform $\int_0^\infty e^{-t\tau} p_{BB}(\tau) \eta(2 i \tau) d\tau$ is absolutely convergent for all $t \in \mathbb{C}$ with ${\rm Re} t \geq -\frac{\pi}{6}$. 
    By the inverse Laplace transform formula, we have
    \begin{equation}\label{eq:lem7.3-2-2}
    y(\tau) \eta(2 i \tau) = \frac{1}{2 \pi i} \int_{-i \infty}^{i \infty} e^{\tau t} g(t) dt \quad \mbox{for any $\tau>0$.}
    \end{equation}
    The right-hand side of~\eqref{eq:lem7.3-2-2} is absolutely convergent because from~\eqref{eq:lem7.3-2-1}, we can deduce the existence of $C>0$ such that $|g(t)| \leq C \exp(-\sqrt{|t|}/C)$ for all $t \in i\mathbb{R}$. Fix $\tau>0$. Next, we will deform the integral contour in~\eqref{eq:lem7.3-2-2} and apply Cauchy's residue theorem to show that
    \begin{equation}\label{eq:lem7.3-2-3}
    y(\tau) \eta(2 i \tau) = \sum_{s \in \mathcal{S}'} {\rm Res}(e^{\tau t} g(t), s)
    \end{equation}
    where $\mathcal{S}'$ is the set of poles of $g(t)$ in $\mathbb{C}$, and ${\rm Res}(e^{\tau t} g(t), s)$ is the residue of $e^{\tau t} g(t)$ at $t=s$. Specifically, $\mathcal{S}'$ consists of all the complex solutions to $\sinh(4 \pi \sqrt{\frac{t}{6\pi}}) + \frac{3}{2}\sqrt{\frac{t}{2\pi}} = 0$ except $0$ and $-\frac{2 \pi}{3}$. We refer to Lemma~\ref{lem:absolute-convergent} for the method to solve this equation and the proof that the right-hand side of~\eqref{eq:lem7.3-2-3} is absolutely convergent. The right-hand side of~\eqref{eq:thm-backbone-2} readily follows from~\eqref{eq:lem7.3-2-3} (note that $\mathcal{S}$ in Theorem~\ref{thm:backbone-crossing} is equivalent to $-\frac{1}{2 \pi} \mathcal{S}'$) and the identity $\eta(2 i \tau) = \tilde q^{\frac{1}{12}} \prod_{n=1}^\infty (1-\tilde q^n)$. 
    
    Now we prove~\eqref{eq:lem7.3-2-3}. We claim that for any $\epsilon>0$, there exists $R > \frac{1}{\epsilon}$ such that 
    \begin{equation}\label{eq:lem7.3-2-4}
    \sup_{\theta \in [\frac{\pi}{2}, \frac{3\pi}{2}]} |g(Re^{i\theta})| <\epsilon.
    \end{equation}
    We can first show that there exists a sequence $\{R_n\}_{n\geq 1}$ such that $\lim_{n \rightarrow \infty} R_n = \infty$ and $|\sinh(\frac{4}{3} \pi x) + \frac{\sqrt{3}}{2} x| \geq \frac{1}{100} \max \{|\sinh(\frac{4}{3} \pi x)|, |x| \}$ for all $x \in \{R_n e^{i \theta}: \theta \in [\frac{\pi}{4}, \frac{\pi}{2}] \cup (-\frac{\pi}{2}, -\frac{\pi}{4}]$ \}. Taking $R = \frac{2 \pi}{3} R_n^2$ for sufficiently large $n$ yields~\eqref{eq:lem7.3-2-4}. Using~\eqref{eq:lem7.3-2-4} and the fact that $|g(t)| \leq C \exp(-\sqrt{|t|}/C)$ for all $t \in i \mathbb{R}$, we can deform the integral contour of~\eqref{eq:lem7.3-2-2} to $\partial D$, where $D = \{ re^{i\theta}: 0 \leq r \leq R,  \frac{\pi}{2} \leq \theta \leq \frac{3 \pi}{2} \}$, with an error of at most $o_\epsilon(1)$. By Cauchy's residue theorem, we get $y(\tau) \eta(2 i \tau) = \sum_{s \in \mathcal{S}' \cap D} {\rm Res}(e^{\tau t} g(t), s) + o_\epsilon(1)$. Note that $\mathcal{S}' \subset \{z \in \mathbb{C}: {\rm Re} z < 0\}$ and the sum $\sum_{s \in \mathcal{S}'} {\rm Res}(e^{\tau t} g(t), s)$ is absolutely convergent for any fixed $\tau>0$. Hence, taking $\epsilon$ to 0 yields~\eqref{eq:lem7.3-2-3}.
\end{proof}

Finally, we solve the set $\mathcal{S}'$ (and $\mathcal{S}$) and show that the right-hand side of~\eqref{eq:lem7.3-2-3} is absolutely convergent.

\begin{lemma}\label{lem:absolute-convergent}
    The right-hand side of~\eqref{eq:lem7.3-2-3} is absolutely convergent.
\end{lemma}

\begin{proof}
    We first solve all the complex solutions to the equation $\sinh(4 \pi \sqrt{\frac{t}{6\pi}}) + \frac{3}{2} \sqrt{\frac{t}{2\pi}} = 0$. The set $\mathcal{S}'$ is given by all these solutions except $0$ and $-\frac{2 \pi}{3}$, and the set $\mathcal{S}$ in Theorem~\ref{thm:backbone-crossing} is given by $-\frac{1}{2 \pi} \mathcal{S}'$. Let $4 \sqrt{\frac{t}{6\pi}} = a + bi$ with $a \in \mathbb{R}_{\geq 0}$ and $b \in \mathbb{R}$. Then we have $\sinh(\pi (a + bi)) + \frac{3 \sqrt{3}}{8}(a + bi) = 0$. Separating the real and imaginary parts yields
\begin{equation}\label{eq:rmk-solve}
\begin{aligned}
\sinh(\pi a)\cos(\pi b) + \frac{3\sqrt{3}}{8}a = 0 \qquad \mbox{and} \qquad \cosh(\pi a)\sin(\pi b) + \frac{3\sqrt{3}}{8}b = 0.
\end{aligned}
\end{equation}
If $(a,b)$ is a solution, then $(a, -b)$ is also a solution. Hence, it suffices to solve~\eqref{eq:rmk-solve} for $a\geq 0$ and $b \geq 0$. If $a = 0$, then $b$ solves the equation $\sin( \pi b) + \frac{3 \sqrt{3}}{8} b = 0$, which has three non-negative solutions: 0, $\frac{4}{3}$, and one root in the interval $(\frac{4}{3}, \frac{5}{3})$. The first two solutions correspond to $x = 0, -\frac{2}{3}\pi$, and the last solution corresponds to the only real solution in $\mathcal{S}'$ smaller than $-\frac{2}{3}\pi$. If $a > 0$, then by the first equation in~\eqref{eq:rmk-solve} and $\sin( \pi b) < 0$, we have $\pi b = 2n \pi - \arccos(-\frac{3 \sqrt{3} a}{8 \sinh( \pi a)})$ for some integer $n \geq 1$. Then the second equation in~\eqref{eq:rmk-solve} becomes
\begin{equation}\label{eq:rmk-solve-1}
\frac{8 \pi }{3\sqrt{3}} \cosh(\pi a) \sqrt{1 - \frac{27 a^2 }{64 \sinh(\pi a)^2}} +  \arccos \Big( -\frac{3 \sqrt{3} a}{8 \sinh( \pi a)} \Big) = 2n \pi.
\end{equation}
The left-hand side of~\eqref{eq:rmk-solve-1} is increasing in $a$ and is between $2 \pi$ and $4 \pi$  at $a = 0$. Hence, this equation has exactly one solution for each $n \geq 2$. Solving these equations gives the set $\mathcal{S}'$ and also $\mathcal{S}$. 

Finally, we show that the sum on the right-hand side of~\eqref{eq:lem7.3-2-3} is absolutely convergent. Let $(a_n,b_n)$ be the solution to~\eqref{eq:rmk-solve-1} with respect to the integer $n \geq 2$. Then as $n \rightarrow \infty$, we have $a_n \sim \log n$ and $b_n \sim n$. By simple calculations, the summand term on the right-hand side of~\eqref{eq:lem7.3-2-3} corresponding to $t \in \mathcal{S}'$ is equal to
$$
\frac{\sqrt{3} \sinh(\frac{2 \pi}{3} \sqrt{\frac{3 t}{2 \pi}}) \sinh( \pi \sqrt{\frac{3 t}{2 \pi}})}{\cosh(\frac{4\pi }{3} \sqrt{\frac{3 t}{2 \pi}}) + \frac{3\sqrt{3}}{8 \pi}} \times e^{\tau t}.
$$
For integers $n \geq 2$, the real part of $t = \frac{3 \pi}{8}(a_n \pm b_n i)^2$ is proportional to $-n^2$. Moreover, for some universal constants $C,C'>0$,
$$
\Big{|}\frac{\sqrt{3} \sinh(\frac{2 \pi}{3} \sqrt{\frac{3 t}{2 \pi}}) \sinh( \pi \sqrt{\frac{3 t}{2 \pi}})}{\cosh(\frac{4\pi }{3} \sqrt{\frac{3 t}{2 \pi}}) + \frac{3\sqrt{3}}{8 \pi}} \Big{|} \leq C e^{C a_n} \leq n^{C'}.
$$
Therefore, the sum on the right-hand side of~\eqref{eq:lem7.3-2-3} is absolutely convergent.
\end{proof}

\bibliographystyle{alpha}
\bibliography{theta}

\end{document}